\documentclass[12pt]{article}
\def\mydate{8 April  2018}
\usepackage{amsmath, amssymb, amsfonts, amsthm}
\usepackage[dvips]{color}
\usepackage[dvipsnames]{xcolor}
\usepackage{hyperref}

\newtheorem{proposition}{Proposition}[section]
\newtheorem{lemma}[proposition]{Lemma}
\newtheorem{corollary}[proposition]{Corollary}
\newtheorem{theorem}[proposition]{Theorem}

\newtheorem{conjecture}[proposition]{Conjecture}

\theoremstyle{definition}
\newtheorem{definition}[proposition]{Definition}

\newtheorem{example}[proposition]{Example}
\newtheorem{problem}[proposition]{Problem}
\newtheorem{question}[proposition]{Question}

\newcommand{\mcal}{\mathcal}

\newcommand{\C}{\mcal{C}}
\newcommand{\D}{\mcal{D}}

\newcommand{\F}{\mcal{F}}
\newcommand{\G}{\mcal{G}}

\newcommand{\K}{\mcal{K}}
\renewcommand{\L}{\mcal{L}}

\newcommand{\R}{\mcal{R}}

\def\yy#1{{\color{blue}#1}}
\def\zz#1{{\color{blue}#1}}
\def\aa#1{{\color{blue}#1}}
\def\xx#1{{\color{blue}#1}}
\def\cc#1{{\color{Magenta}#1}}
\def\REM#1{\footnote{#1}}
\def\NOTE#1{{\bf NOTE:} #1}
\let\ppar=\par
\def\yy#1{{#1}}
\def\zz#1{{#1}}
\def\aa#1{{#1}}
\def\xx#1{{#1}}
\def\qq#1{{#1}}
\def\cc#1{{#1}}
\def\REM#1{}
\def\NOTE#1{}

\begin{document}
\font\smallrm=cmr8




\centerline{{\bf  HYPERBOLIC FAMILIES AND COLORING GRAPHS ON SURFACES}%
}
\vskip .25in
\centerline{{\bf Luke Postle}%
\footnote{Partially
supported by NSERC under Discovery Grant No. 2014-06162, the Ontario Early 
Researcher Awards program and the Canada Research Chairs program.}}
\smallskip
\centerline{Department of Combinatorics and Optimization}
\centerline{University of Waterloo}
\centerline{Waterloo, ON}
\centerline{Canada N2L 3G1}
\centerline{{\tt lpostle@uwaterloo.ca}}
\bigskip
\centerline{and}
\bigskip
\centerline{{\bf Robin Thomas}%
\footnote{Partially supported by NSF under Grant No.~DMS-1202640.}}
\smallskip
\centerline{School of Mathematics}
\centerline{Georgia Institute of Technology}
\centerline{Atlanta, Georgia  30332-0160, USA}
\centerline{{\tt thomas@math.gatech.edu}}

\bigskip
\centerline{\bf ABSTRACT}
\bigskip

\noindent
We develop a theory of linear isoperimetric inequalities for graphs on surfaces and apply it to coloring problems,
as follows.
Let $G$ be a graph embedded in a fixed surface $\Sigma$ of genus $g$ and let $L=(L(v):v\in V(G))$ be
a collection of lists such that either each list has size at least five,
or each list  has size at least four and $G$ is triangle-free, or
each list has size at least three and $G$  has no cycle of length four or less.
 An {\em $L$-coloring} of $G$ is a mapping
$\phi$ with domain $V(G)$ such that $\phi(v)\in L(v)$ for every $v\in V(G)$ and
$\phi(v)\ne\phi(u)$ for every pair of adjacent vertices $u,v\in V(G)$.
We prove
\begin{itemize}
\item if every non-null-homotopic cycle in $G$ has length $\Omega(\log g)$, then $G$ has an $L$-coloring,
\item if $G$ does not have an $L$-coloring, but every proper subgraph does (``$L$-critical graph"), 
then $|V(G)|=O(g)$,
\item  if every non-null-homotopic cycle in $G$ has length $\Omega(g)$, and a set $X\subseteq V(G)$
of vertices that are pairwise at distance $\Omega(1)$ is precolored from the corresponding lists,
then the precoloring extends to an $L$-coloring of $G$,
\item  if every non-null-homotopic cycle in $G$ has length $\Omega(g)$, and the graph $G$
is allowed to have crossings, but every two crossings are at distance $\Omega(1)$, then $G$ has
an $L$-coloring,
\item if $G$ has at least one  $L$-coloring, then it has at least $2^{\Omega(|V(G)|)}$ distinct
$L$-colorings.
\end{itemize}

We show that the above assertions are consequences of certain isoperimetric inequalities satisfied
by $L$-critical graphs, and we
 study the structure of families of embedded graphs that satisfy those inequalities.
It follows that the above assertions
hold for other coloring problems, as long as the corresponding critical graphs
 satisfy the same inequalities.

\vfill \baselineskip 11pt \noindent 27 July 2013. Revised \mydate.
\vfil\eject
\baselineskip 18pt

\section{Introduction}
\label{sec:shortintro}

All graphs in this paper are finite, and have no loops or multiple
edges. Our terminology is standard, and may be found
in~\cite{Bollobas} or~\cite{Diestel}.
 In particular, {\em cycles} and {\em paths} have no repeated vertices, and
by a {\em coloring} of a graph $G$ we mean a proper vertex-coloring; that is,
a mapping $\phi$ with domain $V(G)$ such that $\phi(u)\ne\phi(v)$
whenever $u,v$ are adjacent vertices of $G$.
By a \emph{surface} we mean a (possibly disconnected) compact $2$-dimensional
manifold with possibly empty boundary.
The {\em boundary} of a surface $\Sigma$ will be denoted by bd$(\Sigma)$.
\aa{%
By the classification theorem every surface $\Sigma$ is obtained from the disjoint union
of finitely many spheres by adding $a$ handles, $b$ crosscaps and removing the interiors
of finitely many  pairwise disjoint closed disks. In that case the {\em Euler genus} of $\Sigma$ is $2a+b$.
}%

Motivated by graph coloring problems we study families of graphs that satisfy
the following isoperimetric inequalities.
By an {\em embedded graph} we mean a pair $(G,\Sigma)$, where $\Sigma$ is a surface without boundary
and $G$ is a graph embedded in $\Sigma$. We will usually suppress the surface and speak about an embedded graph $G$,
and when we will need to refer to the surface we will do so by saying that $G$ is embedded in a surface $\Sigma$.
Let $\F$ be a family of non-null embedded graphs.
We say that $\F$ is \emph{hyperbolic} if there exists a constant $c>0$
such that if $G\in\F$ is a graph that is embedded in a surface
$\Sigma$, then for every closed curve $\cc{\xi}:{\mathbb S}^1\to\Sigma$ that bounds
an open disk $\Delta$ and intersects $G$ only in vertices,
if $\Delta$ includes a vertex of $G$, then
the number of vertices of $G$ in $\Delta$ is at most
$c(|\{x\in {\mathbb S}^1\,:\, \cc{\xi}(x)\in V(G)\}|-1)$.
We say that $c$ is a \emph{Cheeger constant} for $\F$.


The hyperbolic families of interest to us arise from coloring problems, more specifically from a generalization
of the classical notion of coloring, introduced by Erd\"os, Rubin and Taylor~\cite{ErdRubTay} and known as 
{\em list coloring} or {\em choosability}.
We need a few definitions in order to introduce it.
Let $G$ be a graph and let $L=(L(v):v\in V(G))$ be
a collection of lists. 
If each set $L(v)$ is non-empty, then we say that $L$ is a {\em list assignment} for $G$.
If $k$ is an integer and $|L(v)|\ge k$ for every $v\in V(G)$, then we say that $L$ is
a {\em $k$-list assignment} for $G$.
An {\em $L$-coloring} of $G$ is a mapping
$\phi$ with domain $V(G)$ such that $\phi(v)\in L(v)$ for every $v\in V(G)$ and
$\phi(v)\ne\phi(u)$ for every pair of adjacent vertices $u,v\in V(G)$.
Thus if all the lists are the same, then this reduces to the notion of
coloring.
We say that a graph $G$ is \emph{$k$-choosable}, also called \emph{$k$-list-colorable}, if $G$ has an $L$-coloring for every $k$-list assignment~$L$. 
The \emph{list chromatic number} of $G$, denoted by $ch(G)$, also known as {\em choosability}, 
is the minimum $k$ such that $G$ is $k$-list-colorable.
A graph $G$ is {\em $L$-critical} if $G$ is not $L$-colorable, but every proper subgraph of $G$ is.
To capture the three most important classes of coloring problems to which our theory applies, let us
say that  a list assignment $L$ for a graph $G$  is of {\em type $345$} or a {\em type $345$ list assignment} if either
\begin{itemize}
\item $L$ is a $5$-list assignment, or
\item $L$ is a $4$-list assignment and $G$ is triangle-free, or
\item $L$ is a $3$-list assignment and $G$ has no cycle of length four or less.
\end{itemize}
We are now ready to to describe the three most  interesting hyperbolic families  of  graphs.


\begin{theorem}
\label{thm:precritishyperbolic}
The family of all embedded graphs that are $L$-critical for some type $345$ list assignment $L$ is  hyperbolic.
\end{theorem}

\noindent 
Theorem~\ref{thm:precritishyperbolic} is a special case of the more general Theorem~\ref{thm:Cisgood},
also stated as Theorem~\ref{thm:Cisgood-2} and proved in Section~\ref{sec:can}.
Additional examples of hyperbolic families  are given in Section~\ref{sec:hyperdef}.
We study hyperbolic families in Section~\ref{sec:structure}.
Corollary~\ref{cor:ewloccyl} implies the following.
A family $\F$ of embedded graphs is {\em closed under curve cutting} if for every embedded graph $(G,\Sigma)\in\F$
and every  simple closed curve $\cc{\xi}$ in $\Sigma$ whose image is disjoint from $G$, if $\Sigma'$ denotes
the surface  obtained from $\Sigma$ by cutting open along $\cc{\xi}$
\zz{%
and attaching disk(s) to the resulting curve(s)%
}, then the  embedded graph $(G,\Sigma')$
belongs to $\F$.


\begin{theorem}
\label{thm:ew}
For every  hyperbolic family $\F$ of embedded graphs
that is closed under curve cutting there exists a constant $k>0$ such that
every  graph $G\in\F$ \aa{embedded in a surface of Euler genus $g$}
 has a non-null-homotopic cycle of length at most
$k\log(g+1)$.
\end{theorem}

\noindent 
\xx{We prove Theorem~\ref{thm:ew} immediately  after Corollary~\ref{cor:ewloccyl}.}

Theorem~\ref{thm:ew} is already quite useful. For instance, by Theorem~\ref{thm:precritishyperbolic} it implies the first result stated in the abstract. 
Our main theorem about hyperbolic families is
 Theorem~\ref{thm:sleevedec1}, which describes their structure.
Roughly speaking, it says that every member of a hyperbolic family on a fixed surface can be obtained from a bounded size 
graph by attaching ``narrow cylinders".
The theorem suggests the following strengthening of hyperbolicity.
Let $G$ be a member of a hyperbolic family $\F$ embedded in a surface $\Sigma$, and let $\Lambda\subseteq\Sigma$ be
a cylinder such that the two boundary components of $\Lambda$ are cycles $C_1,C_2$ of $G$.
If for all $G$ and $\Lambda$
the number of vertices of $G$ in $\Lambda$ is bounded by some function of $|V(C_1)|+|V(C_2)|$, then we say
that $\F$ is {\em strongly hyperbolic}.
Theorem~\ref{thm:Cisgood} implies that the families from Theorem~\ref{thm:precritishyperbolic}  are, in fact, strongly  hyperbolic.

We study strongly hyperbolic families in Section~\ref{sec:strongly}.
The following is a special case of Theorem~\ref{thm:free5}.

\begin{theorem}
\label{thm:prefree5}
For every strongly hyperbolic family $\F$ of embedded graphs
that is closed under curve cutting there exists a  constant $\beta>0$ such that
every  graph $G\in\F$ embedded in a surface of  Euler genus $g$  has at most $\beta g$ vertices.
\end{theorem}

By applying Theorem~\ref{thm:prefree5}  to the families  from Theorem~\ref{thm:precritishyperbolic}
 we deduce the second result stated in the abstract.
To prove the next  result in the abstract we need to extend the notion of  hyperbolicity to rooted graphs.
A {\em rooted graph} is a pair $(G,X)$, where $G$ is a graph and $X\subseteq V(G)$.
We extend all the previous terminology to rooted graphs in the natural way, so we can speak about rooted embedded graphs. 
The definitions of hyperbolicity and strong hyperbolicity extend to rooted graphs with the proviso that the disk $\Delta$
and cylinder $\Lambda$ contain no member of $X$ in their interiors.
%
Now we can state a more general special case of Theorem~\ref{thm:free5}.
\yy{A proof is given immediately after Theorem~\ref{thm:free5}.}

\begin{theorem}
\label{thm:prefree5rooted}
For every strongly hyperbolic family $\F$ of rooted embedded graphs
that is closed under curve cutting there exist  constants $\beta,k,d>0$ such that
every rooted  graph $(G,X)\in\F$ embedded in a surface of Euler  genus $g$ has at most $\beta(g+|X|)$ vertices, and either
has a non-null-homotopic cycle of length at most $kg$, or some two vertices in $X$ are at a distance at most $d$, or $V(G)=X$.
\end{theorem}

By applying Theorem~\ref{thm:prefree5rooted} to the rooted versions of the families from Theorem~\ref{thm:precritishyperbolic} 
we deduce the following theorem,
which we state and formally prove in a slightly more general form as Theorem~\ref{Precolored}.

\begin{theorem}
\label{thm:preAlbertson}
There exist constants $k,d>0$ such that the following holds.
Let $(G,X)$ be a rooted graph embedded in a surface $\Sigma$, and let $L$ be a type $345$ list assignment for $G$  such that 
every non-null-homotopic cycle in $G $ has length at least $kg$ and every two vertices in $X$ are at a distance at least $d$.
Then \xx{every $L$-coloring of $X$ extends to an $L$-coloring of} $G$.
\end{theorem}

\noindent 
Theorem~\ref{thm:preAlbertson} proves the third result stated in the abstract.
The fourth result follows similarly \xx{(formally we prove it as Theorem~\ref{CrossingSurface})}, but it requires a generalization of the above two theorems to embedded graphs
with a set of distinguished cycles. That is why we will introduce the notion of embedded graphs with rings in Section~\ref{sec:mainresults},
where a ring is a cycle or a complete graph on at most two vertices.
The last result mentioned in the abstract needs the strong hyperbolicity of a different family of embedded graphs.
\xx{It follows from Theorem~\ref{ExpSurfacetheorem2}.}%
\REM{Deleted: We refer to Section~\ref{sec:can} for details.}




The paper is organized as follows. In Section~\ref{sec:history} we survey results on list coloring graphs on surfaces.
In Section~\ref{sec:mainresults} we formulate Theorem~\ref{thm:maincol},  our main coloring result,
and derive a number of consequences from it, assuming a generalization of Theorem~\ref{thm:precritishyperbolic},
\yy{stated as Theorem~\ref{thm:Cisgood},}
which we do not prove until Section~\ref{sec:can}.
Theorem~\ref{thm:maincol} is an easy consequence of Theorem~\ref{thm:free5}, our main result about strongly
hyperbolic families.
In the short Section~\ref{sec:expmany} we prove a lemma about exponentially many extensions of a coloring of a facial cycle in a planar graph.
In Section~\ref{sec:hyperdef} we present examples of hyperbolic families. We investigate their structure in depth in Section~\ref{sec:structure},
where the main result is  Theorem~\ref{thm:sleevedec1}, giving structural information about members of a hyperbolic family.
In Section~\ref{sec:strongly} we give examples of strongly hyperbolic families and investigate their structure.
The main theorem is Theorem~\ref{thm:free5}, but there is also the closely related Theorem~\ref{thm:free6}.
\xx{The latter eliminates an outcome of the former at the expense of a worse bound.}%
\REM{Deleted: The two differ in bounds they give; one gives a better bound on one parameter and worse on another parameter.}
In the final Section~\ref{sec:can} we complete the proof of the strong hyperbolicity of the three main families of interest
(a.k.a.\  a generalization of Theorem~\ref{thm:precritishyperbolic}), and prove
Theorem~\ref{thm:maincol}. We also formulate and prove the closely related Theorem~\ref{thm:maincanvar},
which has an identical proof, using Theorem~\ref{thm:free6} instead of Theorem~\ref{thm:free5}.


%

\section{Survey of Coloring Graphs on Surfaces}
\label{sec:history}


In this section we survey the main results about coloring graphs on surfaces in order
  to place our work in a historical context.
However,  familiarity with this section is not required in order to understand the rest of the paper.

%

\subsection{Coloring Graphs on Surfaces}
The topic of coloring graphs on surfaces dates back to 1852, when Francis Guthrie formulated the Four Color Conjecture,
which has since become the Four Color Theorem~\cite{4CT1, 4CT2, 4CTRSST}.
Even though this theorem, its history and ramifications are of substantial interest, we omit any further discussion of it,
because our primary interest lies in surfaces other than the sphere.
Thus we skip to the next development, which is the well-known Heawood formula from 1890, asserting that
a graph embedded in a surface $\Sigma$ of Euler genus $g\ge1$ can be colored with at most 
$H(\Sigma):=\lfloor(7 + \sqrt{24g + 1})/2\rfloor$ colors.  
This formula is actually a fairly easy concequence of Euler's formula. It turned out that the bound it gives is best possible
for every surface except the Klein bottle, but that was not established until the seminal work  of 
Ringel and Youngs~\cite{RingelYoungs} in the 1960s.
Franklin~\cite{Franklin} proved that every graph embeddable in the Klein bottle requires only six colors, whereas Heawood's bound predicts seven. Dirac~\cite{Dirac1} and Albertson and Hutchinson~\cite{AlbHutch1} improved Heawood's result by showing that every graph in $\Sigma$ is actually $(H(\Sigma)-1)$-colorable, unless it has a subgraph isomorphic to the complete graph on $H(\Sigma)$ vertices. 

Thus the maximum chromatic number for graphs embeddable in a surface has been determined for every surface.
However, a more modern approach to coloring graphs on surfaces, initiated by Thomassen in the 1990s, is based on the
observation that only very few graphs embedded  in a fixed surface $\Sigma$ have chromatic number close to $H(\Sigma)$;
in fact, as we are about to see, most have chromatic number at most five.
To make this assertion precise, we need a definition. Let $t\ge2$ be an integer.
We say that a graph $G$ is \emph{$t$-critical} if it is not $(t-1)$-colorable, but every proper subgraph of $G$ is $(t-1)$-colorable. Using Euler's formula, Dirac~\cite{Dirac2} proved that for every $t\ge 8$ and every surface $\Sigma$ there are only finitely many $t$-critical graphs that embed in $\Sigma$. By a result of Gallai~\cite{Gallai}, this can be extended to $t = 7$. We will see in a moment that this extends to $t=6$ by a deep result of Thomassen.
First however, let us mention a predecessor of this theorem, also due to Thomassen~\cite{Tho5colmaps}.

\begin{theorem}\label{ThomEdgeWidth}
For every surface $\Sigma$ there exists an integer $\rho$ such that if
$G$ is a graph embedded in $\Sigma$ such that every non-null-homotopic cycle in $G$ has length at least $\rho$, then $G$ is $5$-colorable.
\end{theorem}

 The following related result was obtained by Fisk and Mohar~\cite{FisMoh}.

\begin{theorem}
\label{thm:fiskmohar}
There exists an absolute constant $\gamma$ such that if
$G$ is a graph embedded in a surface $\Sigma$ of Euler genus $g$  such that every non-null-homotopic cycle in $G$ has length at least $\gamma(\log g+1)$, then 
\begin{itemize}
\item $G$ is $6$-colorable, and
\item if $G$ is triangle-free, then it is $4$-colorable, and
\item if $G$ has no cycles of length five or less, then it is $3$-colorable.
\end{itemize}
\end{theorem}

The other theorem of Thomassen~\cite{ThomCritical} reads as follows.

\begin{theorem}\label{ThomCrit}
For every surface $\Sigma$, there are (up to isomorphism) only finitely many $6$-critical graphs that embed in $\Sigma$.
\end{theorem}

It is clear that Theorem~\ref{ThomCrit} implies Theorem~\ref{ThomEdgeWidth}, but we state them separately for two reasons.
One is historical, and the other is that we will be able to improve the bound on $\gamma$  in  Theorem~\ref{ThomEdgeWidth}
 to an asymptotically best possible value, and that version is no longer a consequence of Theorem~\ref{ThomCrit}. 
Theorem~\ref{ThomCrit} immediately implies a polynomial-time algorithm for deciding whether a graph on a fixed surface is $5$-colorable.
By a result of Eppstein~\cite{Eppstein2} there is, in fact, a linear-time algorithm, as follows.

\begin{corollary}
\label{cor:surfcoloralgo}
For every surface $\Sigma$ there exists a linear-time algorithm that decides whether an input graph embedded in $\Sigma$
is $5$-colorable.
\end{corollary}

Corollary~\ref{cor:surfcoloralgo} guarantees the {\em existence} of an algorithm. To construct an algorithm we would need to find the list of all the $6$-critical
graphs that can be embedded in $\Sigma$.
Such lists are known only for the projective plane~\cite{AlbHutch1}, the torus~\cite{ThomTorus}, and the Klein bottle~\cite{KB, KB2}.
However, the proof of  Theorem~\ref{ThomCrit} can be converted to a finite-time algorithm that will output the desired list;
thus the algorithm, even though not explicitly known, can be constructed in finite time.

Theorem~\ref{ThomCrit} is best possible as it does not extend to $5$-critical graphs.
Indeed, Thomassen \cite{ThomCritical}, using a construction of Fisk~\cite{Fisk}, constructed infinitely many $5$-critical graphs that embed in 
any given surface other than the sphere. 
The next logical question then is whether Corollary~\ref{cor:surfcoloralgo} holds for $3$- or $4$-coloring.
For $3$-coloring the answer is no, unless P=NP, because $3$-colorability of planar graphs is one of the original NP-complete
problems of Garey and Johnson~\cite{GarJoh}.
For $4$-coloring there is a trivial algorithm when $\Sigma$ is the sphere by the Four Color Theorem, and for all other
surfaces it is an open problem. Given the difficulties surrounding the known proofs of the Four Color Theorem, 
the prospects for a resolution of that open problem in the near future are not very bright.




A classical theorem of Gr\"otzsch~\cite{Gro} asserts that every triangle-free planar graph is $3$-colorable.
Thomassen~\cite{Thom3Color, Thom3ListColor, ThoShortlist} found three reasonably short proofs, and extended the theorem to the projective
plane and the torus. However, for the extensions to hold one has to assume that the graph
has no cycles of length at most four.
Those results prompted Thomassen to formulate the following research program.
First, let us recall that the {\em girth} of a graph is the maximum integer $q$ (or infinity)
such that every cycle has length at least $q$.
Thomassen asked for which pairs of integers $k$ and $q$  is it the case that for every
surface $\Sigma$ there are only finitely many $(k+1)$-critical graphs of girth at least
$q$ in $\Sigma$.
If the answer is positive, then $k$-colorability of graphs of girth at least $q$ in
$\Sigma$ can be tested in linear time.
If the answer is negative, then there is still the question whether 
the $k$-colorability of graphs of girth at least $q$ in
$\Sigma$ can be tested in polynomial time.

These questions have now been resolved, except for the already mentioned case $k=4$
and $q=3$.
Let us briefly survey the results.
It is fairly easy to see that for every surface $\Sigma$ there are only finitely many 
$(k+1)$-critical graphs of girth at least $q$ in $\Sigma$ whenever either 
$q\ge6$, or $q\ge4$ and $k\ge4$, or $k\ge6$.
Theorem~\ref{ThomCrit} states that this also holds when $k=5$, and the following deep 
theorem, also due to Thomassen~\cite{ThoGirth5}, says that it also holds when $k=3$
and $q=5$.

\begin{theorem}
\label{thm:Colorcritgirth5}
For every surface $\Sigma$ there are only finitely many $4$-critical graphs of girth at least
five in $\Sigma$.
\end{theorem}

\noindent 
Dvo\v r\'ak, Kr\'al' and Thomas strengthened this theorem by giving an asymptotically
best possible bound on the size of the $4$-critical graphs, as follows.

\begin{theorem}\label{LinearGirth5}
There exists an absolute constant $c$ such that every $4$-critical graph
 of girth at least five embedded in a surface of Euler genus $g$ has at most $cg$ vertices. 
\end{theorem}



Thus the only case we have not yet discussed is $k=3$ and $q=4$; that is, $3$-coloring triangle-free
graphs on a fixed surface.
There are infinitely many $4$-critical triangle-free graphs on any surface other than the sphere,
but, nevertheless,  $3$-colorability of triangle-free graphs on any fixed surface can be tested 
in polynomial time~\cite{3ColLinearTime,DvoKraTho6}.
However, this algorithm requires different methods.
The theory we develop in this paper does not seem to apply to $3$-coloring triangle-free graphs. 


\subsection{List-Coloring Graphs on Surfaces}

Our main topic  of interest is list coloring, a generalization of ordinary vertex coloring,
introduced in the Introduction.
%
%
List coloring differs from regular coloring in several respects. 
One notable example is that the Four Color Theorem does not generalize to list-coloring. 
Indeed, Voigt~\cite{Voigt} constructed a planar graph that is not $4$-choosable.
On the other hand Thomassen~\cite{ThomPlanar} proved the following remarkable theorem with an outstandingly short and beautiful proof.

\begin{theorem}\label{PlanarChoosable}
Every planar graph is $5$-choosable.
\end{theorem}

\yy{
\noindent
The following theorem was also proven  by Thomassen, first in~\cite{Thom3ListColor} and later  he found a shorter
proof in~\cite{ThoShortlist}.
\begin{theorem}\label{Planar3Choosable}
Every planar graph of girth at least five is $3$-choosable.
\end{theorem}
\noindent
The corresponding theorem that every planar graph of girth at least four  is $4$-choosable is an easy consequence of Euler's formula;
a slightly stronger version of it is part of Theorem~\ref{thm:extend4cycle} below.
For later reference we need the following small generalization of the last three results mentioned. It follows easily from known  results.
 A graph has {\em crossing number at most one} if it can be drawn in the plane such that at most one pair of edges cross.
\begin{theorem}\label{CrossOneChoosable}
Let $G$ be a graph of  crossing number at most one and let $L$ be a type $345$ list assignment for $G$.
Then $G$ is $L$-colorable.
\end{theorem}
\begin{proof}
Assume first that $G$ is  triangle-free and that $L$ is a $4$-list-assignment. The  graph $G$ embeds in the projective plane.
By Euler's formula applied to a projective planar embedding of $G$ we deduce that $G$ has a vertex of degree at most three,
and the theorem follows by induction by deleting such vertex.
For  the other two cases we may assume, by Theorems~\ref{PlanarChoosable} and~\ref{Planar3Choosable}, that $G$
is not planar. Thus $G$ has a planar drawing where exactly two edges cross, say $u_1v_1$ and $u_2v_2$.
Since $G$ is not planar, the vertices $u_1,v_1,u_2,v_2$ are pairwise distinct.
Let $G':=G\setminus\{u_1v_1,u_2v_2\}$.
When $L$ is a $5$-list assignment, the theorem follows from~\cite[Theorem~2]{ThomWheels} applied to
the graph obtained from $G'$ by 
adding the edges $u_1u_2$, $u_2v_1$, $v_1v_2$ and
$v_2u_1$, except for those that are already present, because~\cite[Theorem~2]{ThomWheels} guarantees that we can precolor the
subgraph of $G'$ induced by the vertices $u_1,u_2,v_1,v_2$ arbitrarily.
Thus we may assume that $G$ has girth at least five.
Let us assume first that some two of the vertices $u_1,v_1,u_2,v_2$ are adjacent in $G'$, say $u_1$ and $u_2$ are.
By~\cite[Theorem~2.1]{Thom3ListColor} the graph $G'$ has an $L'$-coloring, where $L'(u_1)\subseteq L(u_1)$, 
$L'(u_2)\subseteq L(u_2)$, $|L'(u_1)|=| L'(u_2)|=1$, $L'(u_1)\ne L'(u_2)$, $L'(v_1)=L(v_1)-L'(u_1)$, $L'(v_2)=L(v_2)-L'(u_2)$ and 
$L'(x)=L(x)$ for every other vertex $x$ of $G'$.
Such an $L'$-coloring of $G'$ is an $L$-coloring of $G$,  as desired.
We may therefore assume that no two of the vertices $u_1,v_1,u_2,v_2$ are adjacent in $G'$.
Let $G''$ be obtained from $G'$ by adding two new vertices $x,y$ and three edges so that $u_1xyu_2$ will become a path in $G''$.
Let $c_1\in L(u_1)$ and $c_2\in L(u_2)$.
By~\cite[Theorem~2.1]{ThoShortlist} applied to the graph $G''$  there exists an $L$-coloring $\phi$ of $G'$ such that
$\phi(u_1)=c_1$, $\phi(u_2)=c_2$, $\phi(v_1)\ne c_1$ and $\phi(v_2)\ne c_2$.
Such an $L$-coloring of $G'$ is an $L$-coloring of $G$,  as desired.
\end{proof}
}

Let us  recall that if
 $L$ is a list assignment for a graph $G$, then we say that $G$ is \emph{$L$-critical} if $G$ does not have an $L$-coloring but every proper subgraph of $G$ does. 
Thomassen~\cite[Theorem~4.4]{ThomCritical} proved the following theorem.

\begin{theorem}\label{Thom7}
Let $G$ be a graph embedded in a surface $\Sigma$ of Euler genus $g$. Let $L$ be a list assignment for $G$ and let $S$ be a set of vertices in $G$ such that 
$|L(v)|\ge 6$ for each $v\in V(G)\setminus S$. If $G$ is $L$-critical, then $|V(G)|\le 150(g+|S|)$.
\end{theorem}

Naturally then, Thomassen~\cite[Problem~5]{ThomCritical} asked whether Theorem~\ref{ThomCrit} generalizes to list-coloring,
\yy{%
and in~\cite[Conjecture~6.1]{ThomExpQuestion} he conjectured that it indeed does.
}

\begin{conjecture}\label{FiniteCrit}
For every surface $\Sigma$ there are (up to isomorphism) only finitely many  graphs $G$  such that $G$
embeds in $\Sigma$ and is $L$-critical for some $5$-list-assignment~$L$.
\end{conjecture}

\noindent
A proof of Conjecture~\ref{FiniteCrit} was announced by Kawarabayashi and Mohar~\cite{KM}, 
but no proof appeared so far.
We give a proof  of Conjecture~\ref{FiniteCrit} in Theorem~\ref{LinearListSurface} below.
%
%
Meanwhile, DeVos, Kawarabayashi, and Mohar~\cite{DeVos} generalized Theorem~\ref{ThomEdgeWidth} to list-coloring.

\begin{theorem}\label{DeVosWidth}
For every surface $\Sigma$ there exists an integer $\gamma$ such that if
$G$ is a graph embedded in $\Sigma$ such that every non-null-homotopic cycle in $G$ has length at least $\gamma$, then $G$ is $5$-list-colorable.
\end{theorem}

\noindent
In Theorem~\ref{EdgeWidth} we give an independent proof of Theorem~\ref{DeVosWidth} 
with an asymptotically best possible bound on $\gamma$. 

Kawarabayashi and Thomassen~\cite{KawThofrom} proved that Theorem~\ref{PlanarChoosable}
``linearly extends" to higher surfaces, as follows.

\begin{theorem}\label{thm:linext5choose}
For every graph $G$ embedded in a surface of Euler genus $g$ there exists a set
$X\subseteq V(G)$ of size at most $1000g$ such that $G\setminus X$ is $5$-choosable.
\end{theorem}

We give an alternate proof of Theorem~\ref{thm:linext5choose} (admittedly with a worse constant) 
and extend it to other coloring problems in Theorem~\ref{thm:linext}.



\subsection{Extending Precolored Subgraphs}

An important proof technique is to extend a coloring of a subgraph to the entire graph.
We will need the following result.


\begin{theorem}
\label{thm:extend4cycle}
Let $G$ be a plane graph, let $P$  be a path in $G$ of length at most three  such that
some face of $G$ is incident with every edge of $P$ and let $L$ be a type $345$ list assignment for $G$.
Then every $L$-coloring of $G[V(P)]$  extends to an
$L$-coloring of $G$.
\end{theorem}

\begin{proof}
When $L$ is a $5$-list assignment, the theorem follows from~\cite[Theorem~2]{ThomWheels},
and when $G$ has girth at least five, it  follows from~\cite[Theorem~2.1]{ThoShortlist}.
Thus  we may assume that $G$  is  triangle-free and that $L$  is a $4$-list assignment.
In that case 
Euler's formula implies that $G$ 
has a vertex in $V(G)-V(P)$ of degree at most three, and the theorem follows by induction by
deleting such vertex.
\end{proof}

\noindent 
Theorem~\ref{thm:extend4cycle} obviously does not extend to arbitrarily long paths. However,
in~\cite{PosThoLinDisk} we have shown the following.
By an {\em outer cycle} in a plane graph  we mean the cycle bounding the {\em outer face}
(the unique unbounded face).




\begin{theorem}\label{LinearCycle0}
Let $G$ be a plane graph with outer cycle $C$, let $L$ be a $5$-list-assignment for $G$, and let $H$ be a minimal subgraph of
$G$ such that every $L$-coloring of $C$ that extends to an $L$-coloring of $H$ also extends to an $L$-coloring of $G$.
Then $H$ has at most $\zz{19}|V(C)|$ vertices.
\end{theorem}


\noindent 
Earlier versions of this theorem were proved for ordinary coloring 
by Thomassen~\cite[Theorem~5.5]{ThomCritical},
who proved it with $\zz{19}|V(C)|$ replaced by $5^{|V(C)|^3}$, and by Yerger~\cite{YerPhD},
who improved the bound to $O(|V(C)|^3)$.
If every vertex of $G\setminus V(C)$ has degree at least five and all its neighbors in $G$ belong to $C$,
then the only graph $H$ satisfying the conclusion of Theorem~\ref{LinearCycle0} is the graph $G$ itself.
It follows that the bound in Theorem~\ref{LinearCycle0} is asymptotically best possible.

Theorem~\ref{LinearCycle0} was a catalyst that  led to this paper, because it motivated us to introduce
the notion of a hyperbolic family of embedded graphs and to develop the theory presented here.
It is an easy consequence of Theorem~\ref{LinearCycle0}, and we show it formally in 
Theorem~\ref{thm:5listcritishyperbolic}, that the family of embedded graphs that are
$L$-critical for some $5$-list assignment is hyperbolic. 



An analog of Theorem~\ref{LinearCycle0} for graphs of girth at least five and $3$-list assignments
was obtained by Dvo\v r\'ak and Kawarabayashi~\cite[Theorem~5]{DvoKawChoosegirth5}.

\begin{theorem}\label{DvoKawCycle}
Let $G$ be a plane graph of girth at  least five, let $C$ be the  outer cycle of $G$ and let  $L$ be
a $3$-list-assignment for $G$. 
Let  $H$ be a minimal subgraph of
$G$ such that every $L$-coloring of $C$ that extends to an $L$-coloring of $H$ also extends to an $L$-coloring of $G$.
Then $H$ has at most $37|V(C)|/3$ vertices.
\end{theorem}

\noindent 
Let us remark that, unlike the previous two results, the corresponding theorem for graphs of girth
at least four and $4$-list assignments is a fairly easy consequence of  Euler's formula.
We state it here for future reference. It follows from Theorem~\ref{lem:44}.

\begin{theorem}\label{LinCycle44}
Let $G$ be a triangle-free plane graph, let $C$ be the  outer cycle of $G$ and let  $L$ be
a $4$-list-assignment for $G$. 
Let  $H$ be a minimal subgraph of
$G$ such that every $L$-coloring of $C$ that extends to an $L$-coloring of $H$ also extends to an $L$-coloring of $G$.
Then $H$ has at most $20|V(C)|$ vertices.
\end{theorem}

Thomassen conjectured~\cite{ThomCritical} that if  $S$ is a set of vertices of a planar graph $G$ and any two distinct members of $S$ are 
at least some fixed distance apart, then any precoloring of $S$ extends to a $5$-coloring of $G$.
Albertson~\cite{Alb} proved this in 1997, and conjectured that it generalizes to $5$-list-coloring.
Albertson's conjecture was recently proved by Dvo\v{r}\'ak, Lidick\'y, Mohar, and Postle~\cite{AlbertsonsConj}, as follows.

\begin{theorem}\label{Albertson}
There exists an integer $D$ such that the following holds: 
If $G$ is a plane graph with a $5$-list assignment $L$ and $X\subseteq V(G)$ is such that every two distinct vertices of $X$ are at distance
at least $D$ in $G$, then any $L$-coloring of $X$ extends to an $L$-coloring of $G$.
\end{theorem}

\def\junk#1{}
\junk{
Dvorak, Lidicky, Mohar, and Postle~\cite{AlbertsonsConj} recently announced a proof of Albertson's conjecture. In Chapter 5, we will give a different proof of Albertson's conjecture more in line with the results of Axenovich, Hutchinson, and Lastrina~\cite{ListPrecoloring}. 
Indeed, Thomassen~\cite{ThomCritical} conjectured something more.
\begin{problem}\label{CyclesListThreeConj}
Let $G$ be a planar graph and $W\subset V(G)$ such that $G[W]$ is bipartite and any two components of $G[W]$ have distance at least $d$ from each other. Can any coloring of $G[W]$ such that each component is $2$-colored be extended to a $5$-coloring of $G$ if $d$ is large enough?
\end{problem}
Thomassen proved Problem~\ref{CyclesListThreeConj} when $W$ consists of two components (see Theorem 7.3 of \cite{ThomCritical}). Albertson and Hutchinson~\cite{AlbHutch3} proved Problem~\ref{CyclesListThreeConj}. As for list coloring, Theorem~\ref{Thom} proves Problem~\ref{CyclesListThreeConj} when $W$ has one component and the question asks whether the coloring can be extended to an $L$-coloring of $G$ where $L$ is a $5$-list-assignment. In Chapter 5, we prove the list-coloring version when $W$ has two components. We believe the results of Chapters 3 and 5 will also yield a proof when $W$ has any number of components but for now this remains open. Note that a proof of the list-coloring vertsion of Problem~\ref{CyclesListThreeConj} was announced without proof by Kawarabayashi and Mohar in~\cite{KM}, where the distance $d$ grows as a function of the number of components of $W$.
}

Albertson and Hutchinson~\cite{AlbHutch2} have generalized Albertson's result to other surfaces. They proved that if the graph is locally planar, then any precoloring of vertices far apart extends: 

\begin{theorem}\label{PrecoloringRegular}
Let $G$ be a graph embedded in a surface 
of Euler genus $g$ such that every non-null-homotopic cycle of $G$ has length at least
$208(2^g-1)$. 
If $X\subseteq V(G)$ is such that  every two distinct vertices of $X$ are at distance at least $18$ in $G$,
then any $5$-coloring of $X$ extends to a $5$-coloring of $G$.
\end{theorem}


Meanwhile, Dean and Hutchinson~\cite{DeanHutch} have proven that if $G$ is a graph embedded in a surface $\Sigma$, $L$ is a $H(\Sigma)$-list-assignment for $V(G)$ and $X\subset V(G)$ such that all distinct vertices $u, v\in X$ have pairwise distance at least four, then any $L$-coloring of $X$ extends to an $L$-coloring of $G$. They also asked whether their result extends to lists of
other sizes. We restate their question as follows.


\begin{question}\label{WhichDist}
For which $k\ge 5$ do there exist $d_k,\gamma_k>0$ such that 
if $G$ is a graph embedded in a surface $\Sigma$
such that every non-null-homotopic cycle in $G$ has length at least $\gamma_k$, 
 $L$ is a $k$-list-assignment for $V(G)$ and $X\subseteq V(G)$ is such that all distinct vertices $u, v\in X$ 
have pairwise distance at least $d_k$, then any $L$-coloring of $X$ extends to an $L$-coloring of $G$. 
\end{question}

\noindent 
In Theorem~\ref{Precolored} we give an independent proof of  a common generalization of
Theorems~\ref{Albertson} and~\ref{PrecoloringRegular}, which also answers Question~\ref{WhichDist}.




We say a graph $G$ is \emph{drawn} in a surface $\Sigma$ if $G$ is embedded in $\Sigma$ except that there are allowed to exist points in $\Sigma$ where two---but only two---edges \emph{cross}. We call such a point of $\Sigma$ and the subsequent pair of edges of $G$, a \emph{crossing}.
Dvo\v{r}\'ak, Lidick\'y and Mohar~\cite{Crossings} proved that crossings far apart instead of precolored vertices also leads to $5$-list-colorability.
\cc{%
The distance between the crossing of the edge $u_1v_1$ with the edge $x_1y_1$ and the crossing of the edge $u_2v_2$
 with the edge $x_2y_2$  is the length of the shortest path in $G$ with one end in the set  $\{u_1,v_1,x_1,y_1\}$
and the other end in the set $\{u_2,v_2,x_2,y_2\}$.} 

\begin{theorem}\label{CrossingPlane}
If $G$ can be drawn in the plane with crossings pairwise at distance at least~15, then $G$ is $5$-list-colorable.
\end{theorem}

\noindent 
In Theorem~\ref{CrossingSurface} we give an independent proof of  Theorem~\ref{CrossingPlane} (but with a worse 
constant) and  generalize it
to arbitrary surfaces.

\junk{
In section 6, we provide an independent proof of Theorem~\ref{CrossingPlane}. Indeed, we generalize Theorem~\ref{CrossingPlane} to other surfaces. Of course, Theorem~\ref{CrossingPlane} does not generalize verbatim as some condition is necessary to even guarantee that a graph drawn on a surface without crossings is $5$-list-colorable. A lower bound on the edge-width seems to be a natural condition that guarantees that a graph is $5$-list-colorable. Thus we will generalize Theorem~\ref{CrossingPlane} to other surfaces with addition requirement of having large edge-width. Indeed we will prove that edge-width logaritheoremic in the genus of the surface suffices, which is best possible.
}

\subsection{Coloring with Short Cycles Far Apart}

Earlier we mentioned the classical of theorem of Gr\"otzsch~\cite{Gro} that every triangle-free planar graph is $3$-colorable.
Aksionov~\cite{Aks} proved that the same is true even if the graph is allowed to have at most three triangles.
Havel~\cite{Havzbarvit} asked  whether
%
 there exists an absolute constant $D$ such that if every two triangles in a planar graph 
are at least distance $D$ apart, then the graph is $3$-colorable. 
Havel's conjecture was proved in~\cite{DvoKraThohavel}.
The conjecture cannot be extended to $3$-list-coloring verbatim, because there exist triangle-free planar
graphs that are not $3$-list-colorable.
However, Dvo\v{r}\'ak~\cite{3ChooseFarApart} proved the following.

\begin{theorem}
\label{thm:dvorak}
There exists an absolute constant $D$ such  that if $G$ is a planar graph and every two cycles in $G$
of length at most four are at distance in $G$ of at least $D$, then $G$  is $3$-list-colorable.
\end{theorem}

We   give a different proof of and generalize Dvo\v{r}\'ak's  result to locally planar graphs on surfaces in Theorem~\ref{PrecoloredGirth5}.

\subsection{Exponentially Many Colorings}

\yy{%
It is an easy consequence of the the Four Color Theorem  that every planar graph has exponentially many $5$-colorings.
Birkhoff and Lewis~\cite{BirLew}  obtained an optimal bound, as follows.
\begin{theorem}\label{thm:birlew}
Every planar graph on $n\ge3$ vertices has at least $60\cdot2^{n-3}$ distinct $5$-colorings, and if it has exactly
$60\cdot2^{n-3}$ distinct $5$-colorings, then it is obtained from a triangle by repeatedly inserting  vertices of degree three
inside  facial triangles.
\end{theorem}
\noindent 
Thomassen~\cite{ThomExpQuestion} proved that a similar bound holds for $5$-colorable graphs embedded in a fixed surface.
\begin{theorem}\label{thm:exp5colsurf}
For every surface $\Sigma$ there exists a constant $c>0$ such that
every $5$-colorable graph on $n\ge1$ vertices embedded in $\Sigma$ has at least $c\cdot2^{n}$ distinct $5$-colorings.
\end{theorem}
%
In~\cite[Theorem~2.1]{ThomExpQuestion} Thomassen gave a short and elegant argument that for every fixed surface $\Sigma$,
if a graph $G$ embedded in $\Sigma$ is $5$-colorable, then it has exponentially many $5$-colorings.
The argument also applies to $4$-colorings of  triangle-free graphs and $3$-colorings of graphs of girth at least five.
In~\cite[Problem~2]{ThomWheels} Thomassen asked whether the bound of Theorem~\ref{thm:birlew} holds for $5$-list-colorings and
}%
proved that a planar graph has exponentially many $5$-list-colorings.

\begin{theorem}\label{ExpPlane}
If $G$ is a planar graph and $L$ is a $5$-list assignment for $G$, then $G$ has at least $2^{|V(G)|/9}$ distinct $L$-colorings.
\end{theorem}


\yy{%
Thomassen~\cite[Problem~3]{ThomWheels} also asked  whether Theorem~\ref{thm:exp5colsurf} extends to $5$-list-colorings.%
}

\begin{problem}\label{ExpConj}
Let $G$ be a graph embedded in a surface $\Sigma$ and let  $L$ be a $5$-list-assignment for $G$. Is it true that if $G$ is $L$-colorable, 
then $G$ has at least $c2^{|V(G)|}$ distinct $L$-colorings, where $c>0$ is a constant depending only on the Euler  genus of $\Sigma$?
\end{problem}

\yy{%
\noindent 
In Theorem~\ref{ExpSurfacetheorem2} we  prove the weaker statement that the answer is yes if $c2^{|V(G)|}$ is replaced by $2^{c|V(G)|}$.
This result was announced by Kawarabayashi and Mohar~\cite{KM}, but no proof appeared so far.
}

\yy{%
Thomasssen~\cite{Thomany} also proved that planar graphs of girth at least five have exponentially  many $3$-list-colorings.
\begin{theorem}\label{ExpGirth5Plane}
If $G$ is a planar graph of girth at least five and $L$ is a $3$-list assignment for $G$, then $G$ has at least $2^{|V(G)|/10000}$ distinct $L$-colorings.
\end{theorem}
}

\zz{%
Theorem~\ref{ExpSurfacetheorem2} extends this to graphs on surfaces.
}


\junk{
Note that a proof of Conjecture~\ref{ExpConj} was announced without proof by Kawarabayashi and Mohar in ~\cite{KM}. We provide an independent proof of Conjecture~\ref{ExpConj} in section 7. Indeed, we will show that precoloring a subset of the vertices still allows exponentially many $5$-list-colorings where the constant depends only on the genus and the number of precolored vertices. In fact, we show that the dependence on genus and the number of precolored vertices can be removed from the exponent.
}



\section{Main Coloring Results}
\label{sec:mainresults}

In this section we formulate Theorem~\ref{thm:maincol}, our main coloring result,
and derive a number of consequences from it.


We need to generalize rooted graphs to allow distinguished facial cycles, rather than just distinguished vertices.
Hence the following definition.

\begin{definition}
\label{def:ring}
A \emph{ring} is a cycle or a complete graph on one or two vertices.
A \emph{graph with rings} is a pair $(G,\R)$, where $G$ is a graph and
$\R$ is a set of vertex-disjoint rings in $G$.
We will often say that $G$ is a graph with rings $\R$, and sometimes
we will just say that $G$ is a graph with rings, leaving out the
symbol for the set of rings.
\end{definition}

\begin{definition}
\label{def:embedded}
We say that a graph $G$ with rings $\R$ is \emph{embedded in
a surface $\Sigma$}
if the underlying graph $G$ is embedded in 
$\Sigma$ in such a way that for every ring
$R\in\R$ there exists a component $\Gamma$ of bd$(\Sigma)$ such that
$R$ is embedded in $\Gamma$,  no other vertex or edge of $G$ is
embedded in $\Gamma$, and every component of bd$(\Sigma)$ includes some
ring of $G$.
In those circumstances we say that $(G,\R,\Sigma)$ is an \emph{embedded graph with rings.}
Sometimes we will abbreviate this by saying that $G$ is an \emph{embedded graph with rings.}
We also say that $G$ is a graph with rings $\R$ embedded in a surface $\Sigma$.
A vertex $v\in V(G)$ is called a {\em ring vertex} if it belongs to some ring in $\R$.
\end{definition}

\noindent 
Thus if $G$ is a graph with rings embedded in a surface $\Sigma$, then the rings are in one-to-one correspondence
with the boundary components of $\Sigma$. In particular, if $G$ has no rings, then $\Sigma$ has no boundary.

\begin{definition}
\label{def:can}
By a {\em canvas} we mean a quadruple $(G,\R,\Sigma,L)$, where $(G,\R,\Sigma)$
is an embedded graph with rings and $L$ is a list assignment for $G$
such that the subgraph $\bigcup\R$ is $L$-colorable.
\end{definition}

Let us remark that this definition of a canvas is  more general than the one we used
in~\cite{PosThoTwotwo} or~\cite{PosThoLinDisk}.
We need the notion of a critical canvas, which we define as follows.

\begin{definition}
\label{def:critical}
Let $G$ be a graph, let $H$  be a subgraph of G and let $L$ be a list assignment for $G$.
We say that $G$ is {\em $H$-critical with respect to $L$} if $G\ne H$ and for every proper subgraph $G'$
of $G$ that includes $H$ as a subgraph there exists an $L$-coloring of $H$ that extends
to an $L$-coloring of $G'$, but does not extend to an  $L$-coloring of $G$.
\zz{We shall abbreviate $\bigcup\R$-critical by $\R$-critical.}
A canvas $(G,\R,\Sigma,L)$ is {\em critical} if $G$ is $\R$-critical
with respect to $L$.
\end{definition}


\begin{definition}
\label{def:C345}
For $k=3,4,5$ we define $\C_k$  to be the family of all canvases $(G,\R,\Sigma,L)$
such that $|L(v)|\ge k$ for every $v\in V(G)$ that does not belong to a ring,
and every cycle in $G$ of length at most $7-k$  is \zz{equal to a ring.}
\end{definition}


The above-defined families are the three most interesting families of canvases to which
our theory applies. More generally, the theory applies to what we call good families of canvases.
Those are defined in Definition~\ref{def:goodcan}, but for now we can treat this notion as a black box.
We now state the fact that the families just defined are good. We prove it in Theorem~\ref{thm:Cisgood-2},
using results from other papers.

\begin{theorem}
\label{thm:Cisgood}
The families $\C_3,\C_4,\C_5$ are good families of canvases.
\end{theorem}

\begin{definition}
If $\Sigma$ is a surface with boundary, let $\widehat{\Sigma}$ denote the 
surface without boundary obtained by
gluing a disk to each component of the boundary.
\end{definition}

The following is our main coloring theorem.

\begin{theorem}
\label{thm:maincol}
For every good family $\C$ of canvases
there exist  $\gamma,a,\epsilon>0$ such that the following holds.
Let $(G,\R,\Sigma,L)\in\C$, 
let $g$ be the Euler genus of $\Sigma$, 
 let $R$ be the total number of ring vertices in $\R$,
and let $M$  be the maximum number of vertices in a ring in $\R$.
Then $G$ has a subgraph $G'$ such that $G'$ includes all the rings in $\R$,
$G'$ has at most $\gamma(g+R)$ vertices
and for every $L$-coloring $\phi$ of $\bigcup\R$
\begin{itemize}
\item
either $\phi$ does not extend to an $L$-coloring of $G'$, or 
\item 
$\phi$ extends to 
at least $2^{\epsilon (|V(G)|-a(g+R))}$ distinct $L$-colorings of $G$.
\end{itemize}
Furthermore, \xx{either} 
\begin{itemize}
\item[{\rm(a)}]
there exist distinct rings $C_1,C_2\in\R$ such that the distance in $G'$ between them is
at most \xx{$\gamma( |V(C_1)|+ |V(C_2)|)$}, or
\end{itemize} 
\xx{every component $G''$ of $G'$} satisfies  one of the following conditions: 
\begin{itemize}
\item [{\rm(b)}]
the graph $G''$ has a cycle $C$ that is not null-homotopic in $\widehat\Sigma$;
 if \xx{$G''$ includes no ring vertex}, then the length of $C$ is at most $\gamma(\log g+1)$, 
and otherwise  it is at most $\gamma (g+M)$, or
\item[{\rm(c)}]
$G''$ includes precisely one member $C$ of $\R$,
there exists a disk $\Delta\subseteq\widehat\Sigma$ that includes $G''$,
and every vertex of $G''$ is at distance at most $\xx{\gamma}\log|V(C)|$ from   $C$.
\end{itemize} 
\end{theorem}

\noindent 
Theorem~\ref{thm:maincol} is an immediate consequence of Theorem~\ref{thm:maincan2},
proved in Section~\ref{sec:can}, where it is deduced from Theorem~\ref{thm:free5}.%
\REM{Deleted: 
We also prove the closely related Theorem~\ref{thm:maincanvar2}, where the bound
in outcome (b) is improved at the expense of a worse bound in outcome (a).}
In the rest of this section we derive consequences of Theorem~\ref{thm:maincol},
some of which are new and some of which improve previous results.
As a first consequence
we improve the bound on $\gamma$ in Theorems~\ref{ThomEdgeWidth} and~\ref{DeVosWidth}, and
extend Theorem~\ref{thm:fiskmohar} to list-coloring, while simultaneously improving upon
the first and  third outcome.

\begin{theorem}\label{EdgeWidth}
There exists an absolute constant $\gamma$ such that if
$G$ is a graph embedded in a surface $\Sigma$ of Euler genus $g$ in such a way such that every non-null-homotopic cycle in $G$ 
has length exceeding $\gamma(\log g+1)$, then $G$ is $L$-colorable for every type $345$
list assignment $L$.
\end{theorem}

\begin{proof}
By Theorem~\ref{thm:Cisgood} the family $\C:=\C_3\cup\C_4\cup\C_5$ is a good family of canvases.
Let $\gamma$ be as in Theorem~\ref{thm:maincol} applied to the family $\C$, 
 let $G,\Sigma$ and $L$ be as stated, and let $\R=\emptyset$.
Then $(G,\R,\Sigma,L)\in\C$. Let $G'$ be a subgraph of $G$ whose existence is guaranteed
by Theorem~\ref{thm:maincol}. Since $\R=\emptyset$ and every non-null-homotopic cycle in $G$ 
has length exceeding $\gamma(\log g+1)$, the only way $G'$ can satisfy one of the conditions (a)--(c)
is that $G'$ is the null graph. It follows that $G$ has an $L$-coloring, as desired.
%
\end{proof}

\noindent
In fact, the proof shows the following strengthening.

\begin{theorem}\label{EdgeWidthExp}
There exist  absolute constants $\gamma,\epsilon,\alpha>0$ such that if
$G$ is a graph embedded in a surface $\Sigma$ of Euler genus $g$ in such a way such that every non-null-homotopic cycle in $G$ 
has length exceeding $\gamma(\log g+1)$ and $L$ is a type $345$ list assignment for $G$, then $G$ has at least 
$2^{\epsilon (|V(G)|-\alpha g)}$ distinct $L$-colorings.
\end{theorem}

\noindent
The bound $\gamma(\log g+1)$ in Theorems~\ref{EdgeWidth} and~\ref{EdgeWidthExp} is asymptotically best possible, because 
by~\cite[Theorem~4.1]{BolExtremal} there exist non-$5$-colorable
graphs on $n$ vertices with girth $\Omega(\log n )$,
and the Euler genus of a graph on $n$ vertices is obviously at most $n^2$.

The next consequence of Theorem~\ref{thm:maincol} settles Conjecture~\ref{FiniteCrit},
improves the bound on the size of $6$-critical graphs in Theorem~\ref{ThomCrit}
and extends Theorem~\ref{LinearGirth5} to list-critical graphs.

\begin{theorem}\label{LinearListSurface}
There exists an absolute constant $\gamma$ such that if
$G$ is a graph embedded in a surface $\Sigma$ 
of Euler genus $g$ and there exists a type $345$-list assignment $L$ for $G$
such that $G$ is $L$-critical, then $|V(G)|\le \gamma g$. 
\end{theorem}

\begin{proof}
By Theorem~\ref{thm:Cisgood} the family $\C:=\C_3\cup\C_4\cup\C_5$ is a good family of canvases.
Let $\gamma$ be as in Theorem~\ref{thm:maincol} applied to the family $\C$, 
let $G,\Sigma$ and $L$ be as stated, and let $\R=\emptyset$.
Then $(G,\R,\Sigma,L)\in\C$. Let $G'$ be a subgraph of $G$, whose existence is guaranteed
by Theorem~\ref{thm:maincol}. 
Then $|V(G')|\le \gamma g$,
%
%
and either $G'$ is not $L$-colorable or $G$ is $L$-colorable.
Since $G$ is $L$-critical it follows that $G'=G$, and hence the theorem holds.
\end{proof}

\yy{%
\noindent
The bound in Theorem~\ref{LinearListSurface} is asymptotically best possible.
For $5$-list-assignments  it can be seen by considering the graphs obtained from copies of $K_6$  by applying
Hajos' construction. For $4$- and $3$-list-assignments we replace $K_6$ by a $5$-critical graph of girth four and
a  $4$-critical graph of girth five, respectively.
}

We have the following immediate corollary.
If $k\ge1$ is an integer, then we say that  a graph $G$ is \emph{$k$-list-critical} if $G$ is not $(k-1)$-list-colorable but every 
proper subgraph of $G$ is.

\begin{corollary}\label{LinearListSurface2}
There exists an absolute constant $c$ such that if
$k\in\{3,4,5\}$ and
$G$ is a $(k+1)$-list-critical graph of girth at least $\xx{8}-k$ embedded in a surface of Euler genus $g$, then $|V(G)|\le cg$. 
\end{corollary}

\begin{proof}
Let $G$ be a $(k+1)$-list-critical graph. Then $G$ is not $L$-colorable for some $k$-list-assignment $L$,
and hence $G$ has an $L$-critical subgraph $G'$. Thus $G'$ is not $L$-colorable, and hence not
$k$-list-colorable. The $(k+1)$-list-criticality of $G$ implies that $G=G'$.
Thus we have shown that  $G$ is $L$-critical, and the corollary follows from Theorem~\ref{LinearListSurface}.
\end{proof}

By the result of Eppstein~\cite{Eppstein2} mentioned earlier we have the following algorithmic consequence.

\begin{corollary}
For every surface $\Sigma$  and every $k\in\{3,4,5\}$ there exists a linear-time algorithm to decide whether an input graph of girth at least $\xx{8}-k$ embedded in  $\Sigma$ is 
$k$-list-colorable.
\end{corollary}

We can also deduce the following algorithm for $L$-coloring.

\begin{corollary}
For every surface $\Sigma$ there exists a linear-time algorithm to decide whether 
given an input graph $G$
 embedded in  $\Sigma$ 
and a type $345$ list assignment $L$ for $G$ there exists an $L$-coloring  of $G$. 
\end{corollary}

\begin{proof}
For simplicity we present the proof for $5$-list assignments, but it clearly extends to the other two types of list assignment.
By Theorem~\ref{LinearListSurface} there exists a finite list $\L$ of pairs $(G',L')$, where $G'$ is a graph embedded in $\Sigma$
and $L'$ is a $5$-list assignment for $G'$, such that an input graph $G$ embedded in $\Sigma$ is $L$-colorable if and only if
there is no $(G',L')\in\L$ such that $(G',L')\preceq(G,L)$,
where $(G',L')\preceq(G,L)$ means that there exists a subgraph isomorphism $f:V(G')\to V(G)$ and a mapping
$g:\bigcup_{v\in V(G')} L'(v)\to \bigcup_{v\in V(G)} L(v)$ such that $g(L'(v))=L(f(v))$ for every $v\in V(G')$.
We may actually assume without loss of generality that for all the pairs $G'',L''$ under consideration and every element $x$,
the set of vertices $\{v\in V(G''): x\in L''(v)\}$ induces a connected subgraph of $G''$. 

The corollary now follows from the fact that, under the assumption we just made, 
for fixed $G',L'$ it can be tested in linear time whether $(G',L')\preceq(G,L)$.
This can be done by  using  Eppstein's algorithm~\cite{Eppstein2}, which works in this more general setting, even though it
is not stated that way, or it can be deduced more directly from some of the generalizations of   Eppstein's algorithm to relational structures,
such as~\cite{DvoKraThofol}.
\end{proof}

Let $G$ be a graph embedded in a surface $\Sigma$, and let $\Sigma'$ be the surface obtained
by attaching a disk to every facial walk of the embedded graph  $G$.
Then $G$ is $2$-cell embedded in $\Sigma'$, and we say that the Euler genus of $\Sigma'$
is the {\em combinatorial genus} of the embedded graph $G$. 
\qq{%
We need the following well-known result. We include a proof for convenience.
\begin{lemma}\label{lem:genusadd}
Let $G$ be a graph embedded in a surface of  Euler genus $g$,  let 
 $G_1,G_2,\ldots,G_n$ be pairwise disjoint subgraphs of $G$, and for $i=1,2,\ldots,n$ let $g_i$ be 
the combinatorial genus of $G_i$. Then $ \sum_{i=1}^n  g_i\le g$.
\end{lemma}
\begin{proof}
We actually prove the stronger statement that if $G=G_1\cup G_2\cup\cdots\cup G_n$ and $g$ is 
the combinatorial genus of $G$, then $ \sum_{i=1}^n  g_i= g$.
It suffices to prove this statement for $n=2$ and we may assume that both $G_1$ and $G_2$ are connected.
There exists a simple closed curve $\cc{\xi}:{\mathbb S}^1\to\Sigma$ whose image is disjoint from $G$,
and either $\cc{\xi}$ is non-separating, or it separates $G_1$ from $G_2$.
But $\cc{\xi}$ cannot be non-separating, given that $g$ is the combinatorial genus of $G$;
thus it separates $G_1$ from $G_2$ and the lemma follows.
\end{proof}
}

Another consequence of Theorem~\ref{LinearListSurface} is the following, which gives an independent proof 
and extension of Theorem~\ref{thm:linext5choose} (but with a worse constant).

\begin{theorem}
\label{thm:linext}
There exists an absolute constant $\gamma$ such that
for every graph $G$ embedded in a surface of  Euler genus $g$
and every type $345$ list assignment $L$ for $G$ there exists a set
$X\subseteq V(G)$ of size at most $\gamma g$ such that $G\setminus X$ is $L$-colorable.
\end{theorem}

\begin{proof}
Let $\gamma$ be as in Theorem~\ref{LinearListSurface}. We claim that $\gamma$ 
satisfies the conclusion of the theorem.
Let $G$  be embedded in a surface $\Sigma$ of  Euler genus $g$, and let  $L$
be a type $345$ list assignment for $G$.
Let $G_1,G_2,\ldots,G_n$ be a maximal family of pairwise disjoint $L$-critical subgraphs of $G$,
and let $X:=V(G_1)\cup V(G_2)\cup\cdots\cup V(G_n)$.
Then $G\setminus X$ is $L$-colorable by the maximality of the family.
For $i=1,2,\ldots,n$ let $g_i$ be the Euler genus of $G_i$. By Theorem~\ref{LinearListSurface}
\qq{and Lemma~\ref{lem:genusadd}}
$$|X|=\sum_{i=1}^n |V(G_i)|\le \sum_{i=1}^n \gamma g_i\le\gamma g,$$
as desired.
\end{proof}

The proof of  Theorem~\ref{LinearListSurface} actually implies the following stronger form, which also generalizes
Theorem~\ref{Thom7} to $5$-list-coloring (albeit with $150$ replaced by a larger constant).
 
\begin{theorem}\label{LinearListSurface3}
There exists an absolute constant $\cc{\gamma}$ with the following property. 
Let $G$ be a  graph \cc{without rings} embedded in a surface $\Sigma$ of Euler genus $g$, let $S\subseteq V(G)$ and $L$ 
be a type $345$ list assignment for $G$. Then there exists a subgraph $H$ of $G$ such that $S\subseteq V(H)$, 
$|V(H)|\le\cc{\gamma}(|S|+g)$ and every $L$-coloring  of $S$ either
\begin{enumerate}
\item  does not extend to an $L$-coloring of $H$, or
\item  extends to an $L$-coloring of $G$. 
\end{enumerate}
\end{theorem}

\cc{%
\begin{proof}
By Theorem~\ref{thm:Cisgood} the family $\C:=\C_3\cup\C_4\cup\C_5$ is a good family of canvases.
Let $\gamma$ be as in Theorem~\ref{thm:maincol} applied to the family $\C$, 
let $G,\Sigma$ and $L$ be as stated,  let $\R$ consist of $S$ as isolated vertices, and 
let $\Sigma'$ be obtained from $\Sigma$ by deleting, for every $v\in S$, an open disk disjoint from $G$ whose 
closure intersects $G$ in $v$.
Then $(G,\R,\Sigma',L)\in\C$. Let $G'$ be a subgraph of $G$, whose existence is guaranteed
by Theorem~\ref{thm:maincol}. 
Then $|V(G')|\le \gamma (g+|S|)$,
%
%
and every $L$-coloring of $S$ either does not extend to an $L$-coloring of $G'$, or
  extends to an $L$-coloring of $G$,
as desired.
\end{proof}
}


%

The next theorem implies the former Albertson's conjecture (Theorem~\ref{Albertson}) and generalizes it to other surfaces,
it generalizes Theorem~\ref{PrecoloringRegular} to list coloring, it answers Question~\ref{WhichDist},
and implies \yy{another  result} which we will discuss next.
\cc{Let us remark that the difference between $G\left[V\left(\bigcup\R\right)\right]$ and $\R$ is that the former 
graph includes edges of $G$ with both ends in $\R$, including those that do not belong to $\R$.}


\begin{theorem}\label{Precolored}
There exist absolute constants $c, D$ such that the following holds.
Let $(G,\R,\Sigma,L)\in\C_3\cup\C_4\cup\C_5$ and
assume that every cycle in $G$ that is non-null-homotopic in $\widehat\Sigma$
 has length at least 
\zz{$cg$, where $g$ is the Euler genus of $\Sigma$.}
If each ring in $\R$ has at most four vertices and
 every two distinct members of $\R$ are at distance in $G$ of at least $D$,
then every $L$-coloring of $G\left[V\left(\bigcup\R\right)\right]$ extends to an $L$-coloring of $G$.
\end{theorem}

\begin{proof}
By Theorem~\ref{thm:Cisgood} the family $\C:=\C_3\cup\C_4\cup\C_5$ is a good family of canvases.
Let $\gamma$ be as in Theorem~\ref{thm:maincol} applied to the family $\C$,
let $D>\xx{8\gamma}$,
and let $c>5\gamma$.
We claim that $D$ and $c$ satisfy the conclusion of the theorem.
To prove that let   $(G,\R,\Sigma,L)$ be as stated.
By  Theorem~\ref{thm:maincol} there exists a subgraph $G'$ of $G$ 
such that $G'$ includes all the rings in $\R$
and for every $L$-coloring $\phi$ of $\bigcup\R$,
either $\phi$ does not extend to an $L$-coloring of $G'$, or $\phi$ extends to
an $L$-coloring of $G$;
and  $G'$ satisfies one of (a)--(c) of  Theorem~\ref{thm:maincol}.
But $G'$ does not satisfy (a) or (b) by hypothesis, and hence it satisfies (c).
It follows from Theorem~\ref{thm:extend4cycle} that
every $L$-coloring of $G\left[V\left(\bigcup\R\right)\right]$ extends to an $L$-coloring of $G'$.
Consequently every $L$-coloring of $G\left[V\left(\bigcup\R\right)\right]$ extends to an $L$-coloring of $G$, as desired.
\end{proof}

\REM{Deleted: 
By applying Theorem~\ref{thm:maincanvar2} rather than Theorem~\ref{thm:maincol} 
we obtain the following  related result.}



The following theorem follows immediately from Theorem~\ref{Precolored}.
The first assertion gives an independent proof of and generalizes Theorem~\ref{thm:dvorak}.

\begin{theorem}\label{PrecoloredGirth5}
There exist absolute constants $c, D$ such that the following holds.
Let $G$ be a graph \aa{without  rings} embedded in a surface $\Sigma$ \aa{(without boundary)} 
of Euler genus $g$ in such a way that every non-null-homotopic cycle 
in $G$ has length at least $cg$. Then
\begin{itemize}
\item if every two cycles in $G$ of length at most four are at distance at least $D$ in $G$, then $G$ is $3$-list-colorable, and
\item if every two triangles in $G$ are at distance at least $D$ in $G$, then $G$ is $4$-list-colorable.
\end{itemize}
\end{theorem}

We obtain the following generalization and independent proof  of Theorem~\ref{CrossingPlane}.
\yy{%
For $5$-list-assignments it follows from Theorem~\ref{Precolored} by replacing each crossing by a cycle of length four,
but for the other two cases we need a more careful argument.
}

\begin{theorem}\label{CrossingSurface}
There exist absolute constants $c$ and $D$ such that the following holds.
Let $G$ be a graph  \aa{without  rings} drawn with crossings in a surface $\Sigma$ \aa{(without boundary)} 
of Euler genus $g$,
and let $L$ be a type $345$ list assignment for $G$.
\xx{Assume that every edge crosses at most one other edge.}
Let $G'$ be the graph embedded in $\Sigma$ obtained from $G$ by repeatedly replacing a pair of crossing edges $e,f$ by a degree four vertex adjacent
to the ends of $e$ and~$f$. 
If every non-null-homotopic cycle in $G'$ has length at least $cg$ and every pair of the new vertices of $G'$ are at distance
at least $D$ in $G'$, then $G$ is $L$-colorable.
\end{theorem}

\begin{proof}
\yy{%
By Theorem~\ref{thm:Cisgood} the family $\C:=\C_3\cup\C_4\cup\C_5$ is a good family of canvases.
Let $\gamma$ and $a$ be as in Theorem~\ref{thm:maincol}, when the latter is applied to the family $\C$.
Let $D:=\xx{32\gamma+3}$ and $c:=17\gamma$.
We will show  that $D$ and $c$ satisfy the conclusion of the theorem.
}

\yy{%
We construct a graph $H$ with rings embedded in a surface $\Sigma'$ as follows.
For every pair of crossing edges $u_1v_1$ and $u_2v_2$ we do the following. First of all, we may assume that the vertices $u_1,v_1,u_2,v_2$
are pairwise distinct, for otherwise we could eliminate the crossing by redrawing one of the crossing edges.
We add $12$ new vertices and $16$  new edges in such a way that $u_1,v_1,u_2,v_2$ and the new vertices will form a facial cycle $C$,
the vertices $u_1,u_2,v_1,v_2$ will appear on $C$ in the order listed, and will be at distance at least four apart on $C$.
We remove the face  bounded by $C$ from the surface and we declare $C$ to be a ring.
By repeating this construction for every pair of crossing edges we arrive at a graph $H$ with rings $\R$ embedded in a surface $\Sigma'$.
For $v\in V(G)$ let $L'(v)=L(v)$, and for $v\in V(H)-V(G)$ let $L'(v)$ be an arbitrary set of size five.
It follows that  $(H,\R,\Sigma',L')\in\C$.
}

\yy{%
By Theorem~\ref{thm:maincol} there exists a subgraph $H'$ of $H$ that includes all the rings in $\R$ such that $H'$ satisfies one of
the conditions (a)--(c) of Theorem~\ref{thm:maincol}, and every $L'$-coloring of $\bigcup\R$ that extends to an $L'$-coloring of $H'$
also extends to an $L'$-coloring of $H$.
The choice of $c$ and $D$ implies that $H'$ does not satisfy (a) or (b), and hence it satisfies (c).
}

\yy{%
We now construct an $L'$-coloring $\phi$ of $\bigcup\R$. Let $C\in\R$. Let $H''$ be the unique component of $H'$ containing $C$.
By (c) it includes no other ring. Let $G''$ be the corresponding subgraph of $G$; that is, $G''$ is obtained from $H''$ by deleting
the new vertices and adding the two crossing edges with ends in $V(H'')$. By Theorem~\ref{CrossOneChoosable}  the graph $G''$
has an $L$-coloring. Let the restriction of $\phi$ to $C$ be obtained by taking the restriction of one such $L$-coloring to the four ends  of the crossing edges,
and extending it to $C$  arbitrarily. This completes the construction of the $L'$-coloring $\phi$. It follows from the construction
that $\phi$ extends to an $L'$-coloring of $H'$, and hence it extends to an $L'$-coloring of $H$, which by construction is also an
$L$-coloring of $G$, as desired.
}
\end{proof}
\REM{Deleted: 
Similarly,
\yy{by applying Theorem~\ref{thm:maincanvar2} rather than Theorem~\ref{thm:maincol}, one can obtain}
a version of the above theorem where the length of non-null-homotopic cycles is at least  $\Omega(\log g+1)$
and the new vertices are pairwise at distance at least $\Omega(\log(g+2))$.
\yy{We omit the details.}}

Finally we 
\yy{%
prove a weaker version of Problem~\ref{ExpConj}.%
}
It is implied by the following more general theorem by setting $\R=\emptyset$. 



\begin{theorem}
\label{ExpSurfacetheorem2}
There exist constants $\epsilon,a>0$ such that the following holds. 
Let $(G,\R,\Sigma,L)\in\C_3\cup\C_4\cup\C_5$, let $g$ be the Euler genus of $\Sigma$, and
let $R$ be the total number of ring vertices.
 If $\phi$ is an $L$-coloring of $\bigcup\R$ such that $\phi$ extends to an $L$-coloring of $G$, 
then $\phi$ extends to at least $2^{\epsilon(|V(G)|-a(g+R))}$ distinct $L$-colorings of $G$.
\end{theorem}

\begin{proof}
By Theorem~\ref{thm:Cisgood} the family $\C:=\cc{\C_3\cup\C_4\cup\C_5}$ is a good family of canvases.
Let $\epsilon,a$ be as in Theorem~\ref{thm:maincol} applied to the family $\C$, 
let $(G,\R,\Sigma,L)$ be as stated, and let $G'$ be a subgraph of $G$ as in Theorem~\ref{thm:maincol}.
It follows that if  an $L$-coloring $\phi$ of $\bigcup\R$  extends to an $L$-coloring of $G$, 
then $\phi$ extends to at least $2^{\epsilon(|V(G)|-a(g+R))}$ distinct $L$-colorings of $G$.
\end{proof}


\section{Exponentially Many Colorings}\label{sec:expmany}

The main result of this short section is Lemma~\ref{ExpManyDisc}, which strengthens Theorem~\ref{LinearCycle0}
by saying that if \aa{an $L$-coloring of  $C$} extends to $G$, then it has exponentially many extensions.


\aa{%
When  Thomassen proved Theorem~\ref{ExpPlane} in~\cite[Theorem~4]{ThomWheels},
}%
he proved a stronger statement, 
which we state in  a slightly stronger form.

\begin{theorem}\label{ThomExp}
Let $G$ be a  plane graph, let $Z$ be the set of all vertices of $G$ that are incident with the outer face,
let $P$ be a subpath of $G$ of length one or two with every vertex and edge incident with the outer face,
and let $L$ be a list assignment for $G$ such that $|L(v)|\ge 3$ for every $v\in Z-V(P)$ and
$|L(v)|\ge5$ for every $v\in V(G)-Z$.
Let $r$ be the number of vertices $v\in Z$ such that $|L(v)|=3$.
If $G$ has an $L$-coloring,
then it has at least $2^{|V(G)-V(P)|/9-r/3}$ distinct $L$-colorings, unless $|V(P)|=3$ and there exists 
$v\in V(G)-V(P)$ such that $|L(v)|=4$ and $v$ is adjacent to all the vertices of $P$.
\end{theorem}

\begin{proof}
This is proved in~\cite[Theorem~4]{ThomWheels} under the additional assumption that $G$ is a near-triangulation.
If the outer face of G is bounded by a cycle, then the theorem follows 
from~\cite[Theorem~4]{ThomWheels}  by adding edges.
Otherwise $G$ can be written as  $G=G_1\cup G_2$ , where  $|V(G_1\cap G_2)|\le1$,
and the theorem follows by induction applied to  $G_1$ and $G_2$.
\end{proof}

We use Theorem~\ref{ThomExp} to deduce the following corollary.

\begin{corollary}\label{ExpFourCycle}
Let $G$ be a plane graph with outer cycle $C$ of length at most four,
let $L$ be a $5$-list assignment for $G$,
let $\phi$ be an $L$-coloring of $C$,
and let $E$ be the number of $L$-colorings of $G$ that extend $\phi$.
If $E\ge1$, then $\log_2 E \ge |V(G)-V(C)|/9$, unless $|V(C)|=4$ and there 
exists a vertex not in $V(C)$ adjacent to all the vertices of $C$.
\end{corollary}
\begin{proof}
Let $v\in V(C)$. Let $G'=G\setminus v$, $P=C\setminus v$, and for $w\in V(G')$ let
$L'(w)=L(w)\setminus\{\phi(v)\}$ if $w$ is a neighbor  of $v$ and $L'(w)=L(w)$ otherwise.
 Apply Theorem~\ref{ThomExp} to $G',P$ and $L'$. 
It follows that there are at least 
$2^{|V(G')-V(P)|/9}=2^{|V(G)-V(C)|/9}$ distinct $L'$-colorings of $G'$,
 unless $|V(P)|=3$ and there exists $x\in V(G')-V(P)$ with $|L'(x)|=4$ and $x$ is adjacent to all vertices of $P$. 
But then $|V(C)|=4$ and $x$ is adjacent to all the vertices of $C$.
\end{proof}

\begin{definition}
Let us recall that the \emph{outer face} of a
plane graph $G$  is the unique unbounded face;
all other faces of $G$ are called \emph{internal}. 
We  denote the set of internal faces of $G$ by $\F(G)$.
If $f$ is a face of  a $2$-connected plane graph $G$, then we let $|f|$ denote the length of the cycle bounding $f$.
Likewise, if $C$ is a cycle in $G$, then we denote its length by $|C|$.
\cc{If $G'$ is a subgraph of $G$ and  $f\in\F(G')$, then}
 we denote by $G[f]$ the subgraph of $G$ consisting of all vertices and edges
drawn in the closure of $f$.
If $G$ is a $2$-connected plane graph with outer cycle $C$, then we define
$${\rm def}(G):=|C|-3 - \sum_{f\in \F(G)} (|f|-3).$$
\end{definition}


Let us recall that $C$-critical graphs were defined in Definition~\ref{def:critical}.
When we proved  Theorem~\ref{LinearCycle0} in~\cite{PosThoLinDisk}, we have actually proven the following
stronger result~\zz{\cite[Theorem~6.1]{PosThoLinDisk}}, which we will now need.

\begin{theorem}\label{StrongLinear2}
Let $G$ be a plane graph with outer cycle $C$, and
let $L$ be a $5$-list assignment for $G$.
If $G$ is $C$-critical with respect to $L$, then $$|V(G)\setminus V(C)|/\zz{18} + \sum_{f\in \F(G)}(|f|-3)\le |C|-4.$$
\end{theorem}

We are now ready to prove the main result of this section.

\begin{lemma}\label{ExpManyDisc}
Let $G$ be a plane graph with outer cycle $C$,
let $L$ be a $5$-list assignment for $G$,
 and let $\phi$ be an $L$-coloring of $C$ that extends to an $L$-coloring of $G$.
Then $\log_2 E(\phi) \ge (|V(G)-V(C)| - \zz{19}(|V(C)|-3))/9$, where $E(\phi)$ is the number of extensions of $\phi$ to $G$.
\end{lemma}
\begin{proof}
We proceed by induction on the number of vertices of $G$. 
It  follows from Corollary~\ref{ExpFourCycle} that we may assume that
$G$ is $2$-connected and
there does not exist a  separating triangle in $G$, and that  if there exists a separating $4$-cycle $C'$ in $G$, then there must exist a vertex adjacent to all the vertices of $C'$, as otherwise the theorem follows by induction.
We may assume that $|V(C)|\ge 5$ by Corollary~\ref{ExpFourCycle}.

First suppose there exists a vertex $v\in V(G)$ with 
 at least three neighbors on $C$, and let $G'=G[V(C)\cup \{v\}]$.
Let us first handle the case when $v$ has 
at least four neighbors on $C$. 
Then ${\rm def}(G')\ge 1$, as is easily seen.
Let $\phi'$ be an extension of $\phi$ to $V(C)\cup\{v\}$ that extends to an $L$-coloring of $G$.
For all $f\in \F(G')$, letting $C_{f}$ denote the cycle bounding $f$, it follows by induction that $\phi'$ has at least
$2^{ (|V(G[f]\setminus V(C_f))|-\zz{19}(|C_f|-3))/9}$ extensions into $G[f]$. 
Thus 
\begin{align*}
9\log_2 E(\phi) &\ge \sum_{f\in \F(G')} \cc{\left( |V(G[f]\setminus V(C_f))|-\zz{19}(|C_f|-3)\right)}\\
&\ge |V(G\setminus V(C))|-1 - \zz{19}(|C|-4)
\ge |V(G\setminus V(C))|- \zz{19}(|C|-3),
\end{align*}
 where the second inequality holds because ${\rm def}(G')\ge 1$.
The lemma follows.


Let us assume next that 
$v$ has exactly three neighbors on $C$. 
Let $c_1,c_2\in L(v)$ be distinct colors not equal to $\phi(u)$ for all neighbors $u$ of $v$ that belong to $C$.
For $i\in\{1,2\}$ let $\phi_i(v)=c_i$ and $\phi_i(u)=\phi(u)$ for all $u\in V(C)$. 
If both $\phi_1$ and $\phi_2$ extend to $L$-colorings of $G$, it follows by induction applied to the faces of $G'$,
using the same calculation as in the previous paragraph and the fact that ${\rm def}(G')\ge 0$,
that $\log_2 E(\phi_i)\ge ((|V(G\setminus V(C))|-1) - \zz{19}(|V(C)|-3))/9$. Yet $E(\phi)\ge E(\phi_1)+E(\phi_2)$; hence $\log_2 E(\phi)\ge (|V(G\setminus V(C))|-\zz{19}(|V(C)|-3))/9$ and the lemma follows.

So we may suppose without loss of generality that $\phi_1$ does not extend to an $L$-coloring of $G$. 
Hence there exists $f\in \F(G')$ such that $G[f]$ has a subgraph $G_f$ that is $C_f$-critical
with respect to $L$.
Let $G''=G[V(G_f)\cup V(G')]$.
Then 
$$
|V(G''\setminus V(C))| = |V(G_f\setminus V(C_f))|+1\le \zz{19}({\rm def}(G_f)-1)+1\le \zz{19}{\rm def}(G_f)=\zz{19}{\rm def}(G''),
$$
where the first inequality follows from Theorem~\ref{StrongLinear2}.
%
As $\phi$ extends to an $L$-coloring of $G$, $\phi$ extends to an $L$-coloring $\phi''$ of $G''$ that extends to $G $.
For all $f'\in \F(G'')$, letting $C_{f'}$ denote the cycle bounding $f'$, it follows by induction that $\phi''$ has at least
$2^{ (|V(G[f']\setminus V(C_{f'}))|-\zz{19}(|C_{f'}|-3))/9}$ extensions into $G[f']$.
Thus 
\begin{align*}
9\log_2 E(\phi) &\ge \sum_{f'\in \F(G'')}  \cc{\left(|V(G[f']\setminus V(C_{f'}))|-\zz{19}(|C_{f'}|-3)\right)}\\
& = |V(G\setminus V(C))|-|V(G''\setminus V(C))|- \zz{19}(|C|-3-{\rm def}(G''))\\
& \ge |V(G\setminus V(C))| - \zz{19}(|C|-3),
\end{align*}
 and the lemma follows.

We may therefore assume that every vertex in $V(G)\setminus V(C)$  has at most two neighbors in $C$.
Let us redefine $G' $ to be the graph 
 $G\setminus V(C)$. For $v\in V(G')$ let $L'(v)$ be obtained from $L(v) $
by removing $\phi(u)$ for every neighbor $u$ of $v $ that belongs to $C$, and
let  $S=\{v\in V(G')| |L'(v)|=3\}$. Thus every vertex in $S$ has exactly two neighbors in $C$.
Let us define an auxiliary multigraph $H$ with vertex-set $V(C)$ and edge-set $S$ by saying
that the ends of an edge $s\in S$ are the
 two neighbors of $s$ in $C$.
Then $H$ is an outerplanar multigraph.

Let $u,v$ be adjacent vertices of $H$. 
We claim that there are at most three edges of $H$ with ends $u,v$.
Indeed, suppose that say $s_1,s_2,s_3,s_4\in E(H)=S$ are pairwise distinct and all have ends $u,v$.
Then for all distinct integers $i,j\in\{1,2,3,4\}$ the vertices $u,s_i,v,s_j$ form
a cycle of length four. For one such pair $i,j$  this cycle includes the two vertices of
$\{s_1,s_2,s_3,s_4\}-\{s_i,s_j\}$ in its interior, contrary to what we established at the beginning of the proof.
Thus  at most three edges of $H$ have ends $u,v$.
Furthermore, we claim that if $u,v$ are adjacent in $C$, then at most one edge of $H$ has ends $u,v$.
Indeed, suppose that say distinct $s_1,s_2\in E(H)=S$ both have ends $u,v$.
Then for  $i\in\{1,2\}$ the vertices $u,s_i,v$ form
a cycle of length three, 
contrary to there not existing a separating triangle as established at the beginning of the proof.

The results of the previous paragraph imply that 
$|S|=|E(H)|\le  |C| + 3(|C|-3)$. Since $|C|\ge5$ we deduce that $|S|\le \zz{19}(|C|-3)/3$.
By  Theorem~\ref{ThomExp} applied to $G'$, $L'$ and an arbitrarily chosen one-edge path $P$ we obtain
 $$\log_2 E\xx{(\phi)} \ge |V(G\setminus V(C))|/9- |S|/3\ge |V(G\setminus V(C))|/9- \zz{19}(|C|-3)/9,$$
as desired.
%
\end{proof}

\section{Definition and Examples of Hyperbolic Families}
\label{sec:hyperdef}


%

In this section we define hyperbolic families in  full generality and we give examples of families
of embedded graphs with rings that are hyperbolic.
Let us recall that  embedded graphs with rings were defined at the beginning of Section~\ref{sec:mainresults}.

\begin{definition}
Let $\F$ be a family of non-null embedded graphs with rings.
We say that $\F$ is \emph{hyperbolic} if there exists a constant $c>0$
such that if $G\in\F$ is a graph with rings that is embedded in a surface
$\Sigma$, then for every closed curve $\cc{\xi}:{\mathbb S}^1\to\Sigma$ that bounds
an open disk $\Delta$ and intersects $G$ only in vertices,
if $\Delta$ includes a vertex of $G$, then
the number of vertices of $G$ in $\Delta$ is at most
$c(|\{x\in {\mathbb S}^1\,:\, \cc{\xi}(x)\in V(G)\}|-1)$.
We say that $c$ is a \emph{Cheeger constant} for $\F$.
Let us point out that since $\Delta$ is an open disk, it does \cc{not} include the 
vertices of $G$ that belong to the image of $\cc{\xi}$.
\end{definition}

The next lemma shows that if a planar graph with rings  belongs to a hyperbolic family, then it has at least one ring.

\begin{lemma}\label{lem:onering}
Let $\F$ be a hyperbolic family of embedded graphs with rings, and let $c$ be a Cheeger constant for $\F$. 
Let $G\in \F$ be a graph embedded in 
a surface of genus zero  with $r$ rings and let $R$ be the total number of ring vertices. 
Then $r\ge 1$ and if $r=1$, then 
\yy{$G$ has at most $c(R-1)$ non-ring vertices.}
\end{lemma}
\begin{proof}
As $G$ is non-null, $r \ge 1$ as otherwise we may take a closed curve not intersecting any vertices of $G$ that bounds a disk containing all vertices of $G$, a contradiction to the definition of hyperbolic family. 
If $r=1$, then let $\cc{\xi}$ be a closed curve obtained from a closed curve tracing the ring by pushing it slightly into the interior of
$\Sigma$ in such a way that the image of ${\mathbb S}^1$ under $\cc{\xi}$ will intersect the boundary of $\Sigma$ only in $V(G)\cap\hbox{bd}(\Sigma)$.
By the definition of hyperbolic family applied to the curve $\cc{\xi}$ we deduce that 
\yy{$G$ has at most $c(R-1)$ non-ring vertices,}
as desired.
\end{proof}

\subsection{Examples arising from (list-)coloring}

Theorem~\ref{LinearCycle0} implies that
embedded $6$-list-critical graphs
form a hyperbolic family. To obtain a better bound on the Cheeger constant we actually use
the stronger Theorem~\ref{StrongLinear2}.
More generally, we have the following.

\begin{theorem}
\label{thm:5listcritishyperbolic}
Let $\F$ be the family of all embedded graphs $G$ with rings $\R$ such that
$G$ is $\R$-critical with respect to some $5$-list assignment.
Then $\F$ is hyperbolic with Cheeger constant~$\zz{18}$.
\end{theorem}
\begin{proof}
Let $G$ be a  graph with rings $\R$
embedded in a surface $\Sigma$ of Euler genus $g$
such that $G$ is $\R$-critical with respect to  
a $5$-list assignment $L$, let $R$ be the total number of ring vertices, and
let $\cc{\xi}:{\mathbb S}^1\to\Sigma$ be a closed curve that bounds
an open disk $\Delta$ and intersects $G$ only in vertices.
To avoid notational complications we will assume that $\cc{\xi}$ is a simple
curve; otherwise we split vertices that $\cc{\xi}$ visits more than once to
reduce to this case.
We may assume that $\Delta$ includes at least one vertex of $G$, for 
otherwise there is nothing to show.
Let $X$ be the set of vertices of $G$ intersected by $\cc{\xi}$.
Then $|X|\ge4$ by Theorem~\ref{thm:extend4cycle}.

Let $G_0$ be the subgraph of $G$ consisting of all vertices and edges drawn 
in the closure  of $\Delta$.
Let $G_1$ be obtained from $G_0$ as follows. For every pair of vertices $u,v\in X$
that are consecutive on the boundary of $\Delta$ we do the following. 
If $u$ and $v$ are adjacent in $G_0$, then we re-embed the edge $uv$ so that it will 
coincide with a segment of $\cc{\xi}$.
If $u,v$ are not adjacent, then we introduce a new vertex and join it
by edges to both $u$ and $v$, embedding the new edges in a segment of $\cc{\xi}$.
Thus $G_1$ has a cycle $C_1$ embedded in the image of $\cc{\xi}$, and hence
$G_1$ may be regarded as a plane graph with outer cycle $C_1$.
For $v\in V(G_0)$ let $L_1(v):=L(v)$, and for $v\in V(G_1)-V(G_0)$ let $L_1(v)$
be an arbitrary set of size five.
It follows that $G_1$ is $C_1$-critical with respect to $L_1$.
Since every vertex $v\in V(G_1)-V(G_0)$ is incident with an internal face of length
at least four, Theorem~\ref{StrongLinear2} implies that the number of vertices
of $G$ in $\Delta$ is at most $\zz{18}(|X|-4)$, as desired.
\end{proof}

For our next example, we need the following lemma.
The degree of a vertex $v$ in a graph $G$ will be denoted by $\deg_G(v)$,
or $\deg(v)$ if the graph is understood from the context.

\begin{lemma}
\label{lem:deg7}
Let $G$ be a graph embedded in the closed unit disk with $k$ vertices embedded on the boundary 
and $n\ge1$ vertices embedded in  the interior of the disk.
If every vertex embedded in the interior of the disk has degree in $G$ of at least seven,
then $n\le k-6$. 
\end{lemma}

\begin{proof}
We may assume that every vertex on the boundary of the disk has a neighbor in the
interior, for otherwise we may delete the boundary vertex and apply induction.
Let us first assume that $k\ge3$, and let $H$ be the planar graph obtained from
$G$ by joining by an edge every pair of non-adjacent consecutive vertices on the boundary
of the disk, and then adding one new vertex in the complement of the disk and joining it by an edge 
to every vertex on the boundary. It follows that every vertex on  the boundary of the disk
 has degree at least four in $H$. Since $H$ is planar, we have
 \begin{equation*}
6(n+k+1)-12\ge2|E(H)|=\sum_{v\in V(H)}\deg_H(v)\ge 7n+4k+k,
\end{equation*}
and the lemma follows.
If $k\le2$, then 
\zz{%
by the planarity of $G$
\begin{equation*}
6n\ge 6(n+k)-12\ge2|E(H)|=\sum_{v\in V(G)}\deg_G(v)\ge 7n,
\end{equation*}
}%
a contradiction, since $n\ge1$.
\end{proof}

\begin{example}
Let $\F$ be the family of all embedded graphs $G$ with rings such that
every vertex of $G$ that does not belong to a ring has degree at least seven. 
Then $\F$ is hyperbolic with Cheeger constant $1$ by Lemma~\ref{lem:deg7}.
\end{example}

\begin{example}
The family $\F$ from the previous example includes 
all embedded graphs $G$ with rings $\R$ such that
$G$ is $\R$-critical with respect to some $7$-list assignment.
That gives an alternate proof that the latter family is hyperbolic, and gives
a better Cheeger constant than Theorem~\ref{thm:5listcritishyperbolic}.
\end{example}

For our next example, we need the following lemma.
\qq{%
The {\sl size} of a face $f$ of an embedded graph, denoted by $|f|$,
is the sum of the lengths of the walks bounding $f$.}

\begin{lemma}
\label{lem:deg6}
Let $G$ be a graph embedded in the closed unit disk with $k$ vertices embedded on the boundary 
and $n\ge1$ vertices embedded in  the interior of the disk.
Assume that every vertex embedded in the interior of the disk has degree in $G$ of at least six,
and if it has degree exactly six, then it is either incident with a face of size at least four,
or it is adjacent to a vertex of degree at least seven, or it is adjacent to a vertex embedded on the 
boundary of the disk.
Then $n\le 9k-48$. 
\end{lemma}

\begin{proof}
We may assume that every vertex on the boundary of the disk has a neighbor in the
interior, for otherwise we may delete the boundary vertex and apply induction.
Let us first assume that $k\ge3$, and let $H$ be the planar graph obtained from
$G$ by joining by an edge every pair of non-adjacent consecutive vertices on the boundary
of the disk, and then adding one new vertex in the complement of the disk and joining it by an edge 
to every vertex on the boundary. It follows that every vertex on  the boundary of the disk
 has degree at least four in $H$. 

Let $N_6$ be the set of vertices embedded in the interior of the disk with degree six,
and let $N_7$ be the set of vertices embedded in the interior of the disk that have degree at least seven. 
Let $A$ be the subset of $N_6$ 
consisting of vertices adjacent
to a vertex embedded on the boundary of the disk, 
let $B$ be the subset of $N_6$  consisting of vertices incident with a face of size at least four and let $C$
be the subset of $N_6$ consisting of vertices adjacent to a vertex of degree at least seven. 
By assumption, $N_6 = A \cup B \cup C$.
Since $H$ is planar, we have
$$
6(n+k+1)-12 - 2\sum_{f\in\ F(H)} (|f|-3)= 2|E(H)|=\sum_{v\in V(H)}\deg_H(v) $$
$$
\ge \sum_{v\in N_7}(\deg_H(v)-6) +6n +4k+k.
$$
%
Hence $$k-6\ge \sum_{v\in  N_7}(\deg_H(v)-6) + 2\sum_{f\in \F(H)} (|f|-3).$$
However, $|B|$ is at most the sum of the sizes of faces of sizes at least four. That is, 
$$|B|\le \sum_{f\in \F(H), |f|\ge 4} |f| \le 4 \sum_{f\in F(H), |f|\ge 4} (|f|-3).$$
Similarly, $|C|$ is at most the sum of degrees of vertices of degree at least seven. 
That is, 
$$|C| \le \sum_{v\in N_7} \deg_H(v) \le 7\sum_{v\in  N_7} (\deg_H(v)-6).$$
Hence $|B|+|C|\le 7(k-6)$. Yet since $H$ is planar, 
$$6n+6k-6 \ge \sum_{v\in V(H)} \deg_H(v) \ge 6n + n-|N_6| + 3k + |A| + k.$$
Hence $2k-6\ge |A|+n-|N_6|$. Thus $$n \cc{\le} |B|+|C|+|A|+n-|N_6|\le 7(k-6)+2k-6 = 9k-48,$$ as desired.
If $k\le2$, then a similar calculation applied to the original graph $G$ gives a contradiction.
\end{proof}

\begin{example}
Let $\F$ be the family of all embedded graphs $G$ with rings such that
every vertex of $G$ that does not belong to a ring has degree at least six,
and   if it has degree exactly six, then it is either incident with a face of size at least four,
or it is adjacent to a vertex of degree at least seven, or it is adjacent to a vertex that belongs to a ring.
Then $\F$ is hyperbolic with Cheeger constant $9$ by Lemma~\ref{lem:deg6}.
\end{example}

\begin{example}
The family $\F$ from the previous example includes 
all embedded graphs $G$ with rings $\R$ such that
$G$ is $\R$-critical with respect to some $6$-list assignment.
That gives an alternate proof that the latter family is hyperbolic, and gives a
better Cheeger constant  than Theorem~\ref{thm:5listcritishyperbolic}.
\end{example}

For our next example, we need the following lemma.

\begin{lemma}
\label{lem:44}
Let $(G,\R,\Sigma)$ be an embedded graph with rings such that 
every vertex that does not belong to a ring has degree at least four,
every face of $G$ of size \xx{three is incident with an edge of $\R$, and every face of} 
 size exactly four  is incident with either a vertex of degree at least five,
or a vertex \qq{of degree at least one} that belongs to a ring. 
Let $g$  be the Euler genus of $\Sigma$ and let $R$ be the total number of ring vertices.
Then $|V(G)|\le 20(g+R-2)$.
\end{lemma}

\begin{proof}
For the purpose of this proof we will regard $G$ as a graph without rings embedded in
the surface $\widehat\Sigma$. Thus every cycle in $\R$ bounds a face.
We may assume that every ring vertex has degree at least one, for otherwise we may delete it and apply induction.
Let $\F$ be the set of all faces of $G$.
We begin by giving each vertex  $v\in V(G)$ a charge of $\deg(v)-4$ and every face $f\in\F$
a charge of $|f|-4$. By Euler's formula the sum of the charges is \qq{at most} $4(g-2)$.
Next we add \xx{four units of charge to every ring vertex of degree one}, 
$3$ units of charge to every ring vertex \xx{of degree at least two} and   \xx{$3$  units of charge}  to every face bounded by
a cycle in $\R$. 
Now the sum of charges is at most $4(g+R-2)$, and every vertex and every face has
non-negative charge\qq{, except for triangles which have charge $-1$}. 
In the next step every vertex of degree at least five and every ring vertex of degree at least two
will send $1/5$ of a unit of charge to every incident face of length four, \xx{and every ring vertex of degree 
at least \qq{two} will send an additional $3/5$ units of charge to every incident face of length three}. 
Since no ring vertex of degree one is incident with a face of length four,
this redistribution of the charges \qq{guarantees} that every vertex and every face has
non-negative charge. Furthermore, every face will have charge of at least $1/5$,
\xx{and every ring vertex will have charge of at least $\qq{2}/5$.}
Finally, every face will send $1/20$ of a unit of charge to every incident vertex. Then every face will end up
with nonnegative charge and every vertex will end up with a charge of at least $4/20=1/5$. Thus
$|V(G)|/5\le 4(g+R-2)$, and the lemma follows.
\end{proof}

\begin{example}
\label{ex:44}
Let $\F$ be the family of all embedded graphs $G$ with rings such that
every vertex of $G$ that does not belong to a ring has degree at least four,
every face of $G$ of size \xx{three is incident with an edge of $\R$, and every face of} 
 size exactly four  is incident with either a vertex of degree at least five,
or a vertex \qq{of degree at least one}  that belongs to a ring. 
Then $\F$ is hyperbolic \xx{with Cheeger constant~$19$} by  Lemma~\ref{lem:44}.\REM{Deleted:
\zz{using the argument of the proof of Theorem~\ref{thm:5listcritishyperbolic}.}}
\cc{(The Cheeger constant is indeed $19$ and not $20$, because we count vertices in the {\em open}
disk bounded by $\cc{\xi}$.)}
\end{example}

\begin{lemma}
\label{lem:g4l4}
The family  of
all  embedded graphs $G$ with rings $\R$ such that every triangle in $G$ is not null-homotopic and
$G$ is $\R$-critical with respect to some $4$-list assignment \zz{$L$}
 is hyperbolic \xx{with Cheeger constant~$19$}.
\end{lemma}

\begin{proof}
\zz{%
This follows  by \xx{a similar} argument as Theorem~\ref{thm:5listcritishyperbolic}.
Let $G_1$ and  $C_1$ be as in the proof of that theorem, \xx{except that we do not introduce new 
vertices; instead we make adjacent  every pair of vertices of $X$ that are consecutive on the boundary of $\Delta$}.
 Then the graph $G_1$
with one  ring $C_1$ embedded in the disk belongs to the family of Example~\ref{ex:44}.
To see that let first $v\in V(G_1)-V(C_1)$. If $v$  has degree at most three,
then every $L$-coloring of $G\setminus v$ (and one exists by the $\R$-criticality of $G$) can be extended
to an $L$-coloring of $G$, a contradiction. Thus $v$ has degree at least four.
Now let $f$ be a face of $G_1$ of size four, let $\qq{Y}$ be the set of (four) vertices incident with $f$,
and let us assume  that every vertex in $\qq{Y}$ has degree four and does not belong to $C_1$.
The subgraph of $G$ induced by $\qq{Y}$ is a cycle, \qq{because \cc{this} subgraph is contained in $\Delta$ and}
no triangle in $G$ is null-homotopic.
Thus every $L$-coloring of $G\setminus \qq{Y}$ (and one exists by the $\R$-criticality of $G$) can be extended
to an $L$-coloring of $G$, a contradiction. This proves our claim that  $G_1$
 belongs to the family of Example~\ref{ex:44}, and hence $G$ satisfies the hyperbolicity 
condition with Cheeger constant $19$, as desired.
}
%
%
\end{proof}

\begin{lemma}
\label{lem:g5hyper}
Let $\F$ be the family of embedded graphs $G$ with rings $\R$ such that
every cycle in $G$ of length at most four is not null-homotopic and $G$ is
$\R$-critical with respect to some $3$-list  assignment.
Then $\F$ is hyperbolic \xx{with Cheeger constant $37$}.
\end{lemma}

\begin{proof}
This follows from Theorem~\ref{DvoKawCycle} similarly as the proof of Theorem~\ref{thm:5listcritishyperbolic},
\xx{the difference being that if the vertices $u,v$ are not adjacent we need to introduce two new vertices
forming a path of length three with ends $u$ and $v$}.
\end{proof}

A $k$-coloring $\phi$ of a graph $G$ is {\em acyclic} if every two color classes induce
an acyclic graph.
Borodin~\cite{Borodin} proved that every planar graph has an acyclic $5$-coloring.

\begin{problem}
Let $k\ge2$ be an integer, and let $\F_k$ be the class of all embedded graphs $G$ with
no rings such that $G$ does not admit an acyclic $k$-coloring, but every proper subgraph
of $G$ does.
\cc{Is $\F_k$  hyperbolic for  $k\ge5$?}
\end{problem}

Our next example is related to the following conjecture of Steinberg from 
1976~\cite[Problem~2.9]{JenTof}.

\begin{conjecture}
Every planar graph with no cycles of length four or five is $3$-colorable.
\end{conjecture}

While Steinberg's conjecture \xx{was recently disproved by
Cohen-Addad, Hebdige, Kral, Li and Salgado~\cite{CohHebKraLiSal}},
Borodin, Glebov, Montassier and 
Raspaud~\cite{BorGleMonRas}
showed that every planar graph with no cycles of length four, five, six or seven
is $3$-colorable. That leads us to the following example.

\begin{example}
Let $k\ge5$ be an integer, and let $\F_k$ be the class of all embedded graphs $G$ with
rings $\R$ such that every cycle of length at least four and at most $k$ is a ring, and
$G$ is $\R$-critical with respect to the list assignment $L=(L(v):v\in V(G))$,
where $L(v)=\{1,2,3\}$ for every $v\in V(G)$.
Yerger~\cite{YerPhD} proved that $\F_k$ is hyperbolic for all $k\ge10$.
It seems plausible that the techniques developed by Borodin, Glebov, Montassier and 
Raspaud~\cite{BorGleMonRas} can be
used to extend this argument to $k\ge7$, but the \xx{case $k=6$ is} wide open.
\end{example}

\begin{example}
Let $G$ be a graph, let $L=(L(v):v\in V(G))$ be a list assignment such that
$|L(v)|\ge5$ for every $v\in V(G)$, and let $c\ge1$ be  an integer.
By a {\em $c$-$L$-coloring} of $G$ we mean a mapping $\phi:V(G)\to\bigcup L(v)$
such that $\phi(v)\in L(v)$ for every $v\in V(G)$, and for every $x$, every
component of the subgraph of $G$ induced by vertices $v$ with $\phi(v)=x$
has at most $c$ vertices. Thus $\phi$ is a $1$-$L$-coloring if and only if it
is an $L$-coloring.

Let $\F_c$ be  the class of all embedded graphs $G$ with rings $\R$ such
that for every proper subgraph $G'$ of $G$ that contains all the rings there
exists a $c$-$L$-coloring $\phi$ of the rings such that $\phi$ extends
to a $c$-$L$-coloring of $G'$ but not to a $c$-$L$-coloring of $G$.
The proof of Theorem~\ref{LinearCycle0} can be modified to show that
$\F_c$ is hyperbolic for every integer  $c\ge1$.   
\end{example}

\subsection{Examples pertaining to exponentially many colorings}

\begin{definition}
\label{def:eacrit}
Let $\epsilon\ge0$ and $\alpha\ge0$. 
Let $G$ be a graph with rings $\R$ 
embedded in a surface $\Sigma$ of Euler genus $g$, let $R$ be the total number of ring  vertices,
and let $L$ be a list assignment for $G$.
We say that $G$ is \emph{$(\epsilon,\alpha)$-exponentially-critical with respect to $L$}
if \xx{$G\ne\bigcup\R$} and for every proper subgraph $G'$ of $G$ that includes all the rings there exists
an $L$-coloring $\phi$ of $\bigcup\R$ such that there exist at least
 $2^{\epsilon (|V(G')|-\alpha(g+R))}$ distinct $L$-colorings of $G'$ extending $\phi$,
but there do not exist at least
$2^{\epsilon (|V(G)|-\alpha(g+R))}$ distinct $L$-colorings of $G$ extending $\phi$.
\end{definition}


\begin{theorem}\label{ExpCritDisc}
Let $0\le\epsilon \le 1/18$ and $\alpha\ge 0$. 
Then the family of embedded graphs with rings that are $(\epsilon,\alpha)$-exponentially-critical 
with respect to some $5$-list assignment 
is hyperbolic with Cheeger constant $\zz{57}$ (independent of $\epsilon$ and $\alpha$).
\end{theorem}

\begin{proof}
Let $G$ be a  graph with rings $\R$
embedded in a surface $\Sigma$ of Euler genus $g$
such that $G$ is $(\epsilon,\alpha)$-exponentially-critical with respect to  
a $5$-list assignment $L$, let $R$ be the total number of ring vertices, and
let $\cc{\xi}:{\mathbb S}^1\to\Sigma$ be a closed curve that bounds
an open disk $\Delta$ and intersects $G$ only in vertices.
To avoid notational complications we will assume that $\cc{\xi}$ is a simple
curve; otherwise we split vertices that $\cc{\xi}$ visits more than once to
reduce to this case.
We may assume that $\Delta$ includes at least one vertex of $G$, for 
otherwise there is nothing to show.
Let $X$ be the set of vertices of $G$ intersected by $\cc{\xi}$.
Then $|X|\ge4$ by Corollary~\ref{ExpFourCycle}, since $\epsilon\le1/9$.

Let $G_0, G_1, C_1$ and $L_1$ be defined as in the proof of Theorem~\ref{thm:5listcritishyperbolic}.
We may assume for a contradiction that $\zz{|V(G_0)-X|>57}(|X|-1)$.
Let $G_2$ be the smallest subgraph of  $G_1$ such that
$G_2$ includes $C_1$  as a subgraph and every $L_1$-coloring of $C_1$ that extends to an $L_1$-coloring
of $G_2$ also extends to an $L_1$-coloring of $G_1$.
Then $G_2$ is $C_1$-critical with respect to $L_1$, and hence 
$$|V(G_2)-V(C_1)|/\zz{19}\le |C_1|-4-\sum_{f\in\F(G_2)}(|f|-3)\le|X|-4$$
by Theorem~\ref{StrongLinear2}.
(Let us recall that $\F(G)$ denotes the set of internal faces of $G$.)
%
%
%
Let $G_0'=G\setminus (V(G_1)- V(G_2))$. Thus $G_0'$ is a proper subgraph of $G$, and, in fact,
\zz{%
$$|V(G)-V( G_0')|= |V(G_0)-V(G_0')|=|V(G_0)-X|-|V(G_2)-V(C_1)|  \ge38(|X|-1).$$
}%
As $G$ is $(\epsilon,\alpha)$-exponentially-critical with respect to $L$, there exists an
$L$-coloring $\phi$ of $\bigcup\R$ such that $\phi$ extends to at least
 $2^{\epsilon (|V(G_0')|-\alpha(g+R))}$
distinct $L$-colorings of $G_0'$,
but does not extend to at least $2^{\epsilon (|V(G)|-\alpha(g+R))}$ distinct $L$-colorings of $G$.

Let $\phi'$ be an $L$-coloring of $G_0'$ that extends $\phi$, 
let $f\in \F(G_2)$,
let $G_f$ be the subgraph of $G_1$ drawn in the closure of $f$, and let $C_f$ be the cycle bounding $f$.
Let $l_f$ be defined as $|V(C_f)|$ minus the number of vertices of $V(G_1)-V(G_0)$ incident with $f$.
By the definition of $G_2$, $\phi'$ extends to an $L$-coloring of $G_f$. 

We claim that $\phi'$ extends to at least 
$2^{( |V(G_f)-V(C_f)|- \zz{19}(l_f-3) )/9}$ distinct $L$-colorings of $G_f$. 
If $f$ is incident with no  vertex of $V(G_1)-V(G_0)$, then this follows from  Lemma~\ref{ExpManyDisc},
and otherwise we argue as follows.
If $|f|=4$, then we apply Corollary~\ref{ExpFourCycle},
and otherwise we repeat the following construction.
Let $v\in V(G_1)-V(G_0)$ be incident with $f$, and let $u_1,u_2$ be  its two neighbors.
We delete $v$ and either add an edge joining $u_1$ and $u_2$ (if $\phi'(u_1)\ne\phi'(u_2)$), 
or identify $u_1$ and $u_2$ (if $\phi'(u_1)=\phi'(u_2)$).
We repeat this construction for every $v\in V(G_1)-V(G_0)$  incident with $f$,
and apply Lemma~\ref{ExpManyDisc} to the resulting graph.
This proves our  claim that $\phi'$ extends to at least 
$2^{( |V(G_f)-V(C_f)|- \zz{19}(l_f-3) )/9}$ distinct $L$-colorings of $G_f$. 

Let $E$ be the number of extensions of $\phi'$ to $G$. 
We have  $\sum_{f\in \F(G_2)}|V(G_f)-V( C_f)| = |V(G)- V(G_0')|$ 
and, by Theorem~\ref{StrongLinear2}
\begin{align*}
\sum_{f\in \F(G_2)} (l_f-3) &=\sum_{f\in \F(G_2)} (|C_f|-3)-|V(G_1)-V(G_0)|\\
&\le|C_1|-4-|V(G_1)-V(G_0)|= |X|-4,
\end{align*}
 and hence
\begin{align*}
9\log_2 E&\ge \sum_{f\in \F(G_2)} (|V(G_f)-V(C_f)|-\zz{19}(l_f-3))\ge |V(G)-V( G_0')|-\zz{19}(|X|-4)\\
&\ge|V(G)-V( G_0')|/2,
\end{align*}
where the first inequality uses the last claim of the previous paragraph, and the last inequality uses
the inequality $|V(G)-V( G_0')|\ge \zz{38}(|X|-1)$, which we established earlier.

\zz{%
The coloring $\phi$ extends to at least $2^{\epsilon (|V(G_0')|-\alpha(g+R))}$
distinct $L$-colorings of $G_0'$, and each such extension extends to at least 
$2^{|V(G)-V(G_0')|/18}$ distinct $L$-colorings of $G$.%
}
But then as $\epsilon\le 1/18$, there exist at least $2^{\epsilon (|V(G)|-\alpha(g+R))}$ 
distinct $L$-colorings of $G$ extending $\phi$, a contradiction.
\end{proof}

The following is shown in~\cite[Theorem~\cc{5.10}]{PostleGirth5}.

\begin{lemma}
\label{lem:g5exphyper}
Let $\epsilon>0$ and $\alpha\ge0$, and 
let $\F$ be the family of embedded graphs $G$ with rings   such that 
every cycle of length four or less is 
\xx{equal to a ring}
and  $G$  is $(\epsilon,\alpha)$-exponentially-critical with respect to some $3$-list assignment.
If $\epsilon<1/20000$, then the family $\F$ is hyperbolic
with Cheeger constant  independent of $\epsilon$ and $\alpha$. 
\end{lemma}

The following \cc{lemma follows from the work of} \xx{Kelly and  Postle~\cite{KelPos}}.

\begin{lemma}
\label{ex:g4exphyper}
Let $1/8\ge\epsilon\ge0$ and $\alpha\ge0$, and 
let $\F$ be the family of embedded graphs $G$ with rings   such that 
every triangle is 
\xx{equal to a ring}
and  $G$  is $(\epsilon,\alpha)$-exponentially-critical with respect to some $4$-list assignment.
Then the family $\F$ is hyperbolic
with Cheeger constant \cc{$67.5$}.
\end{lemma}

\cc{%
\begin{proof}
Let $\F$ be as stated, let $G\in\F$, and let $\cc{\xi} :{\mathbb S}^1\to\Sigma$ be a closed curve that 
bounds an open disk $\Delta$ and intersects $G$ only in vertices and $\Delta$ includes a vertex of $G$.
Then by Theorem~\ref{thm:extend4cycle} and~\cite[Theorem~7]{KelPos} the number of times, $N$, 
the curve $\cc{\xi}$ intersects $G$, counting multiplicities, satisfies $N\ge3$.
By~\cite[Theorem~8]{KelPos} the number of vertices in $\Delta$ plus $N$ is at most $46N$.
Thus the number of vertices in $\Delta $ is at most $45N\le 67.5(N-1)$, as desired.
\end{proof}
}



\section{The Structure of Hyperbolic Families}
\label{sec:structure}

\cc{%
\begin{definition}
A graph $G$ with rings $\cal R$ embedded in a surface $\Sigma$ is 
\emph{$2$-cell embedded in
$\Sigma$} if every face of $G$ is simply connected.
\end{definition}
}

In this section we investigate the structure of  hyperbolic families of embedded graphs with rings.
The main result is Theorem~\ref{thm:sleevedec1}.
\zz{%
We progress to the main result as follows. In Subsection~\ref{sec:frames}, we define ``frames", which are essentially a 
subgraph of an embedded graph with only one face. 
We also show that every $2$-cell-embedded graph has a frame (Lemma~\ref{lem:frameexists}) and that minimal 
frames satisfy certain geodesic properties (Lemma~\ref{lem:optframe}). 
In Subsection~\ref{sec:sepgrow} we show (Lemma~\ref{lem:logdist}) that for hyperbolic families the vertices in the 
interior of a disk have ``logarithmic distance" to the boundary of the disk (an improvement over the trivial bound of 
linear distance as guaranteed by the definition of hyperbolic). 
Conversely, then, the neighborhood of a vertex contained in a disk experiences ``exponential growth" in its diameter 
(Lemma~\ref{lem:expgrowth}).
}

\zz{%
Combining these ideas in Subsection~\ref{sec:ew}, we prove the key Theorem~\ref{thm:ewconst},
 which says that large hyperbolic graphs have a short (i.e., as a function of the Cheeger constant) non-null homotopic cycle. 
Theorem~\ref{thm:ewconst} is the key to proving the main result, Theorem~\ref{thm:sleevedec1}. 
The intuition for the  proof \xx{of Theorem~\ref{thm:ewconst}} is that if no such cycle exists, then balls around the midpoints of the segments of the frame 
have exponential growth in the lengths of the segments, while by hyperbolicity, the whole graph is linear in the size of 
the frame which is equal to the sum of the lengths of its segments. 
This linear upped bound versus the exponential growth then provides a contradiction.
}

\zz{%
Finally in Subsection~\ref{subsec:structure}, we prove the main result (Theorem~\ref{thm:sleevedec1}), 
which provides a structural decomposition for hyperbolic families by asserting that all but $O(g)$ vertices of a graph 
in such a family belong to cylinders of small edge-width (which we call ``sleeves"). The proof of 
Theorem~\ref{thm:sleevedec1} proceeds inductively by cutting along the short cycles guaranteed by Theorem~\ref{thm:ewconst}.
}

\subsection{Frames}
\label{sec:frames}

\begin{definition}
Let $G$ be a graph with rings $\R$ embedded in a surface $\Sigma$ and 
let $F$ be a subgraph of $G$ that includes every ring of $G$ as a subgraph.
Thus $F$ may be regarded as a graph with rings $\R$.
We say that $F$ is a \emph{frame} of $G$ if for every component $\Sigma_0$
of $\Sigma$ the subgraph $F_0$ of $F$ embedded in $\Sigma_0$ has
exactly one face, this unique face is  simply connected,
and every vertex of $F$ of degree at most one in $F$ belongs to a ring in $\R$.
\end{definition}

\begin{lemma}
\label{lem:frameexists}
Let $G$ be a graph with rings such that $G$ is $2$-cell embedded 
in a surface $\Sigma$. Then $G$ has a frame.
\end{lemma}

\begin{proof}
By applying the forthcoming argument to every component of $\Sigma$
we may assume that $\Sigma$ is connected.
For a connected surface we proceed by induction on 
\zz{three times}
the number of faces of $G$
plus the number of vertices of degree one that belong to no ring. 
Since $\Sigma$ is connected and $G$ is $2$-cell embedded, it follows that
$G$ is connected.
If $G$ has one face and no vertex of degree one that belongs to no ring,
then $G$ is a frame of itself.
If $v$ is a vertex of $G$ of degree one that belongs to no ring,
then $G\backslash v$ is a graph with the same set of rings that
is $2$-cell embedded in $\Sigma$.
By induction, $G\setminus v$ has a frame which is then a frame of $G$.
So we may suppose that $G$ has at least two faces. 
But then there is an edge $e$ of $G$ such that 
the two faces incident with $e$ are distinct. 
It follows that $e$ belongs to no ring.
Thus $G\setminus e$ is $2$-cell embedded in $\Sigma$ with the same set
of rings and it has fewer faces than $G$,
and at most two more vertices of degree one than $G$. 
By induction, $G\setminus e$ has a frame which is then a frame of $G$.
\end{proof}

\begin{definition}
Let $F$ be a graph with rings.
We say that a vertex $v\in V(F)$ is \emph{smooth} if it has degree two in $F$
and belongs to no ring.
By a \emph{segment} of $F$ we mean either a cycle in $F$ such that every 
vertex of the cycle except possibly one is smooth, or a path $P$ in $F$ such that
every internal vertex of $P$ is smooth, the ends of $P$ are not smooth,
and if $P$ has no edges, then its unique
vertex is an isolated vertex of $F$.
It follows that every graph with rings is an edge-disjoint union of its segments.
An \emph{internal vertex} of a segment $S$ of $F$ is a vertex of
$S$ that is smooth.
Thus if $S$ is a path, then every vertex of $S$ other than its ends is
internal. 
It follows that if $S$ is a segment, then $E(S)$ either is a subset of 
the edge-set of a ring, or is disjoint from the edge-sets of all rings. 
In the latter case we say that $S$ is a \emph{non-ring segment}.
\end{definition}

Our next lemma
shows that the number of non-ring segments of a frame is bounded by a function of
the Euler genus and the number of rings.

\begin{lemma}
\label{lem:numbersegments}
Let $G$ be a graph with $r$ rings such that $G$ is $2$-cell embedded 
in a surface $\Sigma$ of Euler genus $g$.
Let us assume that no component of $\Sigma$ is the sphere (without boundary) and that $\Sigma$ has  $h$ components. 
If $F$ is a frame of $G$, then the number of non-ring segments of $F$ is at most 
$3g+2r-2h$.
\end{lemma}

\begin{proof}
It suffices to prove the lemma for $h=1$. 
\zz{Thus $F$ is connected.}
We contract each ring of $F$ to a vertex, called a \emph{new vertex},
and consider the natural embedding
of the resulting \aa{multi}graph $H'$ in the surface $\widehat\Sigma$ obtained from $\Sigma$ by capping
off each boundary component with a disk.
Since $\Sigma$ is not the sphere it follows that $H'$ has at least one vertex.
If $H'$ has exactly one vertex \aa{and no edges}, then $\Sigma$ is the disk, $r=1$, there are no non-ring
segments, and hence the lemma holds.
If $H'$ is a cycle containing no ring vertex, then $\Sigma$ is the projective plane,
$r=0$, there is exactly one non-ring segment, and, again, the lemma holds.

Thus we may assume that $H'$ has at least two vertices, and that it is not a cycle
containg no ring vertex.
Let $H$ be the \aa{multi}graph obtained from $H'$ by suppressing all vertices of degree two that 
are not new. Then the number of edges of $H$ is equal to the number of non-ring
segments of $F$.
By Euler's formula, $|V(H)|=|E(H)|+1-g$, because $H$ has exactly one face.  
We have $2|E(H)|=\sum_{v\in V(H)} \deg_H(v)\ge 3|V(H)|-2r$, because every vertex of $H$
that is not new has degree at least three, and every new vertex has degree at least one.
Thus $|V(H)|\le 2g+2r-2$, and hence $|E(H)|\le 3g+2r-3$, as desired.
\end{proof}

\begin{definition}
Let $G$ be a graph with rings embedded in a surface $\Sigma$,
and let $F$ be a frame of $G$.
We say that the frame $F$ is \emph{optimal} if the following two
conditions are satisfied for every segment $S$ of $F$:
\begin{enumerate}
\item[(O1)] For every two distinct vertices $x,y\in V(S)$
and every path $P$ in $G$ with ends $x$ and $y$ and otherwise disjoint
from $F$, if $Q$ is a subpath of $S$ with ends $x$ and $y$ and
every internal vertex smooth, then $|V(P)|\ge|V(Q)|$.
\item[(O2)] For every internal vertex $x$ of $S$, every vertex 
$y\in V(F)-V(S)$ and every path $P$ in $G$ with ends $x$ and $y$ and 
otherwise disjoint from $F$, if $Q_1,Q_2$ are the two subpaths of $S$ 
with one end $x$ and 
and the other end not smooth,
then $|V(P)|\ge\min\{|V(Q_1)|,|V(Q_2)|\}$.
\end{enumerate}
Let us remark that if $S$ is a cycle, then in condition (O1) there may be
two choices for the path $Q$; otherwise $Q$ is unique.
In condition (O2) the existence of the path $P$ implies that $y$ and
$S$ belong to the same component of $F$, and hence $S$ includes at least
one vertex that is not smooth; therefore, the paths $Q_1,Q_2$ exist.
\end{definition}

\begin{lemma}
\label{lem:optframe}
Let $G$ be a graph with rings that is $2$-cell embedded 
in a surface $\Sigma$. If $F$ is a frame of $G$ with $|E(F)|$ minimum,
then $F$ is an optimal frame.
\end{lemma}

\begin{proof}
Let $F$ be as stated.
To prove that it satisfies (O1) let $S,x,y,P$ and $Q$ be as in that condition.
Then $|V(Q)| \le |V(P)|$, for otherwise replacing $P$ by $Q$ produces
a frame with fewer edges than $F$, a contradiction.

To prove that $F$ satisfies (O2) let $S,x,y,P,Q_1$ and $Q_2$ be as in (O2).
Now $P$ divides a simply connected face $f$ of $F$ into two faces $f_1$ and
$f_2$. When tracing the boundary of the face $f$ we encounter the
segment $S$ twice. 
We may assume that $f_1,f_2$ are numbered in such a way that when tracing
the boundary of $f_i$ starting on $P$ and moving toward $x$, we encounter
$Q_i$ next.
It follows that there is an index $i\in\{1,2\}$ such that when tracing 
the boundary of $f_i$ we encounter $Q_i$ twice and $Q_{3-i}$ once.
We deduce that by replacing $Q_{3-i}$ by $P$ we obtain a frame.
\aa{%
The minimality of $F$ implies that $|V(P)|\ge|V(Q_{3-i})|$, as desired.
}%
\end{proof}

\subsection{Separations and growth rates}
\label{sec:sepgrow}

\begin{definition}
A \emph{separation} of a graph $G$ is a pair $(A,B)$ such that 
$A\cup B=V(G)$ and no edge of $G$ has one end in $A-B$ and the other in
$B-A$.
If $G$ is a graph with rings that is embedded in a surface $\Sigma$, then
we say that a separation $(A,B)$ of $G$ is \emph{flat} if 
there exists a closed disk $\Delta\subseteq\Sigma$ such that $\Delta$
includes $A$ and 
\zz{%
the interior of $\Delta$ includes $A-B$ and every edge of $G$ incident with a vertex of $A-B$
(regarded as a point set not including its ends).
It follows that  no ring vertex belongs to  the interior of $\Delta$, and hence
$B$ includes all ring vertices.
}%
\end{definition}

If $G$ belongs to a hyperbolic family, then\zz{, as we show next,}
the same \zz{linear} bound holds for
closed curves that bound ``pinched" disks, or, equivalently, it
holds for \zz{flat} separations.
\xx{But first we need the following  lemma, which says that a flat separation 
is determined by the faces of a certain multigraph.}

\xx{%
\begin{lemma}
\label{lem:sep2H}
Let $G$ be a connected  graph with rings $\R$ embedded in a surface  $\Sigma$,
 let $(A,B)$ be a flat  separation of $G$,
and let $\Delta$ be as in the definition of flat separation. 
Then there exists a \aa{multi}graph $H$  such that
\begin{itemize}
\item $V(H)\subseteq A\cap B$,
\item $H$ is embedded in $\Delta$,
\item the point set of $H$ intersects the point set of $G$ in $V(H)$, 
\item every vertex of $H$ has positive even degree, 
\item every face of $H$ includes an edge of $G$, and 
\item all edges  of $G$ that belong to the same face
of $H$ either are all incident with a (possibly different) vertex of  $A-B$ or all have both ends in  $B$.
\end{itemize}
\end{lemma}
}

\begin{proof}
\xx{%
We proceed by induction on $|A\cap B|$.
If a vertex $v\in A\cap B$ has no neighbor in $A-B$, then we can remove $v$ from $A$
and apply induction. Similarly, if a non-ring  vertex $v\in A\cap B$ has no neighbor in $B$,
then we can remove $v$ from $B$ and apply induction. 
We may therefore assume that every vertex in $A\cap B$ has a neighbor in $A-B$,
and if it does not belong to a ring, then it also has a neighbor in $B$.
\qq{If $A-B=\emptyset$ or $B=\emptyset$, then the null graph satisfies the conclusion of the lemma.
We may therefore assume that $A-B\ne\emptyset$ and  $B\ne\emptyset$.
In particular, $A\cap B\ne\emptyset$.}
}


\xx{%
We  construct the graph $H$ in two steps. In the first step we embed, for every vertex 
$v\in A\cap B$, a set of ``half-edges" incident with $v$.
In the second step we partition the set of all half-edges into pairs, and we merge each 
pair to form an edge of  $H$.
 To begin the first step  let $v\in A\cap B$, and let $e,f\in E(G)$ be
incident with $v$, consecutive in the cyclic ordering of edges incident with $v$ determined
by the embedding of $G$ in $\Sigma$, and such that  $e$ has its other end in $A-B$
and $f$ has its other end in $B$.
For every such pair $e,f$ we insert a half-edge incident with $v$ into the ``angle" between
$e$ and $f$ in such a way that the half-edge will be disjoint from $G$.
We also insert a half-edge into the ``angle" between $e$ and $\Gamma$, if $\Gamma$
is a component of the boundary of $\Sigma$,  $v\in\Gamma$, the other end of $e$ belongs 
to  $A-B$, and no other edge of  $G$ is a subset of the ``angle" between $e$ and $\Gamma$.
This process results in the insertion of a positive even number of half-edges incident with $v$.
In the second step we look at an arbitrary face of $G$ and one of its  boundary walks $W$,
where a boundary walk is a sequence of vertices and edges or subsets of bd$(\Sigma)$ with 
the usual properties.
Let $Y$ be the set of half-edges inserted into \cc{the face} that are incident with a vertex of $W$.
The elements of $Y$ divide $W$ into subwalks, which we shall refer to as {\em segments}.
For every segment $S$ either
\begin{itemize}
\item
$S$ has an internal vertex and 
 all its internal vertices belong to $A-B$, or
\item
all internal vertices of $S$ belong to $\cc{B}$,
\end{itemize}
and these types of segment alternate along $W$.
\cc{%
(To see that consider a  vertex $v\in A\cap B\cap V(W)$ such that we inserted an element of $Y$
incident with $v$.
It follows that $v$ is incident with an edge $e\in E(W)$ whose other end belongs to $A-B$,
and either an edge $f\in E(W)$ whose other end belongs to $B$ or a subset of the boundary of $\Sigma$.
Thus $e$ belongs to a segment of the first kind and $f$ or the subset of the boundary of $\Sigma$
belongs to a segment of the second kind. Since changes between segments only occur at vertices 
$v$ as above,  the statement follows.)}
It follows that $|Y|$ is even.
For every segment $S$ delimited by half-edges $h_1,h_2$ and whose internal vertices belong to 
$A-B$ we merge $h_1$ and $h_2$ into an edge of $H$, embedding the edge in the interior of $\Delta$ 
in a close proximity to~$S$.
This completes the  construction of $H$.
It follows from the construction that $H$ has the desired properties.
(For the last condition to hold we need that $G$ is connected.)
}
\end{proof}

\xx{%
We are now ready to prove that if $G$ belongs to a hyperbolic family, then
the same linear bound holds for
closed curves that bound ``pinched" disks, or, equivalently, it
holds for flat separations.
}

\begin{lemma}
\label{lem:linearsep}
Let $\F$ be a hyperbolic family of embedded graphs with rings, let $c$ be a
Cheeger constant for $\F$, let $G\in\F$ and let $(A,B)$ be a flat separation
of $G$. If $A-B\ne\emptyset$, then $|A-B|\le c(|A\cap B|-1)$.
\end{lemma}
\begin{proof}
We proceed by induction on $|A|+|A\cap B|$.
Let  $(A,B)$ be a flat separation of $G$ with $A-B\ne\emptyset$, and let $\Delta$ be
as in the definition of flat separation.%
\REM{Deleted: If a vertex $v\in A\cap B$ has no neighbor in $A\zz{-B}$, then we can remove $v$ from $A$
and apply induction, and similarly if  $v$ has no neighbor in $B\zz{-A}$.
We may therefore assume that every vertex in $A\cap B$ has a neighbor in both $A\zz{-B}$ and $B\zz{-A}$.}

\xx{Let $H$ be the multigraph as in Lemma~\ref{lem:sep2H}.
If $H$ is null, then $A=V(G)$, and  every edge of $G$ is incident with a vertex of $A-B$,
and hence $G$ is contained in $\Delta$, \cc{and the lemma follows from} Lemma~\ref{lem:onering}.
It follows that $H$ is not null.}
Every  cycle $C$ of $H$ bounds a closed disk $\Delta(C)\subseteq\Delta$.
Let us say that a cycle $C$ of $H$ is {\em maximal} if $C$ is not a subset of $\Delta (C')$ for a cycle $C'\ne C$ of $H$.
It follows that every vertex of $A\zz{-B}$ is contained in $\Delta(C)$ for some maximal cycle $C$ of $H$.
The maximal cycles form a tree structure; in particular, there exist a maximal cycle $C$ and a vertex $v\in A\cap B$
such that $C\setminus v$ is disjoint from all other maximal cycles.
Let $X$ be the set of all vertices of $G$ contained in the interior of $\Delta(C)$.
From the definition of hyperbolicity applied to a closed curve tracing $C$ we deduce that $|X|\le c(|V(C)|-1)$.
If $C$ is the unique maximal cycle, then 
$A-B\subseteq X$, and hence $|A-B|\le |X|\le c(|V(C)|-1)\le  c(|A\cap B|-1)$, as desired.
 We may therefore assume that $C$ is not the unique maximal cycle. 
Let $B':=B\cup X$ and let 
\zz{%
$A':=A-X-V(C)$ if $C$ is disjoint from all other maximal cycles, and $A':=A-X-(V(C)-\{v\})$ if it is not.
}%
By induction applied to the separation $(A',B')$ we conclude that 
$$|A-B|\le|A'-B'|+|X|\le c(|A'\cap B'|-1)+c(|V(C)|-1)\le c(|A\cap B|-1),$$
as desired.
\end{proof}

\zz{%
The next lemma says that every vertex on the ``flat" side of a flat separation 
is at most logarithmic distance away from the middle part of the separation.
}

\begin{lemma}
\label{lem:logdist}
Let $\F$ be a hyperbolic family of embedded graphs with rings, let $c$ be a
Cheeger constant for $\F$, let $G\in\F$ and let $(A,B)$ be a flat separation
of $G$. 
Then every vertex of $A-B$ is at distance at most $(2c+1)\log_2|A\cap B|$
from $A\cap B$.
\end{lemma}
\begin{proof}
We proceed by induction on $|A\cap B|$. If $|A\cap B|\le1$, 
then $A\subseteq B$ by Lemma~\ref{lem:linearsep}, and the lemma vacuously
holds.
Thus we may assume that $|A\cap B|\ge2$, and that the lemma holds for
separations $(A',B')$ with $|A'\cap B'|< |A\cap B|$.
Let $v$ be a vertex of $A-B$. 
For $i=0,1,\ldots$ let $C_i$ be the set of all vertices of $A$ at
distance exactly $i$ from $A\cap B$.
For $i=1,2,\ldots$ let $A_i := A\setminus \bigcup_{j=0}^{i-1} C_j$ and 
$B_i := B\cup \bigcup_{j=0}^{i} C_j$. 
Then $(A_i,B_i)$ is a flat separation and $A_i\cap B_i = C_i$. 
There exists an integer $i$ such that $1\le i\le 2c+1$ and 
$|C_i|\le |A\cap B|/2$,
for otherwise $|A-B|\ge|\bigcup_{1\le i\le 2c+1} C_i|\ge
\lfloor2c+1\rfloor|A\cap B|/2\ge c|A\cap B|$, contrary to 
Lemma~\ref{lem:linearsep}.
\aa{%
We may assume that $v\in A_i-B_i$, for otherwise $v$ satisfies the conclusion of the lemma.
}%
By induction
\aa{applied to the separation $(A_i,B_i)$ the vertex} $v$ is at distance at most 
$(2c+1)\log_2|C_i| \le (2c+1)\log_2 (|A\cap B|/2) = (2c+1)(\log_2|A\cap B|-1)$
from $C_i$. But $C_i$ is at distance at most $2c+1$ from $|A\cap B|$, 
and hence $v$ is at distance at most $(2c+1)\log_2|A\cap B|$
from $A\cap B$, as desired.
\end{proof}

\begin{corollary}
\label{cor:ringdist}
\aa{%
Let $\F$ be a hyperbolic family of embedded graphs with rings, let $c$ be a
Cheeger constant for $\F$, and let $G\in\F$
}%
be a graph embedded in the disk with one ring $R$. 
Then every vertex of $G$ is at distance at most $(2c+1)\log_2|V(R)|$
from $R$.
\end{corollary}

\begin{proof}
This follows from Lemma~\ref{lem:logdist} applied to the separation
$(V(G),V(R))$.
\end{proof}

\xx{%
Next we will show that planar
neighborhoods in a hyperbolic family exhibit exponential growth, but first we need the
following lemma.
 The lemma formalizes the fact that
if a vertex $v$ is not near a short non-null-homotopic cycle, then the vertices up to a certain
distance from $v$ form a planar graph, indeed they lie inside a disk.
\begin{lemma}
\label{lem:disk}
Let $k\ge1$ be an integer, let   $G$ be  a graph with rings  embedded in a surface $\Sigma$, and let $v\in V(G)$.
If $v$ is at distance at least $k$ from every ring and there exists 
no non-null-homotopic cycle $C$ in $G$  of length at most $2k$ such that 
every vertex of $C$ is at distance at most $k$ from $v$,
then there exists a closed disk $\Delta\subseteq\Sigma$ such that 
\begin{itemize}
\item every vertex of $G$ at distance at most $k$ from $v$  belongs to $\Delta$,
\item every vertex of $G$ at distance at most $k-1$ from $v$  belongs to the interior of $\Delta$, and
\item every edge of $G$ incident with a vertex at distance at most $k-1$ from $v$  belongs to the interior of $\Delta$.
\end{itemize}
Furthermore, if there exists 
no non-null-homotopic cycle $C$ in $G$  of length at most $2k+1$ such that 
every vertex of $C$ is at distance at most $k$ from $v$,
then the closed disk $\Delta\subseteq\Sigma$ may be chosen so that
every edge of $G$ with both ends at distance at most $k$ from $v$  belongs to the interior of $\Delta$. 
\end{lemma}
\begin{proof}
Let $H$ be the subgraph of $G$ consisting of vertices at distance at most $k$ from $v$ and edges incident with a vertex at distance 
at most $k-1$ from $v$, and let $T$ be a breadth-first search spanning tree of $H$  \cc{rooted at $v$}. 
By hypothesis every fundamental cycle
of $T$ in $H$ bounds a disk. The union of these disks together with  $T$ is simply connected, and the required disk $\Delta$
can be obtained from this union by ``fattening" edges of $H$ and vertices of $H$ at distance at most $k-1$
from $v$ that belong to the boundary of this union.
This proves the first assertion, including the three bullet points. The second one follows analogously.
\end{proof}
}

\zz{%
We now show that planar neighborhoods in a hyperbolic family exhibit exponential growth.
}

\begin{lemma}
\label{lem:expgrowth}
Let $\F$ be a hyperbolic family of embedded graphs with rings, let $c$ be a
Cheeger constant for $\F$, let $G\in\F$ be embedded in a surface $\Sigma$, 
let $v\in V(G)$ and let $k\ge0$ be an integer.
\zz{%
If $v$ is at distance at least $k$ from every ring and there exists 
no non-null-homotopic cycle $C$ in $G$  of length at most $2k$ such that 
\xx{every vertex of $C$ is at distance at most $k$ from $v$,}
}%
then $G$ has at least $2^{k/(2c+1)}$ vertices at distance exactly $k$ from~$v$.
\end{lemma}

\begin{proof}
\zz{%
Let $A$ be the set of vertices of $G$ at distance at most $k$ from $v$ and let 
$B$ be the set of vertices of $G$ at distance at least $k$ from $v$.
\xx{By Lemma~\ref{lem:disk}} the assumptions of the lemma imply that there exists
a closed disk $\Delta\subseteq\Sigma$ such that 
\begin{itemize}
\item every vertex of $G$ at distance at most $k$ from $v$  belongs to $\Delta$,
\item every vertex of $G$ at distance at most $k-1$ from $v$  belongs to the interior of $\Delta$, and
\item every edge of $G$ incident with a vertex at distance at most $k-1$ from $v$  belongs to the interior of $\Delta$.
\end{itemize}
Thus $(A,B)$ is a flat separation of $G$, and hence 
 $k\le(2c+1)\log_2|A\cap B|$ by Lemma~\ref{lem:logdist}.
It follows that $|A\cap B|\ge2^{k/(2c+1)}$, as desired.
}%
\end{proof}

\subsection{Edge-width of hyperbolic families}
\label{sec:ew}

\zz{%
The next lemma says that every vertex that is at least logarithmic distance away from every ring 
is  contained in a non-null-homotopic cycle of at most logarithmic length.
}


\begin{lemma}
\label{lem:esscycle4linsize}
Let $\F$ be a hyperbolic family of embedded graphs with rings, let $c$ be a
Cheeger constant for $\F$, let $G\in\F$ be embedded in a surface $\Sigma$,
let $k$ be an integer with $k\ge (2c+1)\log_2|V(G)|$,
and let $v\in V(G)$ be a vertex of $G$ at distance
at least $k$ from every ring of $G$.
Then $G$ has a non-null-homotopic cycle $C$ of length at most
$2k$ such that 
\xx{every vertex of $C$ is at distance at most $k$ from $v$.}
\end{lemma}

\begin{proof}
Since $\cal F$ is hyperbolic, it follows that $|V(G)|\ge2$, and hence
$k\ge1$.
We may assume for a contradiction that the required cycle does not exist.
By Lemma~\ref{lem:expgrowth} there are at least $2^{k/(2c+1)} \ge |V(G)|$ 
vertices at distance exactly $k$ from $v$, a contradiction,
because those vertices do not include $v$, since $k\ge1$.
\end{proof}

Let us recall that if $\Sigma$ is a surface with boundary, then $\widehat{\Sigma}$ denotes the 
surface without boundary obtained by
gluing a disk to each component of the boundary.

If $G$ is a graph with no rings embedded in a surface, then its edge-width
is usually defined as the maximum integer $k$ such that every non-null-homotopic
cycle in $G$ has length at least $k$.
We extend this definition to graphs with rings as follows.

\begin{definition}
Let $G$ be an embedded graph with rings that is embedded in a surface $\Sigma$,
and let $k\ge0$ be 
the maximum integer (or infinity) such that
\begin{enumerate}
\item every  cycle in $G$ 
that is  non-null-homotopic in $\widehat{\Sigma}$ has at least $k$ edges that
do not belong to a ring, 
\item every non-null-homotopic (in $\Sigma$) cycle in $G$ that is \zz{disjoint} from all rings has length at least $k$, and
\item every two vertices that belong to distinct rings of $G$ are at distance
at least $k$ in $G$.
\end{enumerate}
In those circumstances we say that $G$ is embedded with \emph{edge-width} $k$.
\end{definition}

\zz{%
Our next lemma says  that if an embedded graph with rings belonging to a hyperbolic family
is embedded with moderately large edge-width (logarithmic in the number of vertices), then
it has a specific structure.
First we define the structure.
}

\begin{definition}
Let $G$ be a graph with rings $\R$ embedded in a surface $\Sigma$, and
let $\Gamma_1,\Gamma_2,\ldots,\Gamma_n$ be the components of
$\hbox{\rm bd}(\Sigma)$. 
We say that $G$ is \emph{locally cylindrical} if the components of $G$
can be numbered $G_1,G_2,\ldots,G_n$ in such a way that for all
$i=1,2,\ldots,n$ the graph $G_i$ includes the ring of $G$ that is
embedded in $\Gamma_i$, and there exists a disk 
$\Delta_i\subseteq\widehat\Sigma$ that includes every vertex and edge of $G_i$.
\end{definition}

\begin{lemma}
\label{lem:ewlinsize}
Let $\F$ be a hyperbolic family of embedded graphs with rings, let $c$ be a
Cheeger constant for $\F$, let $G\in\F$ be embedded in a surface $\Sigma$,
and let $k$ be an integer with $k\ge 2(2c+1)\log_2|V(G)|$.
If $G$ is embedded with edge-width at least $k+\xx{3}$, then $G$ is 
locally cylindrical and every vertex is at distance at most \zz{$\lceil k/2\rceil$} from some ring.
\end{lemma}

\begin{proof}
Let $v\in V(G)$. If $v$ is at distance at least 
\zz{%
$\lceil k/2\rceil+1$
}%
 from every ring,
then by Lemma~\ref{lem:esscycle4linsize} 
\zz{%
applied to the integer $\lceil k/2\rceil$
}%
there is a non-null-homotopic 
cycle in $G$ of length at most $\zz{2\lceil k/2\rceil\le k+1}$ 
\zz{%
\xx{such that every vertex of the cycle is at distance at most $\lceil k/2\rceil$ from $v$.}
Since $v$  is at distance at least $\lceil k/2\rceil+1$ from every ring,
it follows that the cycle 
}%
is disjoint from all rings,
contrary to the second condition in the definition of edge-width.
Thus every vertex of $G$ is at distance at most $\zz{\lceil k/2\rceil}$ from some ring.

Let $\Gamma_1,\Gamma_2,\ldots,\Gamma_n$ be the components of
$\hbox{\rm bd}(\Sigma)$, and for $i=1,2,\ldots,n$ let $G_i$ be the
subgraph of $G$ induced by  vertices of $G$ that are at distance 
at most $\zz{\lceil k/2\rceil}$ from the ring of $G$ that is drawn in $\Gamma_i$.
The third condition in the definition of edge-width implies that the
graphs $G_i$ are 
the components of $G$.
The first condition in the definition of edge-width implies\xx{, by the second assertion of
Lemma~\ref{lem:disk} applied to the graph obtained from $G_i$ by contracting the ring
contained in $G_i$ to a vertex $v$, surface  $\widehat\Sigma$ and integer $\lceil k/2\rceil$,}  that
each $G_i$ is contained in a closed disk contained in $\widehat\Sigma$.
Thus $G$ is locally cylindrical, as desired.
\end{proof}


Let us remark that if $G$ is locally cylindrical, then each component
of $G$ may be regarded as a graph with one ring embedded in a disk,
and if this embedded graph belongs to the hyperbolic family $\cal F$,
then its number of vertices is bounded by a linear function of the
length of the ring by Lemma~\ref{lem:onering}.


In what follows we will need the embedded  graph to be $2$-cell embedded.
If it is not, then there is an obvious way to simplify the surface, but we need
to know that the new embedded graph belongs to the same family.
Hence the following definition. 

\begin{definition}
\label{def:curvecut}
Let $\F$ be a family of embedded graphs with rings, let $G\in\F$ be embedded
in a surface  $\Sigma$ with rings $\R$, and let $\cc{\xi}:{\mathbb S}^1\to \Sigma$ 
be a non-null-homotopic simple closed curve disjoint from $G$.
Let $\Sigma'$ be obtained from $\Sigma$ by cutting $\Sigma$ open along $\cc{\xi}$
and pasting a disk or disks on the resulting one or two new boundary components.
Let $G'$ denote the graph $G$ when it is regarded as an embedded graph in $\Sigma'$
with rings $\R$.
If $G'\in\F$ for every $G\in\F$ and every $\cc{\xi}$ as above, then we say that
$\F$ is {\em closed under curve cutting}.
\end{definition}

\zz{%
In  the next  example we exhibit a hyperbolic family  such  that the family obtained from it by 
curve cutting is not  hyperbolic.
}

\begin{example}
\label{ex:curvecut}
Let $k\ge3$ be a fixed integer, and for $n\ge3$ let $G_n$ denote the  Cartesian product of
a cycle of length $k$ and a path on $n$ vertices.
 We start by considering the obvious embedding of $G_n$ in the sphere ${\mathbb S}^2$.
Thus $G_n$ has two faces bounded by cycles $C_1,C_2$ of length $k$ and all other faces are
bounded by cycles of length four.
Let $\Sigma'$ be the surface obtained from ${\mathbb S}^2$ by removing the face bounded by $C_1$.
Thus $\Sigma'$ is a disk whose boundary is the point set of $C_1$.
Let $\Sigma$ be obtained from $\Sigma'$ by
inserting a crosscap inside the face bounded by $C_2$ (or, equivalently, removing  the face
bounded by $C_2$ and gluing a M\"obius band onto $C_2$).
Thus $\Sigma$ has Euler genus $1$ and one boundary component, namely the point set of $C_1$.
Let $\F$ be the family consisting of the embedded graphs $(G_n,\{C_1\},\Sigma)$ with one ring for all $n\ge3$.
Then $\F$ is  hyperbolic, but the family obtained from $\F$ by curve cutting is not,
 because it includes the embedded graphs $(G_n,\{C_1\},\Sigma')$.
\end{example}

\zz{%
The next theorem improves upon the length of the non-null-homotopic cycle guaranteed
by Lemma~\ref{lem:esscycle4linsize} from logarithmic to constant, assuming the graph is 
not too small.
The theorem will be a key tool in proving Theorem~\ref{thm:sleevedec1},
the main result of this section.
}

\begin{theorem}
\label{thm:ewconst}
Let $\F$ be a hyperbolic family of embedded graphs with rings,
let $c$ be a Cheeger constant for $\F$, 
let $G\in\F$ have $r$ rings 
and let it 
be embedded in a surface $\Sigma$ of genus $g$ with $h$ components,
where no component of $\Sigma$ is the sphere.
Let $R$ be the total number of ring vertices.
Assume that either $G$ is $2$-cell embedded, or that $\F$ is closed
under curve cutting.
Let $d:= \lceil 3(2c+1)\log_2(8c+4) \rceil$. 
If $|V(G)|> (c+1)(2R+(32d+10)r + 4(4d+1)(3g-2h))$, 
then there exists a non-null homotopic (in $\Sigma$) cycle that is 
disjoint from all of the rings and has length at most $2d$.
\end{theorem}

\begin{proof}
We begin by observing that
\begin{align*}
4(c+1)(2d+1) &\le 4(c+1)(6(2c+1)\log_2(8c+4)+3)) \le
4(c+1)(6(2c+1)(8c+4)/2+3) \\
&\le 4(c+1)15(2c+1)^2 \le 60(2c+1)^3\le (4(2c+1))^3 \le 2^{d/(2c+1)}.
\end{align*}

We proceed by induction on $g+R$. 
If $G$ is not $2$-cell embedded, then there exists a non-null-homotopic simple closed curve
$\cc{\xi}:{\mathbb S}^1\to \Sigma$ such that the image of ${\mathbb S}^1$ under $\cc{\xi}$
is disjoint from $G$. The theorem then follows by induction applied to the same graph $G$
embedded in the surface $\Sigma'$ as in the definition of curve cutting.
Thus we may assume that $G$ is $2$-cell embedded.

By Lemmas~\ref{lem:frameexists} and~\ref{lem:optframe} there exists 
an optimal frame $F$ of $G$. 
For every non-ring segment $S$ of $F$ we choose a maximal set $X(S)$ of
vertices of $S$ such that the vertices in $X(S)$ are pairwise at distance
in $S$ of at least $2d+1$ and each vertex of $X(S)$ is at distance in $S$ of at
least $2d+1$ from each end of $S$.
Thus $|X(S)|(2d+1)\ge |E(S)|-4d-1$.
Let $X$ be the union of $X(S)$ over all non-ring segments $S$ of $F$.
Thus $|X|(2d+1)\ge l-(4d+1)s$,
where $l$ is the sum of the lengths of all non-ring segments of $F$
and $s$ is their number.

For $x\in X$ let $B_x$ be the set of all vertices of $G$ at distance
at most $d$ from $x$.
Since $F$ is an optimal frame, it follows from the choice of the set
$X$ that the sets $B_x$ are pairwise disjoint and each of them is
disjoint from every ring.
\zz{%
Let $x\in X$.
If  \xx{the subgraph of $G$ induced by $B_x$} has a non-null homotopic  cycle that 
 has length at most $2d$, then such a cycle is disjoint from all of the rings,
and therefore satisfies the conclusion of the lemma.
We may therefore assume that such cycle does not exist.
}%
From Lemma~\ref{lem:expgrowth} we deduce that $|B_x|\ge 2^{d/(2c+1)}$.
Hence $|V(G)|\ge |X|2^{d/(2c+1)}$ as the sets $B_x$ are pairwise disjoint.

The boundary of the unique face of $F$ can be traced by a closed curve.
By pushing the curve slightly into the interior of the face of $F$
we obtain a closed curve $\phi$ that intersects $G$ only in vertices
and is otherwise contained in the unique face of $F$.
By applying the definition of hyperbolicity to $\phi$ we deduce that
$|V(G)|\le (c+1)(2l+R+r)$. 

Combining the upper and lower bounds on $|V(G)|$ gives
$$2^{d/(2c+1)}(l-(4d+1)s)/(2d+1)\le|X|2^{d/(2c+1)} \le |V(G)| \le 
(c+1)(2l+R+r).$$
Using the inequality observed at the beginning of this proof we obtain
$$
l-(4d+1)s\le (c+1)(2d+1)2^{-d/(2c+1)}(2l+R+r)\le (2l+R+r)/4,
$$
and hence $l\le (R+r)/2+2(4d+1)s$.
Substituting this bound into the inequality $|V(G)|\le (c+1)(2l+R+r)$ 
and using the fact that
$s\le3g+2r-2h$ by Lemma~\ref{lem:numbersegments} 
(which can be applied because no component of $\Sigma$ is the sphere)
we obtain
$$ |V(G)| \le (c+1)(2(R+r) + 4(4d+1)(3g+2r-2h)),$$
 a contradiction.
\end{proof}

\zz{%
The following corollary is an upgrade of Lemma~\ref{lem:ewlinsize}.
}

\begin{corollary}
\label{cor:ewloccyl}
Let $\F$ be a hyperbolic family of embedded graphs with rings,
let $\F$ be closed under curve cutting,
let $c$ be a Cheeger constant for $\F$, 
let $G\in\F$ have $r$ rings,
 let it be embedded in a surface $\Sigma$ of genus $g$, and let $R$ be the total number of ring vertices.
Let $d:=\lceil 3(2c+1)\log_2(8c+4)\rceil$ and $k:=\lceil2(2c+1)\log_2(g+R) + 4d\rceil$. 
If the edge-width of $G$ is at least $k+\xx{3}$, then $G$ is locally cylindrical 
and every vertex is at distance at most $(2c+1)\zz{\log_2}R$ from some ring.
\end{corollary}

\begin{proof}\xx{We begin by observing that
$$12(c+1)(4d+1)=4(c+1)3(4d+1)\le 4(c+1)15d\le 180(2c+1)^2\log_2(8c+4)+60(c+1)<$$
$$<(8c+4)^6,$$
and hence 
$$2(2c+1) \log_2(12(c+1)(4d+1))\le 12(2c+1) \log_2 (8c+4)\le4d.
$$} 
If a component of $\Sigma$ is a sphere, then it includes no vertex or edge of $G$
by Lemma~\ref{lem:onering}. By deleting such components we may assume that
no component of $\Sigma$ is a sphere.
We have $$|V(G)|\le (c+1)(2(R+r) + 4(4d+1)(3g+2r-2))\xx{\le 12(c+1)(4d+1)(g+R)},$$ for otherwise by
Theorem~\ref{thm:ewconst} the edge-width of $G$ is at most $2d<k$. 
Thus $2(2c+1)\log_2|V(G)|<k$\xx{, using the inequality observed at the beginning of the proof.} 
Lemma~\ref{lem:ewlinsize}
implies that $G$ is locally cylindrical and every vertex of $G$ has distance 
at most $(2c+1)\zz{\log_2}R$ from some ring by Corollary~\ref{cor:ringdist}, as desired.
\end{proof}

\xx{%
We can now prove Theorem~\ref{thm:ew}.
\begin{proof}[Proof of Theorem~\ref{thm:ew}]
Let $\F$ be  a hyperbolic family of embedded graphs with no rings that is closed under curve cutting.
Thus no member of $\F$ is locally cylindrical, and hence every member of $\F$ has edge-width
$O(\log(g+1))$ by Corollary~\ref{cor:ewloccyl}, as desired.
\end{proof}
}

Our motivation for the study of hyperbolic families is to understand graphs that
contain no subgraph isomorphic to a member of the family.
The following is the first of several results along those lines. 
It is a more precise version of Theorem~\ref{thm:ew}.

\begin{corollary}
\label{cor:free1}
Let $\F$ be a hyperbolic family of embedded graphs with no rings,
let $\F$ be closed under curve cutting,
let $c$ be a Cheeger constant for $\F$, and 
let $H$ be a graph with no rings embedded in a surface $\Sigma$ of genus $g$.
Let $d:=\lceil 3(2c+1)\log_2(8c+4)\rceil$ and $k:=\lceil2(2c+1)\log_2g + 4d\rceil$. 
If the edge-width of $H$ is at least $k+\xx{3}$, then $H$ has no subgraph isomorphic
to a member of $\F$. 
\end{corollary}

\begin{proof}
If $H$ has a subgraph
isomorphic to a graph $G\in\F$, then 
\zz{%
$G$ has no rings, and hence is not locally cylindrical.
By Corollary~\ref{cor:ewloccyl} the graph $G$ has edge-width at most $k+\xx{2}$,
}%
and hence so does $H$, a contradiction.
\end{proof}

\subsection{Structure of hyperbolic families}
\label{subsec:structure}

Our objective is to prove Theorem~\ref{thm:sleevedec1}, but first we need the following notion.

\begin{definition}
\label{def:cyclecut}
Let $\F$ be a family of embedded graphs with rings, let $G\in\F$ be embedded
in a surface  $\Sigma$ with rings $\R$, and  let $C$ be a non-null-homotopic cycle in $G$.
Let $\Sigma'$ be the surface obtained
by cutting $\Sigma$ open along $C$. Thus $\Sigma'$ has one or two more boundary 
components than $\Sigma$, depending on whether $C$ is one-sided or two-sided.
Let $G'$ be the graph obtained from $G$ during this operation; thus if $C$ is one-sided,
then $C$ corresponds to a cycle $C'$ in $G'$ of twice the length of $C$, and if $C$ is
$2$-sided, then it corresponds to two cycles $C_1,C_2$ in $G'$, each of the same length as $C$.
If $C$ is one-sided, then let $\R':=\R\cup\{C'\}$; otherwise let $\R':=\R\cup\{C_1,C_2\}$.
We regard $G'$ as a graph embedded in $\Sigma'$ with rings $\R'$.
If $G'\in\F$ for every
$G\in\F$ and every non-null-homotopic cycle $C$ in $G$, then we say that
$\F$ is {\em closed under cycle cutting}.
\end{definition}

It should be noted that despite the similarity in names, curve cutting and cycle cutting play
different roles: while the next lemma shows that closure under cycle cutting can be obtained
for free, Example~\ref{ex:curvecut} shows that that is not the case for curve cutting.
On the other hand, all hyperbolic families of embedded graphs of interest to us are indeed
closed under curve cutting, and hence we make that assumption whenever it is convenient
to do so.

\begin{lemma}
\label{lem:cyclecut}
Let $\F$ be a hyperbolic family of embedded graphs with rings, let $c$ be a Cheeger
constant for $\F$, and let $\F'$ be the inclusion-wise minimal family of embedded
graphs that includes $\F$ and is closed under cycle cutting.
Then $\F'$ is hyperbolic with Cheeger constant $c$.
\end{lemma}

\begin{proof}
Let $G\in\F$, let $C$ be a non-null-homotopic cycle in $G$, and let $G',\Sigma'$ and
$\R'$ be obtained as in the definition of cycle cutting.
If $\cc{\xi}:{\mathbb S}^1\to\Sigma'$ is a simple closed curve that bounds an open disk
$\Delta\subseteq\Sigma'$, then $\Delta$ can be regarded as an open disk in $\Sigma$.
It follows that $\F'$ is hyperbolic with the same Cheeger constant.
\end{proof}

 In the next two definitions we formalize the concept of narrow cylinder and a graph being decomposed 
into a bounded size graph and narrow cylinders.

\begin{definition}
Let $k,l\ge0$ and let $G$ be a graph with two rings $R_1,R_2$
embedded in a cylinder $\Sigma$.
Let $C_0,C_1,C_2,\ldots,C_{n}$ be cycles in $G$ such that
the following conditions hold:
\begin{enumerate}
\item $C_0=R_1$ and $C_n=R_2$, 
\item the cycles $C_0,C_1,\ldots,C_n$ are pairwise disjoint,
\item for $i=1,2,\ldots,n-1$, the cycle $C_i$ separates $\Sigma$ into two
cylinders, one containing $C_0,\ldots,C_{i-1}$, and the other
containing $C_{i+1},\ldots,C_n$, 
\item for $i=0,1,\ldots,n$ the cycle $C_i$ has length at most $k$, and
\item for $i=1,\ldots,n$ at most $l$ vertices of $G$ belong to the
cylinder $\Lambda\subseteq\Sigma$ with boundary components $C_{i-1}$
and $C_i$.
\end{enumerate}
In those circumstances we say that $G$ is a \emph{$(k,l)$-sleeve}.
Let us remark in the last condition the cylinder $\Lambda$ includes
the vertices of $C_{i-1}\cup C_i$.
\end{definition}

\begin{definition}
Now let $G$ be a graph with rings embedded in a surface $\Sigma$, and
let $G'$ be a $(k,l)$-sleeve with rings $R_1'$ and $R_2'$ embedded in
a cylinder $\Sigma'$.
Let $R_1,R_2$ be distinct rings of $G$ of lengths $|V(R_1')|$ and $|V(R_2')|$,
respectively. Let $H$ be the graph obtained from $G\cup G'$ by identifying
$R_1$ with $R_1'$ and $R_2$ with $R_2'$ and let $\Pi$ be the surface
obtained from $\Sigma\cup\Sigma'$ by identifying the two boundary components
that embed $R_1$ and $R_1'$ and identifying the two boundary components
that embed $R_2$ and $R_2'$.
Thus $H$ is a graph with two fewer rings than $G$ and it is embedded
in $\Pi$. We say that $H$, regarded as an embedded graph with rings,
was \emph{obtained from $G$ by adjoining the $(k,l)$-sleeve $G'$.}

We say that a graph $G$ has a \emph{sleeve-decomposition with parameters
$(k,l,m)$} if it can be obtained from an embedded graph with rings on \xx{at most}
$m$ vertices by repeatedly adjoining $(k,l)$-sleeves.
\end{definition}

The following is an easy bound on the number of sleeves in a sleeve-decomposition.

\begin{lemma}
\label{lem:numsleeves}
Let $G$ be an embedded graph with rings that has a sleeve-decomposition with
parameters $(k,l,m)$. Then the number of sleeves adjoined in the decomposition 
is at most $m/6$.
\end{lemma}

\begin{proof}
By definition of sleeve-decomposition the graph $G$ is obtained from a graph $H$ on
at most $m$ vertices by adjoining $(k,l)$-sleeves.
Since the rings of $H$ are pairwise disjoint, the graph $H$ has at most $m/3$ rings that are cycles,
and every time a sleeve is adjoined, two of those cycles stop being rings.
The lemma follows.
\end{proof}

\zz{%
We are ready to prove the main result of this section.
It states  that a graph in a hyperbolic family has a  sleeve-decomposition with parameters $(k,l,m)$,
 where $k,l,m$ depend only on the genus $g$, the total number of ring vertices $R$, and the Cheeger constant $c$. 
Moreover, the number of sleeves is also bounded.
}

\begin{theorem}
\label{thm:sleevedec1}
Let $\F$ be a hyperbolic family of embedded graphs with rings,
let $\F$ be closed under curve cutting,
let $c$ be a Cheeger constant for $\F$, and
let $G\in\F$ be embedded in a connected surface $\Sigma$ of genus $g$ with $r$ rings,
and let $R$ be the total number of ring vertices.
Let $d:=\lceil3(2c+1)\log_2(8c+4)\rceil$,
$l:=4(c+1)(10d+3)$, and $b:=312d+94$.
Then $G$ has a sleeve-decomposition with parameters $(2d, l, m)$ that uses at most $s$ sleeves, where

\begin{enumerate}
\item[{\rm(1)}] $m = (c+1)R$ and $s=0$ if $g=0$ and $r=1$, 
\item[{\rm(2)}] $m = 2(c+1)(R-4d)+2l$ and $s=1$  if $g=0$ and $r=2$,
\item[{\rm(3)}] $m=(c+1)(2(R-4d) +b(r-2)+(44d+12+2b)g)$ and $s=3(2g+r-2)$ if $g\ge 1$ or $r\ge 3$.
\end{enumerate}
\end{theorem}


\begin{proof}
By Lemma~\ref{lem:cyclecut} we may assume that $\F$ is closed under cycle cutting as well as curve cutting.
We proceed by induction on $3g+r$. If $g=0$, then $r\ge 1$ by Lemma~\ref{lem:onering},
and if $g=0$ and $r=1$, then the same lemma implies that
%
the theorem holds with $m$ as in (1). 

We now handle the case $g=0$ and $r=2$. Let $\Gamma_1$ and $\Gamma_2$ be the boundary components of $\Sigma$. 
Let ${\cal D}=\{C_1, C_2, \ldots, C_t\}$ be a maximal collection of cycles of $G$ such that for each $i$, $1\le i \le t$, $C_i$ is a cycle of length at most $2d$ and $C_i$ separates $\Sigma$ into two cylinders, one containing $\Gamma_1, C_1,\ldots,C_{i-1}$, and the other containing $C_{i+1},\ldots,C_t,\Gamma_2$, and the cycles in $\cal C$ are pairwise disjoint and each is disjoint from $\Gamma_1$ and $\Gamma_2$. Let $\Lambda_1, \Lambda_2, \Lambda_3 \subseteq \Sigma$ be the cylinders with boundary components $\Gamma_1$, $C_1$; $C_1$, $C_t$; $C_t$, $\Gamma_2$, respectively. 

By Theorem~\ref{thm:ewconst} \cc{applied to $R=4d$}, it follows that if there are more than $l$ vertices between two cycles $C,C'$ of length at most $2d$ in a cylinder, then there exists a cycle of length at most $2d$ separating $C$ from $C'$ and disjoint from both. 
Hence, it follows that the subgraph of $G$ drawn in $\Lambda_2$ is a $(2d,l)$-sleeve. 
We may also assume that $|V(G)|>2(c+1)(R-4d)+2l$ as otherwise $G$ has the desired sleeve-decomposition (with no sleeves). 
Analogously by Theorem~\ref{thm:ewconst}, the subgraphs of $G$ drawn in $\Lambda_1$ and $\Lambda_3$ have at most $2(c+1)(R-4d)+2l$ vertices combined as $\D$ was chosen maximal. 
Thus $G$ has a $(2d,l,m)$ sleeve-decomposition with $m$ as in (2) and the theorem follows.
This completes the case $g=0$ and $r=2$.

We may assume then that $g\ge 1$ or $r\ge 3$ and we have to prove that the theorem holds with $m$ as in (3). 
Let $\Gamma$ be a component of $\hbox{\rm bd}(\Sigma)$.
We denote by $\Sigma+\widehat\Gamma$ the surface obtained from $\Sigma$
by gluing a disk on $\Gamma$.
We say that a cycle $C$ in $G$ is a {\em $\Gamma$-cycle} if it is non-null-homotopic
in $\Sigma$, but null-homotopic in $\Sigma+\widehat\Gamma$.

We first show that if $G$ has a non-null-homotopic cycle $C$ of length
at most $2d$ that is not a $\Gamma$-cycle for any component $\Gamma$ of
the boundary of $\Sigma$, then the theorem holds. To see that 
cut along $C$ and declare the resulting cycle(s) to be rings in a new graph $G'$
with $r'$ rings 
embedded in a surface $\Sigma'$ of genus $g'$ with $h'$ components and a total number of $R'$ ring vertices.

Suppose $h'=2$. Then $r'=r+2$,  $R'\le R + 4d$ and  $g'=g$. Let $\Sigma_1$ and $\Sigma_2$ be the two components of $\Sigma'$. 
Let $G_i$ be the subgraph of $G'$ drawn in $\Sigma_i$. 
Then $G_i$ is a graph drawn in $\Sigma_i$ with $r_i$ rings 
and a total of $R_i$ ring vertices.
For each $i\in \{1,2\}$, either $g(\Sigma_i)\ge 1$ or $r_i\ge 3$ as $C$ is not a $\Gamma$-cycle. Thus, for each $i\in \{1,2\}$, by induction $G_i$ has a sleeve-decomposition with parameters $(2d,l,m_i)$ and at most $s_i$ sleeves, where 
$$m_i=(c+1)(2(R_i-4d) +b(r_i-2)+(44d+12+2b)g(\Sigma_i))$$ and $s_i = 3(2\cc{g(\Sigma_i)}+r_i-2)$. Adding, we find that 
\begin{align*}
m_1+m_2 &= (c+1)(2(R'-8d)+b(r'-4)+(44d+12+2b)g') \\
&\le (c+1)(2(R-4d)+b(r-2)+(44d+12+2b)g)=m,
\end{align*}
 and $s_1+s_2 = 3(2g'+r'-4) = 3(2g+r-2) = s$. Thus $G$ has a sleeve-decomposition with parameters $(2d,l,m)$ with at most $s$ sleeves, as desired.

So we may assume that $h'=1$. Suppose $C$ is one-sided, then $r'=r+1$, $R'\le R+4d$ and $g'= g-1$. Further suppose that $g'\ge 1$ or $r'\ge 3$. Then it follows from induction that $G'$ has a sleeve-decomposition with parameters $(2d,l, m')$ and at most $s'$ sleeves, where
\begin{align*}
m'&=(c+1)(2(R'-4d)+b(r'-2)+(44d+12+2b)g')\\ &\le (c+1)(2(R-4d)+b(r-2)+(44d+12+2b)g)=m
\end{align*}
 and
 $$s'=3(2g'+r'-2)=3(2(g-1)+(r+1)-2) = 3(2g+r-\zz{3}) \le s.$$
Thus $G$ has a sleeve-decomposition with parameters $(2d,l,m)$ and at most $s$ sleeves, as desired. So we may assume that $g'=0$ and $r'\le 2$. Hence $s'=0$ if $r'=1$ and $s'=1$ if $r'=2$. Note then that $g=1$. If $r'=1$, $|V(G')| \le(c+1)R'$ by Lemma~\ref{lem:onering}. Yet $r=0$ and thus $R'\le 4d$. Hence $|V(G)|\le |V(G')| \le (c+1)4d$ and $G$ has a sleeve-decomposition with parameters $(2d,l,(c+1)4d)$ and $s'=0=3(2(1)+0-2)=s$ sleeves, as desired since $g\ge 1$. If $r'=2$, then $G'$ has a sleeve-decomposition with parameters $(2d,l, 2(c+1)(R'-4d)+2l)$ and $s'=1$ by (2). Yet $r=1$ and $R'\le R+4d$. Hence $G$ has a sleeve-decomposition with parameters $(2d,l,2(c+1)R+2l)$ and at most $s=3(2(1)+1-2)=3$ sleeves, as desired, since $2(c+1)R+2l\le (c+1)(2(R-4d)+b(-1)+(44d+12+2b)(1))$.

So we may assume that $C$ is two-sided; then $r'=r+2$, $R'\le R+4d$  and $g'=g-2$. Further suppose that $g'\ge 1$ or $r'\ge 3$. Then it follows from induction that $G'$ has a sleeve-decomposition with parameters $(2d,l, m')$ and at most $s'$ sleeves, where 
\begin{align*}
m'&=(c+1)(2(R'-4d)+b(r'-2)+(44d+12+2b)g')\\ &\le (c+1)(2(R-4d)+b(r-2)+(44d+12+2b)g)=m
\end{align*}
 and $$s'=3(2g'+r'-2)=3(2(g-2)+(r+2)-2) = 3(2g+r-4)\le s.$$ Thus $G$ has a sleeve-decomposition with parameters $(2d,l,m)$ and at most $s$ sleeves as desired. So we may assume that $g'=0$ and $r'\le 2$. Hence $r'=2$, $s'=1$, $r=0$, $R=0$ and $g=2$. By induction using (2), $G'$ has a sleeve-decomposition with parameters $(2d,l, 2(c+1)(R'-4d)+2l)$ with at most $s'=1$ sleeve. Yet $R'\le 4d$. Hence $G$ has a sleeve-decomposition with parameters $(2d,l,2l)$ and at most $s=3(2(2)+0-2)=6$ sleeves, as desired, since 
$$2l = (c+1)(80d+24) \le (c+1)(2(0-4d) + b(0-2) + (44d+12+2b)(2)).$$

Thus we may assume that every  non-null-homotopic cycle $C$ in $G$ of length
at most $2d$ is a $\Gamma$-cycle for some component $\Gamma$ of
the boundary of $\Sigma$.
Let $\cal C$ be a maximal collection of cycles of $G$ such that each
member of $\cal C$ is a $\Gamma$-cycle of length at most $2d$ for some component $\Gamma$ of
the boundary of $\Sigma$, and the cycles in $\cal C$ are pairwise
disjoint and each is disjoint from every ring.

Let $C\in\cal C$ and let $\Gamma$ be the boundary component of $\Sigma$ such that $C$ is a $\Gamma$-cycle.
We say that $C$ is \emph{maximal} if the component of $\Sigma-C$ that
contains $\Gamma$ also contains every $\Gamma$-cycle included in ${\cal C}-\{C\}$.
It follows that if $\cal C$ includes a $\Gamma$-cycle, then it includes
a unique maximal $\Gamma$-cycle. Let $\Sigma'$ be the surface obtained from $\Sigma$ by removing,
for every maximal $\Gamma$-cycle $C\in\cal C$, the component of
$\Sigma-C$ that includes $\Gamma$.
Then $\Sigma'$ is a surface with boundary; it has the same number of
boundary components as $\Sigma$. 
Indeed, let $\Gamma$ be a component of the boundary of $\Sigma$.
If $\cal C$ includes no $\Gamma$-cycle, then $\Gamma$ is a component of
the boundary of $\Sigma'$; 
otherwise $C$ is a component of the boundary of $\Sigma'$, where
$C$ is the maximal $\Gamma$-cycle in $\cal C$.

Let $G'$ be the subgraph of $G$ consisting of vertices and edges drawn
in $\Sigma'$. We will regard $G'$ as a graph with rings embedded in $\Sigma'$. By Theorem~\ref{thm:ewconst}, 
$$|V(G')| < (c+1)(2R'+(32d+10)r-(32d+8)+(48d+12)g).$$

Let $\Gamma_1, \Gamma_2, \ldots \Gamma_{s'}$ be all the boundary components of $\Sigma$ such that $\C$ contains a $\Gamma_i$-cycle. Note that $s'\le r$. Let $G_i$ be the graph embedded in the cylinder $\Lambda_i\subseteq\Sigma$ with boundary components $\Gamma_i$ and $C_i$ where $C_i$ is the unique maximal $\Gamma_i$-cycle. We regard $G_i$ as a graph embedded in the 
\qq{cylinder} with two rings $\Gamma_i$ and $C_i$. By induction using (2), $G_i$ has a sleeve-decomposition with parameters $(2d,l,m_i)$ with at most $1$ sleeve where $m_i \le 2(c+1)(|\Gamma_i\cap V(G)|-2d)+2l$ as $|C_i|\le 2d$.

Thus $G$ has a sleeve decomposition with parameters $(2d,l, |V(G')|+\sum_{i=1}^{s'} m_i)$ and at most $r$ sleeves. Note that 
\begin{align*}
|V(G')|+\sum_{i=1}^{s'} m_i &\le (c+1)(2R+(32d+10+2l/(c+1))r-32d-8+(48d+12)g)\\ &= (c+1)(2R + (112d+34)r-(32d+8)+(48d+12)g).
\end{align*}

First suppose $r\ge 3$. Then $|V(G')|+\sum_{i=1}^{s'} m_i$ is at most $$(c+1)(2(R-4d) + b(r-2)+(44d+12+2b)g)=m$$ as
 $b\ge (112d+34)3-32d-8+8d = 312d + 94$. Furthermore, $r\le 3r-6 \le 3(2g+r-2)=s$ and so $G$ has a sleeve decomposition with parameters $(2d,l,m)$ and at most $s$ sleeves, as desired. 

So we may assume that $r\le 2$. But then $g\ge 1$ by assumption. In this case, $|V(G')|+\sum_{i=1}^{s'} m_i\le m$ since 
\begin{align*}
&(c+1)(2R + (112d+34)r-(32d+8)+(48d+12)g)\\ &\le (c+1)(2(R-4d) + b(r-2)+(44d+12+2b)g)
\end{align*}
 as $b=312d+94$ and $g\ge 1$. Furthermore, $r\le 2g+r-2$ since $g\ge 1$ and hence $r\le 3(2g+r-2)=s$ and so $G$ has a sleeve decomposition with parameters $(2d,l,m)$ and at most $s$ sleeves, as desired.
%
\end{proof}

By  using induction when $\Sigma$ is not connected
 we obtain the following easier-to-state bound.

\begin{theorem}
\label{thm:sleevedec2}
Let $\F$ be a hyperbolic family of embedded graphs with rings,
let $\F$ be closed under curve cutting,
let $c$ be a Cheeger constant for $\F$, 
let $G\in\F$ be embedded in a surface $\Sigma$ of genus $g$ with $r$ rings and a total number of $R$ ring vertices.
Let $d:=\lceil3(2c+1)\log_2(8c+4)\rceil$ and $l:=4(c+1)(10d+3)$.
Then $G$ has a sleeve-decomposition with parameters $(2d, l, \zz{702}d(c+1)(g+R))$, and the decomposition uses
at most $3(2g+r)$ sleeves.
\end{theorem}

\zz{%
\begin{proof}
This follows \xx{from Theorem~\ref{thm:sleevedec1}} applied to the components  of $\Sigma$, using the fact that $d\ge6$.
\end{proof}
}


We deduce the following consequence in the spirit of Theorems~\ref{thm:linext5choose} and~\ref{thm:linext},
\zz{%
but we need a definition first.
}

\begin{definition}
\label{def:free}
For $i=1,2$ let $G_i$ be a graph with rings $\R_i$ embedded in a surface $\Sigma_i$.
We say that the embedded graphs with rings $G_1$ and $G_2$ are {\em isomorphic}
if there exists an isomorphism $\phi$ of $G_1$ and $G_2$ as abstract graphs and
a bijection $f:\R_1\to\R_2$ such that for every $R\in\R_1$ the restriction of $\phi$
to $R$ is an isomorphism of $R$ and $f(R)$, and $\phi$ extends to a homeomorphism
$\Sigma_1\to\Sigma_2$.

Let $G$ be a graph with rings embedded in a surface $\Sigma$. By a {\em subgraph} of $G$ we mean
every embedded graph with rings that can be obtained from $G$ by repeatedly applying these operations:
\begin{itemize}
\item deleting a vertex or edge not belonging to any ring, and
\item deleting all vertices and edges of a ring $R$,  removing $R$ from  the list of rings,
\zz{%
and attaching a disk to the boundary component of  $\Sigma$ containing $R$.
}%
\end{itemize}
Let $\F$ be a family of embedded graphs with rings, and let $G$ be an embedded graph
with rings. We say that $G$ is {\em $\F$-free} if no subgraph of $G$ is isomorphic to
a member of $\F$.
\end{definition}

\qq{%
\begin{theorem}
Let $\F$ be a hyperbolic family of embedded graphs with rings,
let $c$ be a Cheeger constant for $\F$, and let $d:=\lceil3(2c+1)\log_2(8c+4)\rceil$.
If $G$ is a graph with no rings 
embedded in a surface $\Sigma$ of Euler  genus $g$,
then there exists a set $X\subseteq V(G)$ of size at most $51d(c+1)g$ such that 
$G\backslash X$ is $\F$-free. 
\end{theorem}
\begin{proof}
We show by induction on $g$ that there exists a required set $X$ of size at most $49d(c+1)g+2dm$,
where $m$ is the maximum number of pairwise disjoint pairwise non-homotopic non-null-homotopic 
separating cycles in $G$ of length at most $2d$.  
Since $m\le g$, this implies the theorem.
If $g=0$, then $G$ is $\F$-free by  Lemma~\ref{lem:onering},
and so we may assume that $g\ge1$ and that the theorem holds for all graphs embedded in
surfaces of Euler genus at most $g-1$.
If $G$ is $\F$-free, then $X=\emptyset$ satisfies the conclusion of the theorem, and so we may assume that 
$G$ has a subgraph isomorphic to $H\in\F$.
By Theorem~\ref{thm:ewconst} either $|V(H)|\le49d(c+1)g'$, where $g'$ is the combinatorial genus of $H$,
or $H$ (and hence $G$) has
a non-null-homotopic cycle of length at most $2d$.
In the former case let $Y:=V(H)$ and in the latter case let $Y:=V(C)$.
Let $g''$ be the combinatorial genus of  $G\backslash Y$.
We apply induction to the embedded graph $G\backslash Y$.
We may do so, because in the former case  $g'\ge1$ by  Lemma~\ref{lem:onering},
and in the latter case either $g''<g$ (if $C$ is non-separating), or $C$ separates $\Sigma$
into two surfaces of genera strictly smaller than $g$.
By induction  $G\backslash Y$ has a set $Z$ of size at most   $49d(c+1)g''+2dm''$
such that $G\backslash Y\backslash Z$ is $\F$-free, where 
$m''$ is the maximum number of pairwise disjoint pairwise non-homotopic non-null-homotopic 
separating cycles in $G\backslash Y$ of length at most $2d$.
Then in the former case $g'+g''\le g$ by Lemma~\ref{lem:genusadd}, 
and hence the theorem follows.
\ppar
So we may assume that the latter case holds. If $C$ does not separate the surface, then $g''<g$,
and the result follows. Finally, if $C$ separates the surface, then $g''=g$ and $m''<m$, and, again,
the result follows.
\end{proof}
}

\subsection{From local freedom to global freedom}
\label{sec:local2global}

In this section we prove a generalization of Corollary~\ref{cor:free1}.

%
%

\begin{definition}
Let $\F$ be a hyperbolic family of embedded graphs with rings, and let $c$ be a Cheeger
constant for $\F$.
Let $H$ be an embedded graph with rings $\R$.
Let $H'$ be the subgraph of $H$ induced by vertices $v\in V(H)$ such that there is a ring
$R\in\R$ such that $v$ is at distance at most $(\xx{2}c+1)\log_2|V(R)|$ from $R$.
If $H'$ is $\F$-free, then we say that $H$ is {\em locally $\F$-free}.
Let us remark that the definition depends on the choice of the Cheeger constant, which will
be implicit whenever we will use this notion.
\end{definition}

The following is a generalization of Corollary~\ref{cor:free1} to graphs with rings.

\begin{corollary}
\label{cor:free2}
Let $\F$ be a hyperbolic family of embedded graphs with  rings,
let $\F$ be closed under curve cutting,
let $c$ be a Cheeger constant for $\F$, and 
let $H$ be a graph with rings and a total of $R$ ring vertices embedded in a surface $\Sigma$ of genus $g$.
Let $d:=\lceil 3(2c+1)\log_2(8c+4)\rceil$ and $k:=\lceil2(2c+1)\log_2(g+R) + 4d\rceil$. 
If the edge-width of $H$ is at least $k+\xx{3}$ and $H$ is locally $\F$-free, then $H$ is $\F$-free. 
\end{corollary}

\begin{proof}
If $H$ has a subgraph
isomorphic to a graph $G\in\F$, then by Corollary~\ref{cor:ewloccyl} either $G$ has edge-width 
at most $k\xx{+2}$, or $G$ is locally cylindrical.
But $G$ has edge-width at least $k+\xx{3}$, because so does $H$, and hence $G$ is locally
cylindrical.
\zz{%
Let $H'$ be as in the definition of local $\F$-freedom.
}%
By Corollary~\ref{cor:ringdist} applied to each component of $G$ (which we may,
because $\F$ is closed under curve cutting) we deduce that $\zz{H'}$ has a subgraph
isomorphic to $G$, and hence \zz{$H$} is not \zz{locally} $\F$-free, as desired.
\end{proof}



\section{Strongly Hyperbolic Families}
\label{sec:strongly}

Let $\F$ be a hyperbolic family of embedded graphs, and let $\Sigma$ and $R$ be fixed.
Are there arbitrarily large members of $\F$ that are embedded in $\Sigma$ with a total of $R$ ring vertices?
Theorem~\ref{thm:sleevedec1} implies that the answer depends on whether there exist arbitrarily large  sleeves.
That motivates the following definition.

\begin{definition}
Let $\F$ be a hyperbolic family of embedded graphs with rings, let $c$ be a Cheeger constant for $\F$, and
let $d:=\lceil3(2c+1)\log_2(8c+4)\rceil$.
We say that $\F$ is {\em strongly hyperbolic} if there exists a constant $c_2$ such that 
for every $G\in\F$ embedded in a surface $\Sigma$ with rings $\R$ and for every two disjoint cycles $C_1,C_2$ of length at
most $2d$ in $G$, if there exists a cylinder $\Lambda\subseteq\Sigma$ with boundary components $C_1$ and $C_2$,
then $\Lambda$ includes at most $c_2$ vertices of $G$.
We say that $c_2$ is a {\em strong hyperbolic constant} for $\F$.
\end{definition}

\zz{%
We can now answer the question posed at the beginning of this section.
}%

\begin{theorem}
\label{thm:stronghyp}
Let $\F$ be a strongly hyperbolic family of embedded graphs with rings,
let $\F$ be closed under curve cutting,
let $c$ be a Cheeger constant and $c_2$ a strong hyperbolic constant for $\F$, and
let $G\in\F$ be embedded in a surface $\Sigma$ of genus $g$ with a total of $R$ ring vertices. 
Let $d:=\lceil3(2c+1)\log_2(8c+4)\rceil$.
Then $G$ has at most $(\zz{702}d(c+1)+6c_2)(g+R)$ vertices.
\end{theorem}

\begin{proof}
Let $l:=4(c+1)(10d+3)$ and $m:= \zz{702}d(c+1)(g+R)$.
By Theorem~\ref{thm:sleevedec2} the graph $G$ can be obtained from a graph on at most $m$ vertices by adjoining
at most $3(2g+r)$ $(2d,l)$-sleeves. But every sleeve has at most $c_2$ vertices by the definition of strongly hyperbolic family,
and hence the theorem follows.
\end{proof}

In the definition of strong hyperbolicity we imposed a constant bound on the number of vertices in cylinders with boundary 
components of bounded length. As the next theorem shows, it follows that this implies a linear bound for cylinders with boundaries of any length.

\begin{theorem}\label{thm:cylconst}
Let $\F$ be a strongly hyperbolic family of embedded graphs with rings,
let $\F$ be closed under curve cutting,
let $c$ be a Cheeger constant and $c_2$ a strong hyperbolic constant for $\F$, and
let $G\in\F$ be a graph with rings embedded in a surface $\Sigma$. 
Let $d:=\lceil3(2c+1)\log_2(8c+4)\rceil$.
Let $C_1,C_2$ be disjoint cycles in $G$ such that there exists a cylinder $\Lambda\subseteq\Sigma$
with boundary components $C_1,C_2$, and let $R:=|V(C_1)|+|V(C_2)|$.
Then $\Lambda$ includes at most $\zz{(702d(c+1)+6c_2)R}$ vertices.
\end{theorem}

\begin{proof}
Let $H$ be the subgraph of $G$ consisting of all vertices and edges drawn in $\Lambda$, 
and let $H$ be regarded as
a graph embedded in $\Lambda$ with rings $C_1$ and $C_2$.
By Lemma~\ref{lem:cyclecut} the graph $H$ satisfies the hyperbolicity condition with Cheeger constant $c$,
\zz{%
and by the same argument it satisfies the strong  hyperbolicity condition with strong hyperbolic  constant $c_2$.
It follows that $H$
}%
has at most $\zz{(702d(c+1)+6c_2)R}$ vertices by Theorem~\ref{thm:stronghyp}, as desired.
\end{proof}

\subsection{Examples}

The following theorem is proved  in~\cite{PosPhD}.
The proof is long and its journal version will be split into multiple papers,
\cc{the first three of which are~\cite{PosThoTwotwo,PosThoLinDisk,PosThoTwoOne}.}

\begin{theorem}
\label{thm:cylcrit}
For every integer $k$ there exists an integer $K$ such that the following holds.
Let $G$ be a graph embedded in a cylinder with rings $C_1$ and $C_2$, where $|V(C_1)|+|V(C_2)|\le k$,
let $L$ be a $5$-list  assignment for $G$,
and assume that $G$ is $C_1\cup C_2$-critical with respect to $L$.
Then $G$ has at most $K$ vertices.
\end{theorem}

\begin{lemma}
\label{lem:5liststrhyp}
Let $\F$ be the family of embedded graphs with rings from Theorem~\ref{thm:5listcritishyperbolic}.
Then $\F$ is strongly hyperbolic.
\end{lemma}

\begin{proof}
The family $\F$ is hyperbolic by Theorem~\ref{thm:5listcritishyperbolic} and
strongly hyperbolic by Theorem~\ref{thm:cylcrit}.
\end{proof}

\begin{lemma}
\label{lem:g4l4strong}
Let  $\F$ be the family from  Lemma~\ref{lem:g4l4}; that is, the family
of all embedded graphs $G$ with rings $\R$ such that \cc{no triangle is} null-homotopic and
$G$ is $\R$-critical with respect to some $4$-list assignment.
Then $\F$ is strongly hyperbolic.
\end{lemma}

\begin{proof}
This follows from Lemma~\ref{lem:44}
\zz{%
in the same way as Lemma~\ref{lem:g4l4},
the only difference being that if $X$ is as in the proof of Lemma~\ref{lem:g4l4},
then the subgraph of $G$ induced by $X$  is a cycle with at most one chord.
The rest of the argument goes through without any changes.
}%
\end{proof}

\zz{%
The analog of Theorem~\ref{thm:cylcrit} for graphs of girth five and $3$-list assignments is shown 
in~\cite[\cc{Theorem~1.8}]{PostleGirth5}.
\begin{theorem}
\label{thm:g5cylcrit}
Let $G$ be a graph of girth at least five embedded in a cylinder with rings $C_1$ and $C_2$, 
let $L$ be a $3$-list  assignment for $G$,
and assume that $G$ is $C_1\cup C_2$-critical with respect to $L$.
Then $G$ has at most \cc{$177(|V(C_1)|+|V(C_2)|)$} vertices.
\end{theorem}
We deduce the following lemma.
}%

\begin{lemma}
\label{lem:g5strong}
Let $\F$ be the family 
of embedded graphs $G$ with rings $\R$ such that
every cycle in $G$ of length at most four is \zz{equal to a ring} and $G$ is
$\R$-critical with respect to some $3$-list  assignment.
Then $\F$ is strongly hyperbolic.
\end{lemma}

\zz{%
\begin{proof}
The family $\F$ is hyperbolic by Lemma~\ref{lem:g5hyper} and
strongly hyperbolic by Theorem~\ref{thm:g5cylcrit}.
\end{proof}
}%

\zz{%
Let us point out a subtlety: the family of Lemma~\ref{lem:g5strong} is a proper subfamily
of   the family of Lemma~\ref{lem:g5hyper}.
In fact, the latter family is not strongly hyperbolic:
}%

\zz{%
\begin{lemma}
Let $\F$ be the family 
from Lemma~\ref{lem:g5hyper}; that is, the family
of embedded graphs $G$ with rings $\R$ such that
every cycle in $G$ of length at most four is not null-homotopic and $G$ is
$\R$-critical with respect to some $3$-list  assignment.
Then $\F$ is not strongly hyperbolic.
\end{lemma}
}%

\zz{%
\begin{proof}
This follows from Figure~4 and the discussion following Lemma~4.6 of~\cite{DvoKraTho3}.
\end{proof}
}%

The strong hyperbolicity of the family in the next example follows from Theorem~\ref{thm:Cisgood-2}.

\begin{example}
\label{ExpHyp}
There exists $\delta>0$ such that following holds. For all $\epsilon>0$ with $\epsilon\le \delta$ and $\alpha\ge 0$,
 the family of graphs that are $(\epsilon,\alpha)$-exponentially-critical with respect to a \zz{type $345$} list assignment
  is a strongly hyperbolic family. Moreover, the Cheeger constant and the strong hyperbolic constant do not depend on $\alpha$ or $\epsilon$.
\end{example}




\xx{%
\subsection{The structure of strongly hyperbolic families}}

Let us recall that the combinatorial genus of an embedded graph was defined prior to Lemma~\ref{lem:genusadd}.
The following is our main result about strongly hyperbolic families of embedded graphs.

\begin{theorem}
\label{thm:free5}
Let $\F$ be a strongly hyperbolic family of embedded graphs with  rings,
let $\F$ be closed under curve cutting,
let $c$ be a Cheeger constant for $\F$, let $c_2$ be a strong hyperbolic constant for $\F$, and 
let $G\in\F$ be a graph with rings $\R$ and a total number of $R$ ring vertices
embedded in a surface $\Sigma$ of Euler genus $g$.
Let $d:=\lceil 3(2c+1)\log_2(8c+4)\rceil$, 
$k:=\lceil2(2c+1)\log_2g + 4d\rceil$,
$\beta:=\zz{702}d(c+1)+6c_2$,
 and for a ring $C\in\R$ let 
$\xx{l_C:=\lceil2\beta|V(C)|\rceil}$.
Then $G$ has at most $\beta(g+R)$ vertices and
\xx{either}
\begin{itemize}
\item[{\rm(a)}]
there exist distinct rings $C_1,C_2\in\R$ such that the distance in $G$ between them is
at most \xx{$l_{C_1}+l_{C_2}-1$}, or
\end{itemize}
\xx{%
for every  component $G'$ of  $G$ 
one of the following conditions holds:}
\begin{itemize}
\item [{\rm(b)}]
\xx{$G'$  contains no rings  and $G'$ has a}
non-null-homotopic cycle  in $\xx{\widehat\Sigma}$ of length at most $k\xx{+2}$, or
\item[{\rm(c)}]
\xx{%
$G'$ contains a unique  ring $C\in\R$, every vertex of $G'$ is at distance strictly less than $l_C$ from $C$, 
and $G'$ has a  non-null-homotopic cycle  in $\widehat\Sigma$  of length at most $2l_C+|V(C)|/2$, or
}%
\item [{\rm(d)}]
$|V(G')|\le\beta(g'+R')$
and $0<R'<g'$, where $g'$ is the combinatorial genus of $G'$
and $R'$ is the number of \zz{ring vertices} in $\R$ that are a subgraph of $G'$,
or
\item[{\rm(e)}]
$G'$ includes precisely one member $C$ of $\R$,
there exists a disk $\Delta\subseteq\widehat\Sigma$ that includes $G'$,
and every vertex of $G'$ is at distance at most $(2c+1)\log_2|V(C)|$ from   $C$.
\end{itemize} 
\end{theorem}

\begin{proof}
Since $\F$ is closed under curve cutting we may assume that $G$ is $2$-cell embedded in $\Sigma$.
We may also assume that $\Sigma$ is connected.
In particular, this implies that $G$ is connected.
The graph $G$ has at most $\beta(g+R)$ vertices by Theorem~\ref{thm:stronghyp}.
To prove that one of the statements (a)--(e) holds let us first assume that $R<g$.
If $\R\ne\emptyset$, then (d) holds, and so we may assume that $\R=\emptyset$.
But then $G$ cannot be locally cylindrical, and hence (b) holds by Corollary~\ref{cor:ewloccyl}.
This proves the theorem when $R<g$.

We may therefore assume that $R\ge g$, and that (a)  does not hold.
It follows that $\R\ne\emptyset$, for otherwise $0=R\ge g$, contrary to Lemma~\ref{lem:onering}.
We claim that for some $C\in\R$
no vertex of $G$ is at distance exactly $l_C$ from $C$.
To prove this claim suppose to the contrary that for every $C\in\R$ there
exists a vertex at distance exactly $l_C$ from $C$.
Let $B_C$ be the set of all vertices of $G$ at distance \xx{at most} $l_C$ 
from $\xx{C}$. 
We have
\xx{$|B_C|> l_C\ge 2\beta|V(C)|$.}
The sets $B_C$ are pairwise disjoint because (a) does not hold, and hence 
\xx{%
$$|V(G)|\ge \sum_{C\in\R} |B_C|> 2\beta R\ge \beta(g+R)\ge|V(G)|,$$}
a contradiction.
This proves our claim that for some $C\in\R$
no vertex of $G$ is at distance exactly $l_C$ from $C$.
An immediate consequence is that $G$ has exactly one ring, 
for otherwise 
$G$ is not connected because (a) does not hold.
Let $C$ be the unique ring of $G$. 
Thus every vertex of $G$ is at distance less than $l_C$ from $C$.
Since we may assume that  (c) does not hold and $G$ is $2$-cell embedded we deduce 
\xx{from Lemma~\ref{lem:disk} applied to the graph obtained from $G$ by contracting $C$
to a vertex $v$} that $g=0$.
It follows from Corollary~\ref{cor:ringdist} that every vertex of $G$ is
at distance at most  $(2c+1)\log_2|V(C)|$ from $C$.
Thus (e) holds, as desired.
\end{proof}

\yy{
We can now prove Theorem~\ref{thm:prefree5rooted}.
\begin{proof}[Proof of Theorem~\ref{thm:prefree5rooted}]
Let $\F$ be a strongly hyperbolic family  of rooted embedded graphs
that is closed under curve cutting.
We convert every rooted graph $(G,X)\in\F$ to  an embedded  graph with rings by turning
each $x\in X$ into a $1$-vertex ring with vertex-set $\{x\}$ and removing from the surface 
the interior of an arbitrary closed disk that intersects $G$ only in $x$ and the intersection
belongs to the boundary of the disk.
Let $\G$ be the family of embedded graphs with rings so obtained.
Then $\G$ is strongly hyperbolic with the same Cheeger and strong hyperbolic constants as $\F$.
Theorem~\ref{thm:free5} implies that $|V(G)|=O(g+|X|)$ and that one of the conditions (a)--(e) holds.
Condition (a) implies that some two distinct vertices of $X$ are at distance $O(1)$,
 condition (b) implies that $G$ has a non-null-homotopic cycle of length $O(\log g+1)$,
condition (c) implies that $G$ has a non-null-homotopic cycle of length $O(1)$, and
if condition (d) holds, then the graph $G'$ has a non-null-homotopic cycle,  which has  length 
at most $|V(G')|=O(g)$.
Finally, if condition (e) holds \xx{for every component of $G$}, then $V(G)=X$ by Lemma~\ref{lem:onering},
because $\F$ is closed under curve cutting.
\end{proof}
}

Theorem~\ref{thm:free5} has the following variation, where we eliminate outcome (d) at the expense
of increasing the value of $l_C$.

\begin{theorem}
\label{thm:free6}
Let $\F,c,c_2,G,\R,R,\Sigma,g,d,k,\beta$ be as in Theorem~\ref{thm:free5},
 and for a ring $C\in\R$ let \xx{$l_C:=\lceil\beta(g+|V(C)|)\rceil$.}
Then $G$ has at most $\beta(g+R)$ vertices and
\xx{either}
\begin{itemize}
\item[{\rm(a)}]
there exist distinct rings $C_1,C_2\in\R$ such that the distance in $G$ between them is
at most \xx{$l_{C_1}+l_{C_2}-1$}, or
\end{itemize} 
\xx{%
for every component $G'$ of $G$ one of the following conditions holds:}
\begin{itemize}
\item [{\rm(b)}]
$G'$ includes no rings  and
$G'$ has a non-null-homotopic cycle  in \xx{$\widehat\Sigma$} of length at most $k\xx{+2}$, or
\item[{\rm(c)}]
\xx{%
$G'$ contains a unique  ring $C\in\R$, every vertex of $G'$ is at distance strictly less than $l_C$ from $C$, 
and $G'$ has a  non-null-homotopic cycle  in $\widehat\Sigma$  of length at most $2l_C+|V(C)|/2$, or
}%
\item[{\rm(d)}]
$G'$ includes precisely one member $C$ of $\R$,
there exists a disk $\Delta\subseteq\widehat\Sigma$ that includes $G'$,
and every vertex of $G'$ is at distance at most $(2c+1)\log_2|V(C)|$ from   $C$.
\end{itemize} 
\end{theorem}

\begin{proof}
Since $\F$ is closed under curve cutting we may assume that $G$ is $2$-cell embedded in $\Sigma$.
The graph $G$ has at most $\beta(g+R)$ vertices by Theorem~\ref{thm:stronghyp}.
If $\R=\emptyset$, then Corollary~\ref{cor:ewloccyl} implies that (b) holds.
We may therefore assume that $\R\ne\emptyset$, and that (a) and (c) do not hold.
We claim that for some $C\in\R$
no vertex of $G$ is at distance exactly $l_C$ from $C$.
To prove this claim suppose to the contrary that for every $C\in\R$ there
exists a vertex  at distance exactly $l_C$ from $C$.
Let $B_C$ be the set of all vertices of $G$ at distance \xx{at most} $l_C$ 
from \xx{$C$}. 
\zz{We have}
\xx{$|B_C|> l_C\ge \beta(g+|V(C)|)$.}
The sets $B_C$ are pairwise disjoint because (a) does not hold, and hence 
\xx{$$|V(G)|\ge\sum_{C\in\R} |B_C|> \beta(g+R)\ge|V(G)|,$$}
a contradiction.
This proves our claim that for some $C\in\R$
no vertex of $G$ is at distance exactly $l_C$ from $C$.
An immediate consequence is that $G$ has exactly one ring, for otherwise 
$G$ is not connected  because (a) does not hold.
Let $C$ be the unique ring of $G$. 
Thus every vertex of $G$ is at distance less than $l_C$ from $C$,
and since (c) does not hold and $G$ is $2$-cell embedded we deduce 
\xx{from Lemma~\ref{lem:disk} applied to the graph obtained from $G$ by contracting $C$
to a vertex $v$} 
that $g=0$.
It follows from Corollary~\ref{cor:ringdist} that every vertex of $G$ is
at distance at most  $(2c+1)\log_2|V(C)|$ from $C$.
Thus (d) holds, as desired.
\end{proof}




\section{Canvases}
\label{sec:can}

In this section we prove Theorem~\ref{thm:maincol}. \xx{It is an immediate consequence of} Theorem~\ref{thm:maincan2}.\REM{%
Deleted: , and the closely related  Theorem~\ref{thm:maincanvar2}.}
Let us recall that canvases and critical canvases were defined in Definition~\ref{def:can}.

\begin{definition}
Let  $\epsilon>0$ and $\alpha\ge0$, and   let $(G,\R,\Sigma,L)$ be a canvas.
We say that $(G,\R,\Sigma,L)$ is \emph{$(\epsilon,\alpha)$-exponentially-critical}
\zz{%
if $(G,\R,\Sigma)$ is $(\epsilon,\alpha)$-exponentially-critical with respect to $L$,
as defined in Definition~\ref{def:eacrit}.
}%
\end{definition}

\begin{definition}
Let $\C$ be a family of canvases. 
\zz{%
We define crit$(\C)$  to be the family of all embedded graphs  with rings $(G',\R,\Sigma)$
such that there exists a graph $G$ such that $G'$ is a subgraph of $G$,
$(G,\R,\Sigma,L)\in\C$ and $(G',\R,\Sigma,L)$ is a critical canvas \xx{for some list assignment~$L$}.
Note that $\C$ need not be closed under taking subgraphs; hence the need for the subgraph $G'$.
}%



We say that $\C$ is {\em critically hyperbolic} if 
crit$(\C)$ is closed under curve-cutting, and is strongly hyperbolic.

Now let  $\epsilon>0$ and $\alpha\ge0$. 
\zz{%
We define crit$_{\epsilon,\alpha}(\C)$  to be the family of all embedded graphs  with rings $(G',\R,\Sigma)$
such that there exists a graph $G$ such that $G'$ is a subgraph of $G$,
$(G,\R,\Sigma,L)\in\C$ and $(G',\R,\Sigma,L)$ is an $(\epsilon,\alpha)$-exponentially-critical canvas.
}%

%
%

We say that $\C$ is {\em critically exponentially  hyperbolic} if 
there exists a real number $\epsilon>0$ such that for all real numbers $\alpha>0$
\begin{itemize}
\item crit$_{\epsilon,\alpha}(\C)$ is closed under curve-cutting, and
\item crit$_{\epsilon,\alpha}(\C)$ is strongly hyperbolic  with Cheeger constant and strong 
hyperbolic constant that do not depend on $\alpha$.
\end{itemize}
\end{definition}

\begin{definition}
\label{def:goodcan}
A family of canvases is {\em good} if it is both critically hyperbolic and critically exponentially hyperbolic.
\end{definition}

\begin{definition}
\label{def:excision}
Let $(G,\R,\Sigma)$ be an embedded graph with rings.
Let $C_1,C_2$ be disjoint cycles in $G$ such that there exists a cylinder $\Lambda\subseteq\Sigma$
with boundary components $C_1$ and $C_2$.
Let $H$ be the subgraph of $G$ consisting of all vertices and edges drawn in $\Lambda$,
\zz{%
and let $H'$ be a subgraph of $H$ that includes both $C_1$ and $C_2$ as subgraphs.
}%
We say that $\zz{H'}$ is a {\em cylindrical excision} of $G$ with boundary components $C_1$ and $C_2$.
\end{definition}

To prove that certain families of canvases are good we need a lemma.
To state the lemma we need to consider the following properties of families of canvases.

\begin{definition}
Let $\C$ be a family of canvases. We say that $\C$ {\em satisfies property} (P1) if
for every $(G,\R,\Sigma,L)\in\C$ the following holds. Let  $G'$ be a subgraph of $G$, and let
$C$ be a cycle in $G'$. If $G'$ is planar, then there exist
 at least five $L$-colorings of $C$ that extend to $L$-colorings of $G'$.

We say that $\C$ {\em satisfies property} (P2) if for every integer $d$  
there exist  integers \zz{$\gamma$ and} $\kappa$ such that the following holds.
Let $(G,\R,\Sigma,L)\in\C$.
Then for every cylindrical excision $G'$ of $G$ with boundary
cycles $C_1,C_2$ of length at most $2d$,  if the distance between $C_1$ and $C_2$ in $G'$ is at least $\kappa$,
then 
there exist disjoint subgraphs $J_1,J_2$ of $G'$ such that $C_i$ is a subgraph
of $J_i$, $|V(J_i)|\le\gamma|V(C_i)|$,
and if an $L$-coloring of $C_1\cup C_2$ extends to an  $L$-coloring of $J_1\cup J_2$,
then it extends to an $L$-coloring of $G'$.
\end{definition}

The lemma we need is the following.

\begin{lemma}\label{lem:expcritisstrhyp}
Let $\C$ be a  family of canvases that satisfies properties (P1) and (P2), let $\epsilon>0$ and $\alpha\ge0$,
and assume that  $\hbox{\rm crit}_{\epsilon,\alpha}(\C)$ is closed under  curve cutting and is hyperbolic
with Cheeger constant $c$.
Let $d:=\lceil3(2c+1)\log_2(8c+4)\rceil$ and $l:=4(c+1)(10d+3)$,
and let $\gamma,\kappa$ be as in property (P2) applied to $\C$ and the integer $d$.
Assume further that $|L(v)|\le5$ for every  $(G,\R,\Sigma,L)\in\C$ and every $v\in V(G)$.
If $\epsilon\le1/(2l\kappa)$, then $\hbox{\rm crit}_{\epsilon,\alpha}(\C)$
 is strongly hyperbolic  with strong hyperbolic constant $2l\kappa(4\gamma d+2)+2l$.
\end{lemma}

\begin{proof}
Suppose for a contradiction that there exists an  embedded graph  with rings 
$(G,\R,\Sigma)\in\hbox{\rm crit}_{\epsilon,\alpha}(\C)$,
two cycles $D_1,D_2$ in $G$ of length at most $2d$ and a cylinder $\Lambda_1 \subseteq \Sigma$ with boundary components $D_1,D_2$ such that $\Lambda_1$ includes more than \zz{$2l\kappa(4\gamma d+2)+2l$} vertices of $G$.

Let $H_1$ be the subgraph of $G$ consisting of all vertices and edges drawn in $\Lambda_1$. 
We regard $H_1$ as a graph embedded in the  cylinder with two rings $D_1$ and $D_2$.
Then $H_1$ was obtained from $G$ by cycle cutting, as defined in Definition~\ref{def:cyclecut}.
By Theorem~\ref{lem:cyclecut} the family of embedded graphs with rings obtained from
$\hbox{\rm crit}_{\epsilon,\alpha}(\C)$ by cycle cutting is hyperbolic with the same Cheeger constant.
By applying Theorem~\ref{thm:sleevedec1} to this family we deduce that \zz{$H_1$} has a 
sleeve decomposition with parameters $(2d,l,m)$ with at most one sleeve, where 
$m = 2(c+1)(|V(D_1)|+|V(D_2)|-4d)+2l \le 2l$.
Since $H_1$ has  strictly more than $2l$ vertices, it follows that the sleeve decomposition uses
exactly one sleeve. Thus $H_1$ was obtained from an embedded graph $H_2$ with four rings
by adjoining a $(2d,l)$-sleeve $H$, where $|V(H_2)|\le 2l$.
But then $G$ can be obtained from some embedded graph $H_3$ with rings by adjoining the $(2d,l)$-sleeve $H$.
Let $C_0, C_1 \ldots, C_n$ be cycles in $H$ as in the definition of a $(2d,l)$-sleeve.

Since $(G,\R,\Sigma)\in\hbox{\rm crit}_{\epsilon,\alpha}(\C)$, there exists  a list assignment $L$ 
for a graph $G^+$ such that
\zz{$G$ is a subgraph of $G^+$,}
 $(G^+,\R,\Sigma,L)\in\C$
and $(G,\R,\Sigma,L)$ is $(\epsilon,\alpha)$-exponentially-critical.
Then $H$ is a cylindrical excision of $G^+$ with boundary cycles $C_0$ and $C_n$.

Let $k:=\lfloor|V(H)|/(l\kappa)\rfloor$; since $|V(H)|\le ln$ by the definition of $(2d,l)$-sleeve,
we have $k\kappa\le n$.
For $i=0,1,\ldots, k\zz{-1}$ let $D_i:=C_{i\kappa}$,
\zz{let $D_k:=C_n$,}
 and for $i=1,2,\ldots, k$ let $J_i$  denote the subgraph of  $H$ consisting of all vertices and  edges
 drawn in the cylinder with boundary components $D_{i-1}$ and $D_i$ that is a subset of  $\Lambda_1$.
Then $J_i$ is a cylindrical excision of $G^+$ with boundary components $D_{i-1}$ and $D_i$,
and the boundary components $D_{i-1}$ and $D_i$ are at distance at least $\kappa$ in $J_i$.
By property (P2)  there exist  disjoint subgraphs $J_{i-1}^+$ and $J_i^-$ of $J_i$
such that  $D_{i-1}$ is a subgraph of $J_{i-1}^+$,  $D_i$ is a subgraph of $J_i^-$,
both $J_{i-1}^+$ and $J_i^-$ have at most $2d\gamma$ vertices,
and if an $L$-coloring of $D_{i-1}\cup D_i$ extends to an $L$-coloring of $J_{i-1}^+\cup J_i^-$, then it extends
to an $L$-coloring of $J_i$.

Let $H':=J_0^+\cup J_k^-$ and $G'=G\setminus (V(H)\setminus V(H'))$. 
Then $|V(H')|\le 4d\gamma$. Let $g$ be the Euler genus of $\Sigma$.
As $(G,\R,\Sigma,L)$ is $(\epsilon,\alpha)$-exponentially-critical, there exists an $L$-coloring $\phi$ of $\bigcup\R$ 
such that there exist 
at least $2^{\epsilon(|V(G')|-\alpha(g+R))}$ distinct $L$-colorings of $G'$ extending $\phi$, 
but there do not exist 
$2^{\epsilon (|V(G)|-\alpha(g+R))}$ distinct $L$-colorings of $G$ extending $\phi$. 
Thus there exists a set $\K$ of distinct $L$-colorings of $G\setminus (V(H)\setminus (V(C_0)\cup V(C_n)))$ 
extending $\phi$ with $|\K|\ge2^{\epsilon(|V(G')|-\alpha(g+R))} /5^{|V(H')|}$ such that every $\phi'\in \K$ extends to an $L$-coloring of $G'$ since $|L(v)|\le5$ for all $v\in V(H')$.

Let $\phi'\in \K$. 
We claim that $\phi'$ extends to at least $5^{k-1}$ distinct $L$-colorings of $G$. 
To see this, we first notice that property (P1) implies that for every $i=1,2, \dots,k-1$
 there exist at least five $L$-colorings of $D_i$ that extend to $L$-colorings of $J_i^-\cup J_i^+$.
Since $\phi'$ extends into $J_0^+\cup J_k^-$, property (P2) implies that every choice of
the  $L$-colorings of $D_i$ as above extends to an $L$-coloring of $G$ extending $\phi'$.
This proves our claim that $\phi'$ extends to at least $5^{k-1}$ distinct $L$-colorings of $G$. 

%
%
%
%
%
%
%
%

We have that 
$$k-1-\epsilon |V(H)|\ge|V(H)|(1/(l\kappa)-\epsilon)-2\ge |V(H)|/(2l\kappa)-2\ge 4\gamma d\ge|V(H')|,$$
\zz{%
where the first inequality uses  $|V(H)|\le (k+1)l\kappa$,
  the second inequality uses $\epsilon \le 1/(2l\kappa)$ and the third inequality uses 
}%
$|V(H)|\ge |V(H_1)|-|V(H_2)|\ge 2l\kappa(4\gamma d+2)$.
By the claim of the previous paragraph there exist at least $|\K|5^{k-1}$ distinct $L$-colorings of 
$G$ that extend $\phi$. But
$$|\K|5^{k-1}\ge 2^{\epsilon(|V(G')|-\alpha(g+R))} 5^{k-1-|V(H')|}\ge
2^{\epsilon(|V(G')|-\alpha(g+R))} 5^{\epsilon|V(H)|}\ge2^{\epsilon(|V(G)|-\alpha(g+R))},$$
where the first inequality
follows by the assumption on $\K$, the second inequality uses $k-1-|V (H')|\ge\epsilon|V (H)|$ from
above, and the third inequality uses $|V (G')|\ge|V (G)|-|V (H)|$,
 a contradiction.
%
\end{proof}

Let us recall that the families $\C_3,\C_4$ and $\C_5$ were defined in Definition~\ref{def:C345}.
In order to apply Lemma~\ref{lem:expcritisstrhyp} we need to show that these families satisfy properties (P1) and (P2).

\begin{lemma}
\label{lem:aux1}
The families $\C_3,\C_4$ and $\C_5$ satisfy properties (P1) and (P2).
\end{lemma}

\begin{proof}
Property (P1) follows from Theorem~\ref{thm:extend4cycle} applied to a subpath of $C$ of length one.
To prove that these families satisfy property (P2) let $\F$ be the family of embedded graphs with rings
$(G,\R,\Sigma)$ such that $(G,\R,\Sigma,L)$ is critical and belongs to  $\C_3\cup\C_4\cup\C_5$ for some
list assignment  $L$.
Then $\F$ is strongly hyperbolic by Lemma~\ref{lem:5liststrhyp}, Lemma~\ref{lem:g4l4strong}
and Lemma~\ref{lem:g5strong}. Let $c$ be a Cheeger constant and $c_2$ a strong
hyperbolic constant for $\F$. Let
 $d$ be given, let $\gamma=\zz{20}$,
and let $\kappa$ be an integer strictly larger than the bound on the number of vertices in $\Lambda$
in  Theorem~\ref{thm:cylconst}, when the latter is applied to the family $\F$ and cycles of length
at most $2d$.
Let $(G,\R,\Sigma,L)\in\xx{\C_3\cup\C_4\cup\C_5}$, let $G'$  be a cylindrical excision  of $G$ with boundary
cycles $C_1,C_2$ of length at most $2d$, let $\Lambda$ be the corresponding cylinder
 and let the distance between $C_1$ and $C_2$ in $G'$ be at least $\kappa$.

Let $G''$  be a minimal subgraph of $G'$ such that both $C_1$ and $C_2$ are subgraphs of $G''$
and every $L$-coloring that extends to an  $L$-coloring of $G''$ also extends to an $L$-coloring of $G'$.
Then $(G'',\{C_1,C_2\},\Lambda,L)\in \C_3\cup\C_4\cup\C_5 $ and it is a critical canvas.
Thus $(G'',\{C_1,C_2\},\Lambda)\in\F$, and by Theorem~\ref{thm:cylconst} the graph 
$G''$ \cc{has strictly fewer than $\kappa$ vertices, and hence} is not connected.
It follows that $G''$ has exactly two components, one containing $C_1$ and one containing $C_2$.
For $i=1,2$ let $J_i$ be the component of $G''$ containing $C_i$.
Then $G_i$ is $C_i$-critical with respect to $L$, and hence $|V(J_i)|\le\gamma|V(C_i)|$ by
Theorems~\ref{LinearCycle0}, \ref{DvoKawCycle} and~\ref{LinCycle44}.
\end{proof}

We now prove Theorem~\ref{thm:Cisgood}, which we restate.

\begin{theorem}
\label{thm:Cisgood-2}
The families $\C_3,\C_4,\C_5$ are good families of canvases.
\end{theorem}

\begin{proof}
\xx{For $i=3,4,5$ the families crit$(\C_i)$ and crit$_{\epsilon,\alpha}(\C_i)$ are clearly closed 
under curve cutting.}
The family $\C_3$ is critically  hyperbolic by Lemma~\ref{lem:g5strong},
the family $\C_4$ is critically  hyperbolic by Lemma~\ref{lem:g4l4strong}, and
the family $\C_5$ is critically  hyperbolic by Lemma~\ref{lem:5liststrhyp}.
The hyperbolicity of crit$_{\epsilon,\alpha}(\C_3)$, crit$_{\epsilon,\alpha}(\C_4)$ and
crit$_{\epsilon,\alpha}(\C_5)$ follows from Lemma~\ref{lem:g5exphyper},
 Lemma~\ref{ex:g4exphyper} and Theorem~\ref{ExpCritDisc}, respectively.
By Lemma~\ref{lem:aux1} the families $\C_3,\C_4,\C_5$ satisfy properties (P1) and (P2), and hence by Theorem~\ref{lem:expcritisstrhyp} all three families are good.
\end{proof}

In the next theorem we show that if a coloring of the rings extends  for a canvas in a critically exponentially hyperbolic
family, then it has  exponentially many extensions.

\begin{theorem}
\label{thm:col1toexp}
For every critically exponentially hyperbolic family $\C$ of canvases
there exist  $a,\epsilon>0$ such that the following holds.
Let $(G,\R,\Sigma,L)\in\C$, let $R$ be the total number of ring vertices in $\R$, and let $g$ be the
Euler genus of $\Sigma$. 
If $\phi$ is an $L$-coloring of $\bigcup\R$ such that $\phi$ extends to an $L$-coloring of $G$, 
then $\phi$ extends to at least $2^{\epsilon(|V(G)|-a(g+R))}$ distinct $L$-colorings of $G$.
\end{theorem}

\begin{proof}
Let $\epsilon>0$ be such that $\C$ satisfies the conditions in the definition of exponentially critically hyperbolic
family of canvases for every $\alpha>0$.
By Theorem~\ref{thm:stronghyp} there exists a constant $a>0$ that depends on $\C$ but not on
 $\alpha$ such that $|V(G)|\le a(g+R)$ for every 
$\alpha>0$
and every  graph $G\in\hbox{crit}_{\epsilon,\alpha}(\C)$  
with rings 
embedded in a surface of Euler genus $g$ and a total of $R$ ring vertices.

We will show that  $a$ and $\epsilon$ satisfy the theorem.
To prove that let $(G,\R,\Sigma,L)\in\C$, let $R$ and $g$ be as in the statement of the theorem,
and assume for a contradiction that there exists an $L$-coloring  of $\bigcup\R$  that 
extends to an $L$-coloring of $G$, but does not  extend
 to at least $2^{\epsilon(|V(G)|-a(g+R))}$ distinct $L$-colorings of $G$.
Let $G'$ be a minimal subgraph of $G$ such that $G'$  includes all the rings in $\R$ 
and there exists an $L$-coloring $\phi$ of $\bigcup\R$ such that $\phi$ 
extends to an $L$-coloring of $G'$, but does not  extend
 to at least $2^{\epsilon(|V(G')|-a(g+R))}$ distinct $L$-colorings of $G'$.
Then  the  canvas $(G',\R,\Sigma,L)$ is $(\epsilon,a)$-exponentially-critical.

It follows that $\aa{(G',\R,\Sigma)}\in\hbox{crit}_{\epsilon,\xx{a}}(\C)$.
As noted earlier this implies that $|V(G')|\le a(g+R)$.
 As $\phi$ extends to an $L$-coloring of $G$, $\phi$ also extends to an 
$L$-coloring of $G'$. 
But $|V(G')|-a(g+R)\le 0$, and hence  $\phi$ extends to at least 
$\xx{2^{\epsilon(|V(G')|-a(g+R))}}$\REM{Used to say $2^{\epsilon(|V(G')|-R-a(g+R))}$}
distinct $L$-colorings of $G'$, a contradiction.
\end{proof}

The following is our main result about good families of canvases.

\begin{theorem}
\label{thm:maincan}
For every critically hyperbolic family $\C$ of canvases
there exist  $\gamma,a,\epsilon>0$ such that the following holds.
Let $(G,\R,\Sigma,L)\in\C$, 
let $g$ be the Euler genus of $\Sigma$, \xx{let $c$ be a Cheeger constant of $\hbox{\rm crit}(\C)$} and
 let $R$ be the total number of  ring vertices in $\R$.
Then $G$ has a subgraph $G'$ such that $G'$ includes all the rings in $\R$,
$G'$ has at most $\gamma(g+R)$ vertices
and for every $L$-coloring $\phi$ of $\bigcup\R$
\begin{itemize}
\item
either $\phi$ does not extend to an $L$-coloring of $G'$, or 
\item $\phi$ extends to an $L$-coloring of $G$,
and if $\C$ is critically exponentially  hyperbolic, then $\phi$ extends to 
at least $2^{\epsilon (|V(G)|-a(g+R))}$ distinct $L$-colorings of $G$.
\end{itemize}
Furthermore, 
\xx{either}
\begin{itemize}
\item[{\rm(a)}]
there exist distinct rings $C_1,C_2\in\R$ such that the distance in $G'$ between them is
at most $\gamma( |V(C_1)|+ |V(C_2)|)$, or
\end{itemize}
\xx{every component  $G''$ of $G'$} satisfies  one of the following conditions: 
\begin{itemize}
\item [{\rm(b)}]
\xx{$G''$  contains no rings  and $G''$ has a}
non-null-homotopic cycle  in $\xx{\widehat\Sigma}$ of length at most $\gamma(\log g+1)$, or
\item[{\rm(c)}]
\xx{%
$G''$ contains a unique  ring $C\in\R$, every vertex of $G''$ is at distance strictly less than $\gamma|V(C)|$ from $C$, 
and $G''$ has a  non-null-homotopic cycle  in $\widehat\Sigma$  of length at most $\gamma|V(C)|$, or
}%
\item [{\rm(d)}]
$|V(G'')|\le\gamma(g'+R')$
and $0<R'<g'$, where $g'$ is the combinatorial genus of $G''$
and $R'$ is the number of \zz{ring vertices} in $\R$ that are a subgraph of $G''$,
or
\item[{\rm(e)}]
$G''$ includes precisely one member $C$ of $\R$,
there exists a disk $\Delta\subseteq\widehat\Sigma$ that includes $G''$,
and every vertex of $G''$ is at distance at most $(2c+1)\log_2|V(C)|$ from   $C$.
\end{itemize} 
\end{theorem}

\begin{proof}
%
Let $c$ be the Cheeger constant of crit$(\C)$, and
let $d,\beta>0$ be as in
Theorem~\ref{thm:free5} when the latter is applied to the family crit$(\C)$.
If $\C$ is exponentially critically hyperbolic, then let $\epsilon$ and $a$ 
be as in Theorem~\ref{thm:col1toexp}; otherwise let $\epsilon$ and $a$  be arbitrary.
Finally, let $\gamma>0$ be such that 
$\xx{\gamma\ge4\beta+5/2}$.%
\REM{Deleted: and $\gamma\ge2(2c+1)\log_2(2\beta )+2$.}
\xx{Then} $\gamma\ge4c+2$ and  $\gamma\ge4d+\xx{3}$.
We claim that $\gamma,a,\epsilon$   satisfy the conclusion of the theorem.


Let $(G,\R,\Sigma,L)\in\C$,  let $g$ be the Euler genus of $\Sigma$, and
 let $R$ be the total number  of  ring vertices in $\R$.
Let $G'$ be a minimal subgraph of $G$ such that $G'$ includes all the rings in $\R$
and every $L$-coloring of $\bigcup\R$ that extends to an $L$-coloring of $G'$ also extends
to an $L$-coloring of $G$. 
We will show that \aa{$G'$}  satisfies the conclusion of the theorem.

It follows that \xx{either $G'=\bigcup\R$; or} 
the canvas $(G',\R,\Sigma,L)$ is critical, and hence $(G',\R,\Sigma)\in \hbox{crit}(\C)$.
%
It follows from Theorem~\ref{thm:free5} \xx{applied to the family $\hbox{crit}(\C)$}
that $G'$ has at most $\beta(g+R)$ vertices 
and that it satisfies one of the conditions (a)--(e) of Theorem~\ref{thm:free5}.
\xx{Thus $G'$ satisfies one of the conditions (a)--(e) of the present theorem.}
By the definition of $G'$  every $L$-coloring of $\bigcup\R$ either does not extend
 to an $L$-coloring of $G'$ or extends to an $L$-coloring of $G$. 
In the latter case, if $\C$ is critically exponentially hyperbolic, then 
Theorem~\ref{thm:col1toexp} implies that $\phi$ extends to 
at least
$2^{\epsilon (|V(G)|-a(g+R))}$ distinct  $L$-colorings of $G$,
as desired.
\end{proof}

We can consolidate outcomes (b)-(d) of the previous theorem to obtain the following simpler version,
which implies Theorem~\ref{thm:maincol}.

\begin{theorem}
\label{thm:maincan2}
For every critically hyperbolic family $\C$ of canvases
there exist  $\gamma,a,\epsilon>0$ such that the following holds.
Let $(G,\R,\Sigma,L)\in\C$, 
let $g$ be the Euler genus of $\Sigma$, 
 let $R$ be the total number of  ring vertices in $\R$,
\xx{let $c$ be a Cheeger constant of $\hbox{\rm crit}(\C)$}
and let $M$  be the maximum number of vertices in a ring in $\R$.
Then $G$ has a subgraph $G'$ such that $G'$ includes all the rings in $\R$,
$G'$ has at most $\gamma(g+R)$ vertices
and for every $L$-coloring $\phi$ of $\bigcup\R$
\begin{itemize}
\item
either $\phi$ does not extend to an $L$-coloring of $G'$, or 
\item $\phi$ extends to an $L$-coloring of $G$,
and if $\C$ is critically exponentially  hyperbolic, then $\phi$ extends to 
at least $2^{\epsilon (|V(G)|-a(g+R))}$ distinct $L$-colorings of $G$.
\end{itemize}
Furthermore, \xx{either} 
\begin{itemize}
\item[{\rm(a)}]
there exist distinct rings $C_1,C_2\in\R$ such that the distance in $G'$ between them is
at most \xx{$\gamma( |V(C_1)|+ |V(C_2)|)$}, or
\end{itemize} 
\xx{every component $G''$ of $G'$} satisfies  one of the following conditions: 
\begin{itemize}
\item [{\rm(b)}]
the graph $G''$ has a cycle $C$ that is not null-homotopic in $\widehat\Sigma$;
 if \xx{$G''$ includes no ring vertex}, then the length of $C$ is at most $\gamma(\log g+1)$, 
and otherwise  it is at most $\gamma (g+M)$, or
\item[{\rm(c)}]
$G''$ includes precisely one member $C$ of $\R$,
there exists a disk $\Delta\subseteq\widehat\Sigma$ that includes $G''$,
and every vertex of $G''$ is at distance at most $(2c+1)\log_2|V(C)|$ from   $C$.
\end{itemize} 
\end{theorem}

By applying Theorem~\ref{thm:free6} instead of Theorem~\ref{thm:free5} we obtain the
following variation of Theorem~\ref{thm:maincan}.
We omit the almost identical proof.

\begin{theorem}
\label{thm:maincanvar}
For every critically hyperbolic family $\C$ of canvases
there exist  $\gamma,a,\epsilon>0$ such that the following holds.
Let $(G,\R,\Sigma,L)\in\C$, 
let $g$ be the Euler genus of $\Sigma$, and
 let $R$ be the total number of  ring  vertices in $\R$.
Then $G$ has a subgraph $G'$ such that $G'$ includes all the rings in $\R$,
$G'$ has at most $\gamma(g+R)$ vertices
and for every $L$-coloring $\phi$ of $\bigcup\R$
\begin{itemize}
\item
either $\phi$ does not extend to an $L$-coloring of $G'$, or 
\item $\phi$ extends to an $L$-coloring of $G$,
and if $\C$ is critically exponentially  hyperbolic, then $\phi$ extends to 
at least $2^{\epsilon (|V(G)|-a(g+R))}$ distinct $L$-colorings of $G$.
\end{itemize}
Furthermore, \xx{either} 
\begin{itemize}
\item[{\rm(a)}]
there exist distinct rings $C_1,C_2\in\R$ such that the distance in $G'$ between them is
at most $\xx{\gamma(|V(C_1)|+|V(C_2)|+g)}$, or
\end{itemize}
\xx{every component $G''$ of} $G'$ satisfies  one of the following conditions:
\begin{itemize}
\item [{\rm(b)}]
\xx{$G''$ includes no rings  and
$G''$ has a non-null-homotopic cycle  in \xx{$\widehat\Sigma$} of length at most   $\gamma(\log g+1)$}, or
\item[{\rm(c)}]
\xx{%
$G''$ contains a unique  ring $C\in\R$, every vertex of $G''$ is at distance strictly less than $\gamma(|V(C)|+g)$ from $C$, 
and $G''$ has a  non-null-homotopic cycle  in $\widehat\Sigma$  of length at most $\gamma(|V(C)|+g)$, or
}%
\item[{\rm(d)}]
\xx{%
$G''$ includes precisely one member $C$  of $\R$,
there exists a disk $\Delta\subseteq\widehat\Sigma$ that includes $G''$,
and every vertex of $G''$ is at distance at most $\gamma\log|V(C)|$ from   $C$.}
\end{itemize} 
\end{theorem}
\REM{Deleted: 
Similarly as above, we can consolidate outcomes (b) and (c) to obtain the following simpler version.
\begin{theorem}
\label{thm:maincanvar2}
\end{theorem}
}

\section*{Acknowledgment}
This paper is based on part of the doctoral dissertation~\cite{PosPhD} of the first author,
written under the guidance of the second author.

%

\begin{thebibliography}{99}

\def\JCTB{{\it J.~Combin.\ Theory Ser.\ B}}
\def\CMUC{{\it Comment. Math. Univ. Carol.}}
\def\TAMS{{\it Trans.\ Amer.\ Math.\ Soc.}}
\def\JAMS{{\it J.~Amer.\ Math.\ Soc.}}
\def\PAMS{{\it Proc. Amer. Math. Soc.}}
\def\DM{{\it Discrete Math.}}
\def\CM{{\it Contemporary Math.}}
\def\GC{{\it Graphs and Combin.}}
\def\COM{{\it Combinatorica}}
\def\JGT{{\it J.~Graph Theory}}
\def\JAlgorithms{{\it J.~Algorithms}}
\def\SIAMDM{{\it SIAM J.~Disc.\ Math.}}
\def\CPC{{\it Combinatorics, Probability and Computing}}
\def\EJC{{\it Electron.\ J.~Combin.}}

\bibitem{Aks} V.~A.~Aksionov (a.k.a.\ Aksenov),
On an extension of a $3$-coloring on planar graphs (Russian),
{\it Diskret.\ Analiz  Vyp.\ {\bf26} Grafy i Testy} (1974), 3--19.

\bibitem{Alb} M. Albertson, You can't paint yourself into a corner, \JCTB\  {\bf 73} (1998), 189--194.

\bibitem{AlbHutch1} M. Albertson and J. Hutchinson, The three excluded cases of Dirac's map-color theorem, Second International Conference on Combinatorial Mathematics (New York, 1978), pp. 7--17, Ann. New York Acad. Sci. 319, New York Acad. Sci., New York, 1979.

\bibitem{AlbHutch2} M. Albertson and J. Hutchinson, Extending colorings of locally planar graphs, \JGT\ {\bf36} (2001), 105--116. 

\bibitem{AlbHutch3} M. Albertson and J. Hutchinson, Graph color extensions: when Hadwiger's conjecture and embeddings help, 
\EJC\ {\bf9} (2002), no. 1, Research Paper 37, 10 pp. 

\bibitem{AlbHutch4} M. Albertson and J. Hutchinson, Extending precolorings of subgraphs of locally planar graphs, 
{\it European J.~Combin.} {\bf25} (2004),  863--871.  


\bibitem{4CT1} K. Appel and W. Haken,
Every planar map is four colorable, Part I: discharging,
{\it Illinois J. of Math.} {\bf 21} (1977), 429--490.

\bibitem{4CT2} K. Appel, W. Haken and J. Koch,
Every planar map is four colorable, Part II: reducibility,
{\it Illinois J. of Math.} {\bf 21} (1977), 491--567.


\yy{%
\bibitem{BirLew} G.~Birkhoff and D.~C.~Lewis,
Chromatic polynomials,
{\it\TAMS} {\bf 60} (1946), 355--451.
}

\bibitem{BolExtremal} B. Bollob\'as, Extremal graph theory, 
Academic Press, London, New
York, San Francisco, 1978.

\bibitem{Bollobas}
B. Bollob\'as,
\newblock {\em Modern Graph Theory}.
\newblock Springer-Verlag Heidelberg, New York, 1998.

\bibitem{Borodin} O.~V.~Borodin, 
On acyclic colorings of planar graphs,
\DM\ {\bf25} (1979), 211–-236. 

\bibitem{BorGleMonRas} O.~V.~Borodin, A.~N.~Glebov, M.~Montassier, A.~Raspaud,
Planar graphs without 5- and 7-cycles and without adjacent
triangles are 3-colorable,
{\it\JCTB} {\bf99} (2009), 668--673.

\bibitem{KB} N.~Chenette, L.~Postle, N.~Streib, R.~Thomas and C.~Yerger,
Five-coloring graphs in the Klein bottle,
\JCTB\ {\bf102} (2012), 1067--1098.

\bibitem{CohHebKraLiSal} V.~Cohen-Addad, M.~Hebdige, D.~Kral, Z.~Li and  E.~Salgado,
Steinberg's conjecture is false,
{\tt arXiv:1604.05108}.

\bibitem{DeanHutch} A. Dean and J. Hutchinson, List-coloring graph on surfaces with varying list sizes, 
\cc{%
{\it Electron.\ J.~Combin.} {\bf 19} (2012), Paper 50, 16pp.}

\bibitem{DeVos} M. DeVos, K. Kawarabayashi, B. Mohar, Locally planar graphs are $5$-choosable,
\JCTB\ {\bf98} (2008), 1215--1232.

\bibitem{Diestel} R. Diestel, Graph theory, vol. 173 of Graduate Texts in Mathematics, Springer, Heidelberg, fourth ed.,
2010.

\bibitem{Dirac1} G. A. Dirac, Map color theorems, {\it Canad. J.~Math.} {\bf4} (1952), 480--490.

\bibitem{Dirac2} G. A. Dirac, The coloring of maps, {\it  J.~London Math. Soc.} {\bf28} (1953), 476--480.

\bibitem{3ChooseFarApart} Z. Dvo\v r\'ak, $3$-choosability of planar graphs with ($\le4$)-cycles far apart,
\cc{%
{\it\JCTB} {\bf 104} (2014), 26--59.}

\bibitem{DvoKawChoosegirth5} Z.~Dvo\v{r}\'ak and K.~Kawarabayashi,
Choosability of planar graphs of girth 5,
{\tt arXiv:1109.2976v1}.

\bibitem{DvoKawListEmb} Z.~Dvo\v{r}\'ak and K.~Kawarabayashi,
List-coloring embedded graphs,
\cc{%
 Proceedings of the Twenty-Fourth Annual ACM-SIAM Symposium on Discrete Algorithms,
 1004–-1012, SIAM, Philadelphia, PA, 2012.}



\bibitem{3ColLinearTime}Z.~Dvo\v{r}\'ak, K.~Kr\'al' and R.~Thomas,
Three-coloring triangle-free graphs on surfaces,
{\em Proceedings of the twentieth Annual ACM-SIAM Symposium on
Discrete Algorithms (SODA)}, New York, NY (2009), 120--129.

\bibitem{DvoKraThofol} Z.~Dvo\v{r}\'ak, K.~Kr\'al' and R.~Thomas,
Testing first-order properties for subclasses of sparse graphs,
{\it J.~ACM} {\bf60} (2013), no. 5, Art. 36, 24 pp.

\zz{%
\bibitem{DvoKraTho3} Z.~Dvo\v{r}\'ak, D.~Kr\'al' and R.~Thomas,
Three-coloring triangle-free graphs on surfaces III.
Graphs of girth five,
{\tt arXiv:1402.4710}.
}

\bibitem{DvoKraThohavel}  Z.~Dvo\v{r}\'ak, K.~Kr\'al' and R.~Thomas,
Three-coloring triangle-free graphs on surfaces V.  
Coloring planar graphs with distant anomalies,
{\tt arXiv:0911.0885}.

\bibitem{DvoKraTho6} Z.~Dvo\v{r}\'ak, K.~Kr\'al' and R.~Thomas,
Three-coloring triangle-free graphs on surfaces VI. 
3-colorability of quadrangulations,
{\tt arXiv:1509.01013}.

\bibitem{Crossings} Z. Dvo\v r\'ak, B. Lidick\'y and B. Mohar, $5$-choosability of graphs with crossings far apart, 
\cc{{\it\JCTB} {\bf123} (2017), 54--96.}

\bibitem{AlbertsonsConj} Z. Dvo\v r\'ak, B. Lidick\'y, B. Mohar and L. Postle, 
5-list-coloring planar graphs with distant precolored vertices, 
\cc{{\it\JCTB} {\bf122} (2017), 311--352.}



\bibitem{Eppstein2} D. Eppstein, Subgraph isomorphism in planar graphs and related problems, {\it J.~Graph Algorithms and Applications \bf3} (1999), 1--27.

\bibitem{ErdRubTay} P.~Erd\"os, A.~L.~Rubin and H.~Taylor,
Choosability in graphs,
{\sl Congressus Numer.} {\bf 26} (1979), 125--157.


\bibitem{Franklin} P. Franklin, A Six Color Problem, {\it J.~Math. Phys.} {\bf13} (1934), 363--379.

\bibitem{Fisk} S.~Fisk, The nonexistence of colorings,
\JCTB\ {\bf 24} (1978), 247--248.


\bibitem{FisMoh} S.~Fisk and B.~Mohar, Coloring graphs without short
non-bounding cycles,
\JCTB\ {\bf60} (1994), 268--276.

\bibitem{Gallai} T. Gallai, Kritische Graphen I, II, {\it Publ. Math. Inst. Hungar. Acad. Sci. \bf8} (1963), 165--192 and 373--395.

\bibitem{GarJoh} M.~R.~Garey and D.~S.~Johnson,
Computers and intractability. A guide
to the theory of NP-completeness, W. H. Freeman, San Francisco, 1979.

\bibitem{GimbelThom} J.~Gimbel and C.~Thomassen,
Coloring graphs with fixed genus and girth,
\TAMS\ {\bf349} (1997), 4555--4564.

\bibitem{Gro} H.~Gr\"otzsch,
Ein Dreifarbensatz f\"ur dreikreisfreie Netze auf der Kugel,
{\sl Wiss. Z. Martin-Luther-Univ. Halle-Wittenberg Math.-Natur. Reihe}
{\bf 8} (1959), 109--120.

\bibitem{Havzbarvit} I.~Havel,
O zbarvitelnosti rovinn\'ych graf\r{u} t\v{r}emi barvami,
Mathematics (Geometry and Graph Theory), Univ.\ Karlova, Prague (1970), 89--91.

\bibitem{HutchOuterplanar} J. Hutchinson, On list-coloring extendable outerplanar graphs, {\it Ars Mathematica Contemporanea \bf5} (2012), 171--184.

\bibitem{JenTof} T. R. Jensen and B. Toft,
Graph Coloring Problems, J. Wiley \& Sons,
New York, Chichester, Brisbane, Toronto, Singapore 1995.


\bibitem{KB2} K.~Kawarabayashi, D.~Kr\'al', J.~Kyn\v{c}l,
and B.~Lidick\'y,
6-critical graphs on the Klein bottle,
{\it SIAM J.~Discrete Math.}  {\bf23}  (2008/09), 372--383.

\bibitem{KM} K. Kawarabayashi and  B. Mohar, List-color-critical graphs on a fixed surface, in Proceedings of the twentieth Annual ACM-SIAM Symposium on Discrete Algorithms (Philadelphia, PA, USA, 2009), SODA '09, Society for Industrial and Applied Mathematics, pp. 1156--1165.

\bibitem{KawThofrom} K.~Kawarabayashi and C.~Thomassen,
From the plane to higher surfaces,
{\it\JCTB} {\bf102} (2012), 852--868.


\bibitem{KelPos} T.~Kelly and L.~Postle, Exponentially many 4-list-colorings of triangle-free graphs
on surfaces, 
\cc{%
{\JGT} {\bf87} (2018), 230--238.}


\bibitem{MoharThom} B. Mohar and C. Thomassen, Graphs on surfaces, Johns Hopkins University Press, Baltimore, MD, 2001.

\bibitem{PosPhD} L. Postle, $5$-List-Coloring Graphs on Surfaces, Ph.D.\ Thesis, Georgia Institute of Technology, 2012.


\bibitem{PostleGirth5} L.~Postle, 3 list coloring graphs of girth at least five on surfaces, 
\cc{%
{\tt arXiv:1710.06898v1}.}


\bibitem{PosThoTwotwo} L.~Postle and R.~Thomas,
Five-list-coloring graphs on surfaces I. Two lists of size two in planar graphs,
{\it\JCTB} {\bf111} (2015), 234--241.

\bibitem{PosThoLinDisk} L.~Postle and R.~Thomas,
Five-list-coloring graphs on surfaces II. A linear bound for critical
graphs in a disk,
{\it\JCTB\ \bf119} (2016), 42--65.

\cc{%
\bibitem{PosThoTwoOne} L.~Postle and R.~Thomas,
Five-list-coloring graphs on surfaces III.
One list of size one and one list of size two,
{\it\JCTB} {\bf128} (2018), 1--16.}

\bibitem{RingelYoungs} G. Ringel and J.W.T. Youngs, Solution of the Heawood map-coloring problem, 
{\it Proc. Nat. Acad. Sci. USA \bf60} (1968), 438--445.


\bibitem{4CTRSST} N.~Robertson, D.~P.~Sanders, P.~D.~Seymour and R.~Thomas,
The four-colour theorem, {\it\JCTB} {\bf 70} (1997), 2-44.


\bibitem{Tho5colmaps} C.~Thomassen,
$5$-coloring maps on surfaces,
\JCTB {\bf 59} (1993), 89--105.

\bibitem{ThomTorus}   C. Thomassen, Five-coloring graphs on the torus,
\JCTB\ {\bf62} (1994), 11--33.

\bibitem{ThomPlanar} C.~Thomassen, Every planar graph is 5-choosable,
\JCTB\ {\bf62} (1994), 180--181.

\bibitem{Thom3Color}C.~Thomassen,
Gr\"otzsch's $3$-color theorem and its counterparts
for the torus and the projective plane,
\JCTB\ {\bf62} (1994), 268--279.

\bibitem{Thom3ListColor} C.~Thomassen,
$3$-list coloring planar graphs of girth 5,
\JCTB\ {\bf64} (1995), 101--107.

\bibitem{ThomCritical}  C. Thomassen, Color-critical graphs on a fixed surface,
\JCTB\ {\bf70} (1997), 67--100.

\bibitem{ThoShortlist} C. Thomassen,
A short list color proof of Grotzsch's theorem,
\JCTB\ {\bf88} (2003), 189--192.

\bibitem{ThoGirth5} C. Thomassen,
The chromatic number of a graph of girth 5 on a fixed surface,
\JCTB\ {\bf87} (2003), 38-71.

\bibitem{ThomExpQuestion} C. Thomassen, 
The number of k-colorings of a graph on a fixed surface, \DM\ {\bf306} (2006), 3145--3253.

\bibitem{ThomWheels} C.~Thomassen,
Exponentially many $5$-list-colorings of planar graphs,
{\it\JCTB} {\bf 97} (2007), 571--583.

\yy{%
\bibitem{Thomany}
C.~Thomassen,
Many $3$-colorings of triangle-free planar graphs,
{\it\JCTB} {\bf 97} (2007), 334--349.}


\bibitem{Voigt} M. Voigt, List colourings of planar graphs, \DM\ {\bf120} (1993), 215--219.

\bibitem{YerPhD} C. Yerger, Color-Critical Graphs on Surfaces, Ph.D.\ Thesis, Georgia Institute of Technology, (2010).

\end{thebibliography}

\baselineskip 11pt
\vfill
\noindent
This material is based upon work supported by the National Science Foundation.
Any opinions, findings, and conclusions or
recommendations expressed in this material are those of the authors and do
not necessarily reflect the views of the National Science Foundation.
\eject

\end{document}